\documentclass[a4paper, 12pt]{amsart}
\usepackage{amsmath,amsthm}
\usepackage{amsfonts}
\usepackage{mathrsfs}
\usepackage{amssymb}
\usepackage{enumitem}
\usepackage[usenames,dvipsnames]{pstricks}

\usepackage{epsfig}
\usepackage{esint}    
\usepackage{pst-grad} 
\usepackage{pst-plot} 
\usepackage{pinlabel}
\usepackage{graphicx}
\usepackage{color}

\usepackage{epstopdf}

\newtheorem{theor}{Theorem}[section]

\newtheorem{lemma}[theor]{Lemma}
\newtheorem{prop}[theor]{Proposition}

\newtheorem*{question*}{Question}

\theoremstyle{definition}
\newtheorem{defn}{Definition}

\theoremstyle{remark}

\numberwithin{equation}{section}
\numberwithin{defn}{section}

\newcommand{\R}{\mathbb{R}}        

\newcommand{\RN}{\mathbb{R}^N}
\newcommand{\Rd}{\mathbb{R}^d}
\newcommand{\vf}{\varphi}
\newcommand{\eps}{\varepsilon}

\renewcommand{\div}{\mathrm{div}\,} 
\DeclareMathOperator{\Lap}{(-\Delta)} 

\newcommand{\Ds}{\Lap^{s}}

\renewcommand{\d}{\mathrm{d}}
\newcommand{\Ls}{\mathscr{L}_s}

\newcommand{\Lsm}{\mathscr{L}_{s/2}}

\newcommand{\abs}[1]{\left| #1 \right|}

\usepackage{color}
\usepackage[normalem]{ulem} 

\definecolor{darkblue}{rgb}{0.05, .05, .65}
\definecolor{darkgreen}{rgb}{0, .6, .2}
\definecolor{darkred}{rgb}{0.8,0,0}

\begin{document}

\title{ \bf On a fractional thin film equation}

\author{Antonio Segatti and Juan Luis V\'azquez}
\date{}

\begin{abstract}{This paper deals with  a nonlinear degenerate parabolic equation of order $\alpha$ between 2 and 4 which is a kind of fractional version of the Thin Film Equation. Actually, this one corresponds to the limit value $\alpha=4$ while the Porous Medium Equation is the limit $\alpha=2$.  We prove  existence of a nonnegative weak solution for a general class of initial data, and establish its main properties. We also construct the special solutions in self-similar form which turn out to be explicit and compactly supported. As in the porous medium case, they are supposed to give the long time behaviour or the wide class of solutions.
This last result is proved to be true under some assumptions.

 Lastly, we consider nonlocal equations with the same nonlinear structure but with order from 4 to 6. For these equations we construct self-similar solutions that are positive and compactly supported, thus contributing to the higher order theory.
 }
\end{abstract}

\maketitle


\section{\bf Introduction}

In this paper we are mainly interested in the analysis of the following system of partial differential equations
\begin{equation}
\label{eq:prob_intro1}
\begin{cases}
\partial_t u - \div (m(u) \nabla p) = 0, \,\,\hbox{ in } \Rd \times (0,T)\\
p = \Ls u, \,\,\hbox{ in }\Rd \times (0,T)\\
u(x,0) = u_0(x), \,\,\,\hbox{ in } \Rd,
\end{cases}
\end{equation}
where $\Ls:=(-\Delta)^s$, $s\in (0,1)$,  is the fractional Laplacian (see, e.g., \cite{land, stein}), the dimension $d\ge 1$, and the mobility function $m$ is linear, namely $m(u) =u$.
 From a mathematical point of view, System \eqref{eq:prob_intro1} appears, at least formally, as an interpolation between the second-order nonlinear diffusion model called Porous Medium Equation (case $s=0$,
 described in the survey paper \cite{AR} and in the monograph \cite{Vaz2007}, where complete references to origins, theory and applications are given) and
 and the fourth-order Thin Film Equation (case $s=1$) for which the
 theory of existence of weak solutions in one and in higher dimensions is quite advanced. Without claiming any completeness, we refer to  \cite{bernis-fried90, BBD, BDGG, BF97, CT_thinfilm} and to the review papers \cite{BG, My}.

 As for physical applications, the system has been analysed in dimension one for $s=1/2$ and power law mobilities by Imbert-Mellet in  \cite{imbert_mellet11} on a bounded interval with Neumann boundary conditions as a model for the dynamics of cracks. The study is continued in \cite{imbert_mellet15}.
The one dimensional analysis for a general $s\in (0,1)$ has been completed in
 \cite{ranat} and \cite{ranatphd}. Selected references to the applied literature are given in those papers. We recall that Barenblatt  was quite involved in the mathematical modeling of hydraulic fractures, \cite{Bar62}.

Another mathematical motivation comes from comparison with the system studied in the papers \cite{BKM2010} in 1D and \cite{CV2011} in all dimensions, respectively. This system reads
\begin{equation}
\label{eq:prob_intro2}
\begin{cases}
\partial_t u - \div (u \nabla p) = 0, \,\,\hbox{ in } \Rd \times (0,T)\\
p = \Ls^{-1} u, \,\,\hbox{ in }\Rd \times (0,T)\\
u(x,0) = u_0(x), \,\,\,\hbox{ in } \Rd,
\end{cases}
\end{equation}
This model has been widely studied and has interesting applications \cite{BKM2010, GL1, GL2, Head}. The difference with  system \eqref{eq:prob_intro1} clearly lies in the constitutive law that relates the density $u$ with the pressure $p$, that implies that the order of differentiation is (formally) $2-2s$. Consequently, \eqref{eq:prob_intro2} can be seen as an interpolation between the porous medium equation (2nd order) and the (0-th order) superconductor model analysed by Ambrosio and Serfaty in \cite{ambro-serf}. The model is called in \cite{VazCIME} ``Porous Medium Diffusion with Nonlocal Pressure''. On the other hand, the present Model  \eqref{eq:prob_intro1} has formally order of differentiation $2+2s\in (2,4)$.

Our aim in this paper is to develop a basic theory for System (1.1).  As a first issue, we prove existence of suitably defined weak solution in the general multidimensional setting for linear mobilities (see the Section Open Problems \ref{sec:OP} for a discussion on this topic). A remarkable feature of the weak solutions we construct is positivity.
This property is proved in our general setting along the lines of papers \cite{bernis-fried90} and \cite{imbert_mellet11} and it is a nontrivial effect of the degeneracy of the mobility.
We also show that weak solutions originating from
initial conditions with finite first moment, keep their first moment finite during the whole evolution.
Based on this estimate, we will also prove (for a particular class of weak solutions, see below and Subsection \ref{ss:mo_ene}) that also the second moment, if finite at $t=0$, remains finite.

The investigation of the intermediate range is thus quite important from the mathematical point of view since both borderline cases belong to very different types of equations. We point out that uniqueness is not proved, it seems to be a difficult problem.

A second issue of our analysis concerns the existence of self-similar solutions. Our strategy has some similarity with the analysis in \cite{CVlarge} and, in general, with the analysis of the long time behavior of the porous medium equation (see \cite{Ca_To2000}). In particular, we show that self-similar solutions to \eqref{eq:prob_intro1} with the regularity provided by the existence Theorem \ref{th:existence} are related to stationary solutions of a nonlocal Fokker-Plank type equation. More precisely,
starting from a weak solution $u$ of \eqref{eq:prob_intro1}, if function $v$ is implicitly defined as
\begin{equation}
\label{eq:self_similar_change_intro}
u(x,t) = \frac{1}{(1+t)^\alpha}v\Big(\frac{x}{(1+t)^\beta}, \log (1+t)\Big)\,,
\end{equation}
with the  proper choice of $\alpha$ and $\beta$
\begin{equation}\label{eq:selfparam}
\alpha = \frac{d}{d  + 2(1+ s)}, \,\,\,\,\beta = \frac{1}{d+ 2(1+ s)}\,,
\end{equation}
then (see the details in Section \ref{sec:self_similar}) $v(y,\tau)$ is a weak solution of the following Fokker-Planck type equation
\begin{equation}
\label{eq:fokker-cahn_intro}
\begin{cases}
\partial_\tau v -\div_{y}\Big(v\big( \nabla_y w + \beta y\big)\Big)= 0,\\
w = \Ls v.
\end{cases}
\end{equation}
Among the class of stationary solutions to \eqref{eq:fokker-cahn_intro} we are interested
in those nonnegative functions for which
\begin{equation}
\label{eq:stationary_intro}
\begin{cases}
v\big( \nabla_y w + \beta y\big) = 0\,\,\,\,\hbox{ in }\Rd,\\
w= \Ls v.
\end{cases}
\end{equation}
The reason for looking at this particular class of stationary solutions is motivated by the fact that these are the stationary solutions that emerge in the long time behavior of \eqref{eq:fokker-cahn_intro} as  solutions  with zero dissipation (see Section \ref{sec:long_time}).
Recalling that we are looking for positive solutions, \eqref{eq:stationary_intro} reduces to the free boundary problem
\begin{equation}
\label{eq:stationary2_intro}
\nabla\Big( \Ls v + \frac{\beta}{2}\vert y\vert^2\Big) = 0\,\,\,\,\hbox{ on } \,\,\,\mathscr{P}:=\{ v>0\}.
\end{equation}
In principle, the geometry of the positivity set $\mathscr{P}$ can be quite complicated (see \cite{GKO} and the PhD Thesis
\cite{knu}
for the Thin Film case). In particular, $\mathscr{P}$ can be {disconnected}. However, restricting to solutions with connected support, we have a quite complete picture of the self-similar solutions to \eqref{eq:prob_intro1}. More precisely, we can show that solutions to \eqref{eq:stationary2_intro} are  indeed solutions of an obstacle problem (the obstacle being the zero level set) for the energy
\begin{equation}\label{energy}
 \mathscr{E}(v):= \frac{1}{2}\int_{\Rd}\vert \Lsm v\vert^2 \d y + \int_{\Rd}\Big(\frac{\beta}{2} \vert y\vert^2 -1 \Big)v\d y,
\end{equation}
thus showing that self-similar solutions are somehow minimal for the energy $\mathscr{E}$.
Remarkably, the self-similar solutions are radially decreasing, compactly supported and with explicit form given by  formula  \eqref{eq:self_similar_change_intro},
 with $\alpha$ and $\beta$ as in \eqref{eq:selfparam}. Moreover, the stationary profile $v$ has the form
 \begin{equation}
 \label{eq:ss_intro}
 v(y)=(C_1-C_2  |y|^2)_+^{1+s}\,,
 \end{equation}
where $C_2=C_2(d,s)>0$ (see \eqref{eq:C2value} for the exact value) while $C_1>0$ is a free constant that allows to adjust either the mass of the solution or the radius of its support.
{Showing that from the minimizers of \eqref{energy} one can obtain solutions with the
explicit form \eqref{eq:ss_intro} and with
the free constant $C_1$ requires some work.
In particular, our analysis relies on the following
steps. At first, by scaling and comparison and relying
on the results of Dyda (\cite{dyda}) we show
that the solution of the obstacle problem for \eqref{energy}
has the explicit form (the $D$ in the subscript
refers to Dyda)
\begin{equation}
 \label{eq:explicit_obstacle_intro}
 v_D(y):= \frac{1}{\lambda^{2s}\kappa_{s,d}}(1- \lambda^2 \vert y\vert^2)_{+}^{1+s},
 \end{equation}
where $\lambda:=\sqrt{\beta/(2\gamma_{s,d})}\,$ and is supported in $B_{R_D}$ where $R_D=1/\lambda$.
In a second step, by a further scaling, we finally
obtain \eqref{eq:ss_intro}
}
These questions are discussed in full detail in Subsection 4.2.1.

Note that the limit cases $s=0$ and $s=1$, are known and agree with this formula. For $s=0$ we get the well-known Barenblatt profile
\begin{equation}
v=(C_1-C_2|y|^2)_+
\end{equation}
for the porous medium case, that was found around 1950 in papers by Zeldovich-Kompanyeets \cite{ZK1950} and Barenblatt \cite{Bar52} (they deal with general power-like mobilities  $m(u)=u^k$). For $s=1$ we get the zero-angle profile
\begin{equation}
v=(C_1-C_2|y|^2)_+^2
\end{equation}
for the corresponding Thin Film equation (see \cite{SH}, \cite{BPW} and \cite{BF97}).
The similarity exponents $\alpha$ and $\beta$ also agree, being based only on dimensional considerations. It is interesting to note that these results are somewhat similar to the ones obtained in \cite{CVlarge} for the porous medium with fractional pressure, which is a quite different setting.  Remarkably,  the self-similar solutions of that problem follow formulas \eqref{eq:self_similar_change_intro} and \eqref{eq:ss_intro} with $s\in (0,1)$ replaced by $-s$, cf. \cite{BKM2010}, \cite{BIK} and \cite{VazAbel}. In this way we get a panorama of related self-similar patterns for equations of (formal) order ranging from 0 to 4.

In all cases the self-similar solutions are of the type called {\sl source-type solutions}, which means that the initial data of $u(x,t)$ is necessarily a point mass distribution, i.e., a Dirac delta. This property follows easily from  the conservation of mass due to the divergence form of the equation and the compact expanding support that shrinks to a point as $t\to 0$. Actually, all of these solutions have free boundaries of the form
$|x|=R \, t^{\beta}$. The study of the behaviour and regularity of free boundaries for solutions with general initial data is a difficult topic
(see Section \ref{sec:OP} in this paper for some discussion).

Our analysis is purely variational and uses
symmetrization comparison arguments to prove the compactness and radial symmetry of the support. Moreover, the analysis  works in any dimension of space and for any $s\in (0,1)$. We must point out that our analysis is restricted  to a linear mobility function.
The general case of power function mobility is considered, with a different analysis,
only in dimension one and for $s=1/2$ in the paper
 \cite{imbert_mellet15}.
In particular, the self-similar solution \eqref{eq:ss_intro}
 corresponds to the solution of the ``Zero Toughness Case"  for
 dimension one  in \cite{imbert_mellet15} with a linear mobility function and $s=1/2$.

A third issue of the paper is  the long-time behavior of the weak solutions to \eqref{eq:prob_intro1}. As in \cite{CVlarge} and in \cite{Ca_To2000}, this is done by working on the Fokker-Planck equation
\eqref{eq:fokker-cahn_intro}.

As we have already mentioned, if we rescale according to \eqref{eq:self_similar_change_intro} a weak solution of \eqref{eq:prob_intro1} we get a weak solution of the Fokker-Planck equation \eqref{eq:fokker-cahn_intro} that preserves mass and positivity. The aim is then to prove that the rescaled orbits converge to our selected class of self-similar solutions. This is achieved at the prize of accepting some regularity assumption that restricts the class for which we can justify the classical study of long time behaviour. We  explain the problem at the beginning of Section \ref{sec:long_time}.
Let us now say that a main ingredient in the proof of the needed energy-dissipation estimate is the following equality (see Lemma \ref{lem:fourier})
\begin{equation}
\label{eq:fourier_nice_intro}
\frac{d-2s}{2}\int_{\Rd}\vert \Lsm u\vert^2 \d x = -\int_{\Rd}p (x\cdot \nabla u) \d x.
\end{equation}
This identity furnishes the exact balance between the second moment and the fractional energy of Section \ref{sec:long_time}. At present, we are able to prove \eqref{eq:fourier_nice_intro}
for functions that satisfy a suitable decay at infinity.
As the proof will show, this is needed to ensure the finiteness of the righthand side of \eqref{eq:fourier_nice_intro}.
Therefore, we investigate the long time behaviour for weak solutions for which
the right hand side is finite (we refer to \ref{lem:fourier} for the precise assumption).
It is important to observe that weak solutions with compact support
actually satisfy \eqref{eq:fourier_nice_intro}. It is an open problem
to prove that \eqref{eq:fourier_nice_intro} holds for all weak solutions.

It is interesting to note that an analogous identity holds
also the weak solutions of the fractional porous medium equation \eqref{eq:prob_intro2} constructed in \cite{CV2011} (see Lemma \ref{lemma:key_fract_pm}).
In this case the term in the left hand side is the energy for which \eqref{eq:prob_intro2} is a Wasserstein gradient flow (see \cite{LMS}).

The first step in the long time behaviour analysis is to prove (see Theorem \ref{th:long_time1}) that
for large times the weak solutions to the Fokker-Planck equation approach the stationary solutions, up to the extraction of a subsequence.
It is interesting to observe that the above energy $\mathscr{E}$ \eqref{energy}  decreases on  weak solutions to the Fokker-Planck equation (namely, properly rescaled
weak solutions to \eqref{eq:prob_intro1}), thus suggesting that the long-time behaviour of the weak solutions of \eqref{eq:prob_intro1} can be described by the constructed self-similar solutions.  This is indeed the case, as we prove in this paper,
under a connectedness condition on the positivity set of the cluster points for large times of the weak solutions of the  nonlocal Fokker-Planck equation.

Due to our success in constructing self-similar solutions for Equation \eqref{eq:prob_intro1}, and also the interest in treating nonlinear parabolic equations of even higher order, we devote another section to discuss the existence of self-similar solutions for equations of the type
\begin{equation}
\label{eq:prob_intro3}
\begin{cases}
\partial_t u - \div (u \nabla p) = 0 \,\,\hbox{ in } \Rd \times (0,T),\\
p = \Ls(-\Delta u) \,\,\hbox{ in }\Rd \times (0,T),\\
u(x,0) = u_0(x) \,\,\,\hbox{ in } \Rd.
\end{cases}
\end{equation}
with $0<s<1$ (hence the total order of the equation goes from 4 to 6). We find explicit compactly supported and nonnegative self-similar solutions with a Barenblatt profile of the type similar to \eqref{eq:ss_intro}, that is solutions $u(x,t)$ of the self-similar form \eqref{eq:self_similar_change_intro}
with adjusted similarity exponents
\begin{equation}\label{eq:selfparam2}
\alpha = \frac{d}{d  + 2(2+ s)}, \,\,\,\,\beta = \frac{1}{d+ 2(2+ s)}\,,
\end{equation}
and profile of the Barenblatt type:
{ \begin{equation}
\label{eq:ss_ho}
v(y) = (C_1- C_2\vert y\vert^2)_{+}^{2+s},
\end{equation}
}
This holds for  all $0<s<1$, the constant $C_2=C_2(s,d)$ is fixed and $C_1>0$ is a free constant. See whole details in  Section \ref{sec.higher4} where parameter $C_2$ is explicitly computed. Its value
is consistent with corresponding  self-similar solution for the Thin Film equation in one dimension mentioned in  \cite{BPW}.

It is worthwhile commenting on the repeated appearance of the Barenblatt profiles, that looks surprising. We recall that these profiles appear in the Porous Medium equation
 $u_t=\nabla(u^{m-1}\nabla u)=\Delta (u^m/m)$, for $m>1$, in the form
 \begin{equation}\label{sss.pme}
 v=(C_1-C_2|y|^2)^{\sigma}
 \end{equation}
in all the range of exponents $0<\sigma<\infty$ since $\sigma=1/(m-1)$, and they are quite relevant at all levels of the theory,  as amply documented in \cite{Vaz2007}. As a consequence of our results in our paper, we find that they appear as relevant  self-similar solution for the nonlocal equations \eqref{eq:prob_intro1}, \eqref{eq:prob_intro2} and \eqref{eq:prob_intro3}, and they are expected to play a big role in the theory.
 As a further observation, notice that the solution profile \eqref{eq:ss_intro} coincides with the PME solution profile \eqref{sss.pme} for the precise choices $m-1=1/(1+s)$, while  \eqref{eq:ss_ho} leads to a similar identification with the PME when $m-1=1/(2+s)$
(see \cite{DMc} and \cite{McMS} where this similarity between the Porous Medium equation and the Thin Film equation is noticed and used).

\medskip

\noindent {\sc Outline of results.}
We gather preliminary material in Section \ref{sec:material}.
In Section \ref{sec:existence} we discuss the existence of a suitably defined weak solution
The very important topic of existence of self-similar solutions is settled in  Section \ref{sec:self_similar}, and the long-time convergence to a stationary solution is studied in Section \ref{sec:long_time}. We develop the higher order application in Section  \ref{sec.higher4}. A final section contains a number of 
open directions.

\section{\bf Preliminary Material}
\label{sec:material}
In this section we collect some of the material that is needed for our analysis.

First of all, we recall that the Fractional Laplacian $\Ds$ ($s\in (0,1)$
is the nonlocal operator defined, at least for functions in the Schwartz class $\mathcal{S}(\Rd)$, as
\begin{equation}
\label{de:def_Deltas}
\Ds v(x) = C(d,s)\hbox{p.v.}\int_{\Rd}\frac{v(x)-v(y)}{\vert x-y\vert^{d+2s}}\d y,
\end{equation}
where $\hbox{p.v.}$ denotes the principal value and $c(d,s)$ is a scaling constant. If we define the Fourier transform of $v$ as
\begin{equation}
\label{eq:Fourier_trans}
\mathfrak{F}v(\xi)=\hat v(\xi) :=
 (2\pi)^{-d/2}\int_{\Rd}e^{-i\xi\cdot x} v(x)\d x, \,\,\,\,\xi\in \Rd, \,\,\,v\in \mathcal{S}(\Rd),
\end{equation}
then the Fractional Laplacian can be equivalently defined
as the operator with symbol $\vert\xi\vert^{2s}$, namely
\begin{equation}
\label{eq:fractional_f}
\widehat{\Ds v}(\cdot) = |\xi|^{2s}\hat{v}(\cdot),\,\,\,\,\forall v\in \mathcal{S}(\Rd).
\end{equation}
{For a function in $\mathcal{S}(\Rd)$ and for $s\in (0,1)$
we define, component-wise, the operator $\Lsm \circ \nabla$ by
\[
(\Lsm\circ \nabla) u = \Lsm (\nabla u).
\]
Note that using the Fourier transform we have that
$\nabla$ and $\Lsm$ commute. More precisely there holds for any $j=1,\ldots, d$
\[
\widehat{\Lsm \frac{\partial u}{\partial x_j}} = c_d i \vert \xi\vert^s \xi_j \hat{u}
= \widehat{\frac{\partial }{\partial x_j}\Lsm u}.
\]
Moreover, again using the Fourier transform, we can express the seminorm in $H^{1+s}(\Rd)$ ($s\in (0,1)$) using this operator. Namely,
\[
\| u\|_{\dot{H}^{1+s}(\Rd)}^2 = \int_{\Rd}\vert \xi\vert^{2+ 2s}\vert \hat{u}\vert^2\hbox{d}\xi =
\int_{\Rd}\vert \Lsm(\nabla u)\vert^2\hbox{d}x
\]
}


\section{\bf Existence of a weak solution}
\label{sec:existence}

We discuss the existence of {nonnegative } weak solution to system  \eqref{eq:prob_intro1}:
\begin{equation*}
\begin{cases}
\partial_t u - \div (u \nabla p) = 0, \,\,\hbox{ in } \Rd \times (0,T),\\
p = \Ls u, \,\,\hbox{ in }\Rd \times (0,T),\\
u(x,0) = u_0(x), \,\,\,\hbox{ in } \Rd.
\end{cases}
\end{equation*}
Weak solutions are defined as follows.
\begin{defn}
\label{def:weak}{\sl
Given $u_0 \in L^1_{loc}(\Rd)$ and nonnegative,
we say that $u$ is a weak solution of
 \eqref{eq:prob_intro1} if
 \begin{enumerate}
 \item $u\ge 0$ a.e. on $\Rd\times (0,+\infty),$
 \item $u\in L^\infty(0,+\infty; H^{s}(\Rd))\cap L^2(0,+\infty;H^{1+s}(\Rd)),$
 \item $p=\Ls u \in L^2(0,+\infty;H^{1-s}(\Rd)),$
 \item The following relation holds for any test
function $\varphi\in C^{\infty}_{c}(\Rd\times [0,+\infty))$
\begin{align}
\label{eq:weak_formulation}
-\iint_{Q} u\partial_t \vf \d x \d t
-\int_{Q} p u \Delta \vf \,\d x \d t
-\int_{Q} p\nabla  u \cdot\nabla \vf \,\d x \d t =  \int_{\Rd} u_0 \vf(x,0) \d x,\nonumber\\
 p = \Ls u\,\,\,\hbox{ a.e. in }  Q=\Rd\times(0,+\infty). \normalcolor \quad \qquad \qquad
\end{align}
\end{enumerate}
}
\end{defn}

\medskip
Here is the Existence Theorem.
\begin{theor}
\label{th:existence}
Given $u_0:\mathbb{R}^d\to \mathbb{R}$ { measurable and nonnegative} such that $u_0\in H^s(\Rd)$ and such that
\begin{equation}
\label{eq:initial_entropy}
\mathscr{F}(u_0):=\int_{\Rd}u_0 \log u_0 \,\d x < + \infty,
\end{equation}
there exists a weak solution $u$ according to the Definition \ref{def:weak} that moreover satisfies
\begin{enumerate}
 \item {\sl  Mass Conservation}
\begin{equation}
\label{eq:mass_cons}
\int_{\Rd} u(x,t) \,\d x = \int_{\Rd}u_0(x) \d x\,\hbox{ for a.a. } t\in (0,+\infty),
\end{equation}

\item {\sl  Entropy Estimate}
\begin{equation}
\label{eq:entropy_est}
\mathscr{F}(u(t)) + \int_{0}^{t}\int_{\Rd}\vert \Lsm(\nabla u)\vert^2 \d x\d r
\le \mathscr{F}(u_0)\,\,\hbox{ for a.a. }  t \in (0,+\infty),
\end{equation}
\item {\sl  Energy Estimate}
\begin{equation}
\label{eq:energy_est}
\frac{1}{2}\int_{\Rd}\vert \Lsm u(t)\vert^2 \d x +
\int_{0}^t\int_{\Rd}\xi^2 \d x \d r
 \le
\frac{1}{2}\int_{\Rd}\vert \Lsm u_0\vert^2 \d x\,\,\hbox{ for a.a. } t \in (0,+\infty),
\end{equation}
where the vector field ${\bf\xi}\in L^2(0,+\infty;L^2(\Rd))$
 satisfies
\begin{equation}
\label{eq:def_xi}
  \nabla(u p) - p\nabla u = u^{1/2}\xi \,\,\,\,\,\hbox{ almost everywhere in }\Rd\times (0,+\infty).
\end{equation}
\end{enumerate}
\end{theor}

\noindent {\bf Important functional remark.}
The vector field $\xi$ emerges as a weak $L^2$ limit of a sequence, in the approximation scheme we introduce
for the proof of the Existence Theorem, and it is related to the product $u^{1/2}\nabla p$. In particular, due to the nonlinear and degenerate structure
of the system \eqref{eq:prob_intro1} we are not able to rigorously identify $\xi =u^{1/2}\nabla p$, as the formal estimates suggest.
However, thanks to the characterization \eqref{eq:def_xi} we can conclude that
in the regions of $\Rd\times (0,+\infty)$ in which $u>0$ we have \ $\xi = u^{-1/2}(\nabla(u p) - p\nabla u)$
and, if $p$ were regular enough to give a pointwise meaning to $\nabla p$,
 we would have (still in the regions where $u$ is nonzero) the plain expression \ $\xi = u^{1/2}\nabla p$.
 On the contrary, in the regions where $u\equiv 0$, \eqref{eq:def_xi} gives no detailed information on $\xi$, we know that  $\xi\in L^2$ in space and time.

\subsection{Approximate problem and main estimates}
\label{ss:approx_problem}
We approximate Equation \eqref{eq:prob_intro1} following \cite{bernis-fried90} and \cite{imbert_mellet11}.
For any $\eps>0$ and $M>0$ we consider the (strictly positive and bounded) mobility function $m_\eps :\R\to (0, \infty)$
defined by
\begin{equation}
\label{eq:mobility_approx}
m_\eps^M(y) := \min\left\{M,m_\eps(y)\right\},
\end{equation}

where $m_\eps(y):=y^+ +\eps$ and $y^+ = \max\left\{y, 0\right\}$ for $y\in \R$.

We then consider the following nondegenerate approximate problem
\begin{equation}
\label{eq:approx_prob_strong_M}
\begin{cases}
\partial_t u - \div (m_\eps^M(u) \nabla p) = 0 \,\,\hbox{ in } \Rd \times (0,T),\\
p = \Ls u \,\,\hbox{ in }\Rd \times (0,T),\\
u(\cdot,0) = u_0(\cdot) \,\,\,\hbox{ in } \Rd.
\end{cases}
\end{equation}
To be precise the problem above should be intended in the distributional sense on $\Rd\times (0,T)$.
Note however that since $0<\eps\le m_\eps^M\le M$, the first equation is not degenerate and with bounded mobility and this will imply that the approximate solutions $u_\eps$
(for notational simplicity, when no confusion arises we do not
write the dependence on $M$)
 are regular enough for positive times
to justify  all the estimates we perform at the approximate level. See details below.

\subsubsection{Existence for Problem \eqref{eq:approx_prob_strong_M}}

\label{sss:existence_approx_probl}
The existence of an approximate solution follows from a nested approximation scheme.
Given a bounded domain $\Omega$, for any $s\in (0,1)$ we introduce the Hilbert space
\[
\mathcal{X}_s(\Omega):= \left\{ v\in H^s(\Rd): v\equiv 0 \,\,\,\hbox{ a.e. in }\Rd\setminus \Omega\right\},\,\,\,\,\|v\|_{X_s(\Omega)}^2
:=\int_{\Rd}\vert\Lsm v\vert^2\d x.
\]

\noindent \sc Step 1\rm. We let $\tau>0$ and $R>0$ and we consider the following stationary problem:

\noindent {(SP): \sl Given $v\in \mathcal{X}_s(B_R(0))$ to find $u\in \mathcal{X}_s(B_R(0))$}
\begin{equation}
\label{eq:approx_approx_probl}
\begin{cases}
u  = v + \tau\div (m_\eps^{M}(u) \nabla p) \,\,\,\,\hbox{ in } B_R(0),\\
p = \Ls u\,\,\,\hbox{ in } B_R(0).
\end{cases}
\end{equation}

This problem is related (see below)
to the implicit Euler scheme for the
evolution
\[
\begin{cases}
\partial_t u=\div (m_\eps^{M}(u) \nabla p)
 \,\,\,\,\hbox{ in } B_R(0)\times (0,+\infty),\\
p = \Ls u\,\,\,\hbox{ in } B_R(0)\times (0,+\infty),\\
u\equiv 0,\,\,\,\hbox{ in } \Rd\setminus B_R(0).
\end{cases}
\]
Now we discuss, using the Leray-Schauder fixed point Theorem,
the existence of a solution of (SP). 
To this end, we let $\sigma\in [0,1]$ and we implement the following scheme
\begin{enumerate}
\item Given $\bar u\in \mathcal{X}_s(B_R(0))$, we let $p\in H^1(B_R(0))$
the weak unique solution
of
\begin{equation}
\label{eq:fixed_point1}
\begin{cases}
\tau\div(m_\eps^{M}(\bar u)\nabla p) = \bar u- v,\\
p = 0 \qquad \mbox{on}  \ \partial B_R(0).
\end{cases}
\end{equation}
\item Given $p$ from step $1$, we let $u\in \mathcal{X}_s(B_R(0))$ the unique solution of
\begin{equation}
\label{eq:fixed_point2}
\begin{cases}
\Ls u = \sigma p \,\,\,\,\,{ \hbox{ in } B_R(0)}\\
u = 0\,\,\,\,{ \hbox{ in }\mathbb{R}^d\setminus B_R(0)}.
\end{cases}
\end{equation}
\end{enumerate}
Therefore, the procedure above produces
a map $A:\mathcal{X}_s(B_R(0))\times [0,1]\to \mathcal{X}_s(B_R(0))$ such that
\[
A:(\bar u, \sigma)\mapsto u,
\]
where $(u,p)\in \mathcal{X}_s(B_R(0))\times H^1(B_R(0)) $
is the unique solution of  \eqref{eq:fixed_point1}-\eqref{eq:fixed_point2}.
We can check that the map $A$ has the following properties

\noindent (1) $A(u,0)=0$ for any $u\in \mathcal{X}_s(B_R(0))$.

\noindent (2) $A$ is compact.

\noindent (3) There exists $M>0$ such that
\begin{equation}
\label{eq:leray_schauder}
\| u\|_{\mathcal{X}_s(B_R(0))} \le M, \,\,\,\,\forall (u,\sigma) \,\,\,\hbox{ satisfying } u = A(u,\sigma).
\end{equation}

Then, the Leray-Schauder Fixed Point Theorem (see \cite[Theorem 11.6]{gilb_trud}) gives the
existence of a fixed point for the map
\[
A_1 u = A(u,1),\,\,\,\,\hbox{ for } u \in \mathcal{X}_s(B_R(0)),
\]
namely a solution of \eqref{eq:approx_approx_probl}.

The first two properties listed above are evident. In particular, the second comes from fractional elliptic regularity.
Concerning this last point, note that from \eqref{eq:fixed_point2} and the fact that $p\in H^1(B_R(0))\subset L^2(B_R(0))$ we
 conclude that, at least, $u\in H^{3s/2-\delta}(\Rd)$ for any $\delta>0$
(see \cite[Corollary 1.1]{ASSS}). In particular, this last space is compactly embedded in $\mathcal{X}_s(B_R(0))$.  More regularity can be obtained by bootstrapping but we will not use it.

We still have to verify the boundedness property \eqref{eq:leray_schauder}. To this end, let $u\in \mathcal{X}_s(B_R(0))$ such that $u = A(u,\sigma)$. Recall that $0< \eps\le m_\eps^M\le M$ and thus $m_\eps^M \nabla p\in H^1(B_R(0))$.
Then we have
\[
\begin{cases}
\int_{B_R(0)}(u-v) \phi\d x + \int_{B_R(0)}m_\eps^{M}(u)\nabla p\cdot \nabla \phi\, \d x = 0,\,\,\,\,\forall \phi\in H^1_{0}(B_R(0))\\[4pt]
\int_{\Rd}\Lsm u\, \Lsm\psi\d x = \sigma \int_{\Rd}p \psi\,\d x,\,\,\,\,\forall \psi \in X_s(B_R(0))
\end{cases}
\]

If $\sigma=0$ there is nothing to check, hence we can assume $\sigma>0$.
We take $\phi =p$ in the first equation (we still denote with $p$ the truncation to $0$ of $p$
outside $B_R(0)$), and $\psi = (u-v)/\sigma$ in the second equation.
All this is justified in the above mentioned regularity framework.
We thus obtain
\begin{align*}
\frac{1}{\sigma}\Big(\int_{\Rd}\vert \Lsm u\vert^2\d x - \int_{\Rd}\Lsm u \,\Lsm v\d x \Big)& = \int_{\Rd}(u-v)p\d x\nonumber\\
& = -
\tau\int_{B_R(0)}m_\eps^{M}(u)\vert\nabla p\vert^2 \d x\le 0,
\end{align*}
 that easily implies \eqref{eq:leray_schauder}.

\noindent \sc Step 2\rm. Next, we tackle the evolution process. { Given $u_0\in H^s(\Rd)$, we consider a smooth function that is supported in $B_R(0)$ and such that $u^0(R)\xrightarrow{R\to+\infty}u_0$ in $H^s(\Rd)$}.
 Then, we
introduce the uniform partition $\mathcal{P}$ of $(0,+\infty)$, i.e.,
\[
\mathcal{P}:=\left\{0=t_0<t_1<\ldots<t_k<\ldots \right\}, \,\,\,\,\tau:=t_i-t_{i-1}, \,\,\,\,\lim_{k\to +\infty}t_k = +\infty.
\]
Then, we
iteratively solve \eqref{eq:approx_approx_probl} with $v=u^0, u^1, \ldots, u^{k-1}, \ldots$, where $u^k$ is a solution
of \eqref{eq:approx_approx_probl} with $v=u_{k-1}$. In a  standard way we introduce the piecewise-linear ($\hat u_k$) and
the piecewise-constant $(\bar u_k)$ interpolants of the discrete values $u_k$. We set
\begin{align*}
\hat u_k(0)&:= u^0_R,\,\,\,\,\,\hat{u}_k(t):= \alpha_k(t) u_k + (1-\alpha_k(t))u^{k-1},\\
\bar{u}_k(0)&:= u^0_R,\,\,\,\,\,\bar{u}_k(t):= u_k\,\,\,\,\hbox{for} \,\,\,\,t\in ((k-1)\tau,k\tau], \,\,\,k\ge 1,
\end{align*}
 where $\alpha_k(t):=(t-(k-1)\tau)/\tau$ for $t\in ((k-1)\tau,k\tau]$ and $k\ge 1$.
 The couple $(\hat{u}_k, \bar{u}_k)$ solves
 \begin{equation}
 \label{eq:time_discrete}
 \begin{cases}
 \partial_t \hat{u}_k =\div (m_\eps^{M}(\bar{u}_k) \nabla \bar{p}_k)\,\,\hbox{ in } \,\,\,B_R(0)\times (0,+\infty),\\
 \bar{p}_k = \Ls \bar{u}_k,\,\,\hbox{ in } \,\,\,B_R(0)\times (0,+\infty),\\
\bar{u}_k = 0\,\,\hbox{ in } (\Rd\setminus B_R(0))\times (0,+\infty),\,\,\,\,\bar{u}_k(0) = u^0_R\,\,\,\,\hbox{ in } B_R(0).
 \end{cases}
 \end{equation}
Now, in order to pass $\tau\to 0$ we perform some a priori estimates on $\hat{u}_k$ and $\bar u_k$.
First of all, since $\hat{u}_k \equiv 0$ in $\Rd\setminus B_R(0)$, we have that $\partial_t \hat u_k\equiv 0$
in $\Rd\setminus B_R(0)$ and thus the second equation in \eqref{eq:time_discrete} gives, for any $t\in (0,+\infty)$,
\[
\int_{B_R(0)} \partial_t \hat u_k \bar p_k \d x = \int_{\Rd} \partial_t \hat u_k \bar p_k \d x = \int_{\Rd}\partial_t \hat u_k \Ls \bar u_k \d x.
\]
Therefore, fixing $T = {\tau}N$ for some $N\in \mathbb{N}$ and integrating the above
relation on $(0,T)$ we have (recall that $2a(a-b) = a^2 + (a-b)^2 - b^2$ for any $a,b\in \mathbb{R}$)
\begin{align}
\label{eq:time_discrete1}
&\int_{0}^T\int_{\Rd}\partial_t \hat u_k \Ls \bar u_k \d x = \sum_{k=1}^N\int_{\Rd}\Lsm u_k (\Lsm u_k - \Lsm u_{k-1})\d x\nonumber\\
&= \frac{1}{2}\|\Lsm \bar u_k(T)\|^2_{L^2(\Rd)} + \sum_{k=1}^N\|\Lsm u_k - \Lsm u_{k-1}\|_{L^2(\Rd)}^2 - \frac{1}{2}\|\Lsm u^0_R\|^2_{L^2(\Rd)}
\end{align}
Moreover, the first equation in \eqref{eq:time_discrete} gives
\[
\int_{0}^T\int_{B_R(0)} \partial_t \hat u_k \bar p_k \d x \d r= -\int_{0}^T\int_{B_R(0)}m_\eps^{M}(\bar u_k) \vert\nabla \bar p_k\vert^2 \d x\d r
\]
and thus we have the estimate on the discrete solution
\begin{equation}
\label{eq:discrete_est}
 \frac{1}{2}\|\Lsm \bar u_k(T)\|^2_{L^2(\Rd)} +   \int_{0}^T\int_{B_R(0)} m_\eps^{M}(\bar u_k)\vert\nabla \bar p_k\vert^2 \d x\d r
 \le \frac{1}{2}\|\Lsm u^0_R\|^2_{L^2(\Rd)}.
\end{equation}
This estimate is the core of the existence theory for \eqref{eq:approx_prob_strong_M} and produces one of two estimates available
for \eqref{eq:prob_intro1}.
Note that, for any fixed $\eps>0$,
\[
\int_{0}^T\int_{B_R(0)}m_\eps^{M}(\bar u_k) \vert\nabla \bar p_k\vert^2 \d x\d r \ge \eps\int_{0}^T\int_{B_R(0)} \vert\nabla \bar p_k\vert^2 \d x\d r,
\]
and thus a comparison in the first equation gives that the time derivative $\partial_t \hat u^k$ is bounded in
$L^2(0,T;W^{-1,q}(B_R(0)))$, uniformly in $M$ and in $R$, for some $q>1$.
In particular, since $\eps>0$ is kept fixed, the bounds above are sufficient to pass to the limit with respect to $\tau$ via standard compactness arguments and find in the limit a solution $u_R$ of the following problem:
\begin{equation}
\label{eq:approx_domain}
\begin{cases}
\partial_t u = \div(m_\eps^M(u)\nabla p)\,\,\,\,\hbox{ in }B_R(0) \times (0,T),\\
p = \Ls u \,\,\hbox{ in }B_R(0) \times (0,T),\\
u = 0\,\,\hbox{ in } (\Rd\setminus B_R(0))\times (0,+\infty),\,\,\,\,\,u(\cdot,0) = u^0_R(\cdot) \,\,\,\hbox{ in } B_R(0).
\end{cases}
\end{equation}
Note that for $u_R$ we have,
\begin{equation}
 \frac{1}{2}\|\Lsm  u_R(T)\|^2_{L^2(\Rd)} +   \int_{0}^T\int_{B_R(0)} m_\eps^M(u_R)\vert\nabla  p_R\vert^2 \d x\d r
 \le \frac{1}{2}\|\Lsm u^0_R\|^2_{L^2(\Rd)}.
\end{equation}

\noindent \sc Step 3\rm.
Now, since $u^0_R \xrightarrow{R\to +\infty}u^0$ in $H^s(\Rd)$, the estimate above is uniform w.r.t. to $R$.
Then, again as before we can easily pass to the limit in \eqref{eq:approx_domain} and obtain a solution $u_{\eps,M}$
of \eqref{eq:approx_prob_strong_M} with, at least, the energy regularity $(u_{\eps,M},p_{\eps,M})\in L^\infty(0,+\infty;H^s(\Rd))\cap L^2(0,+\infty;H^1(\Rd))$, namely it satisfies
\begin{align}
\label{eq:energy_M}
\frac{1}{2}\|\Lsm  u_{\eps,M}(T)\|^2_{L^2(\Rd)} +   \int_{0}^T\int_{\Rd} m_\eps^M(u_{\eps,M})\vert\nabla  p_{\eps,M}\vert^2 \d x\d r
 \le \frac{1}{2}\|\Lsm u^0\|^2_{L^2(\Rd)}
\end{align}
and the companion estimate
\begin{equation}
\label{eq:uH1}
\| \nabla p_{\eps,M}\|_{L^2(0,+\infty;L^2(\Rd))}\le \frac{1}{2\eps}
\|\Lsm u^0\|^2_{L^2(\Rd)}.
\end{equation}
The above estimate
 and a comparison in \eqref{eq:approx_prob_strong_M} guarantees that
$\partial_t u_{\eps,M}$ is bounded, uniformly in $M$, in
$L^2(0,T;W^{-1,q}(\Rd))$ for some $q>1$.
Note that for any fixed $\eps$ and $M$ we also have that
\begin{equation}
\label{eq:uthm1}
\partial_t u_{\eps,M}\in L^2(0,+\infty;H^{-1}(\Rd)).
\end{equation}
The membership of $\partial_t u_{\eps,M}$ to this space of distributions is clearly not uniform with respect to $\eps$ and $M$ as
it heavily depends on the boundedness and nondegenerate character
of $m_\eps^M$.

At the end, we can let $M\to +\infty$ and obtain a (weak) solution $u_\eps$ of
\begin{equation}
\label{eq:approx_prob_strong}
\begin{cases}
\partial_t u - \div (m_\eps(u) \nabla p) = 0 \,\,\hbox{ in } \Rd \times (0,T),\\
p = \Ls u \,\,\hbox{ in }\Rd \times (0,T),\\
u(\cdot,0) = u_0(\cdot) \,\,\,\hbox{ in } \Rd,
\end{cases}
\end{equation}
where we recall $m_\eps (y) :=y^+ + \eps$.

\subsection{Uniform estimates with respect to $\eps$: Energy and Entropy Estimate}
\label{ss:estimates}
In this Subsection we derive the two basic estimates, uniform on the approximate parameter $\eps$, on the solution $u_\eps$ of the approximate
problem \eqref{eq:approx_prob_strong}, namely the Energy Estimate and the Entropy Estimate
which correspond to the estimate \eqref{eq:energy_est} and \eqref{eq:entropy_est} in the limit $\eps\to 0$, respectively.

\noindent{\bf Energy Estimate }
The Energy Estimate follows by semicontinuity from the analogous estimate \eqref{eq:energy_M}. Note that
$m_\eps^{M}(u_{\eps,M})\xrightarrow{M\to +\infty} m_\eps(u_\eps)$ strongly in $L^2(0,T;L^2(\Rd))$.
We have, for almost any $t\le T$
\begin{equation}
\label{eq:second_est}
\frac{1}{2}\int_{\Rd}\vert \Lsm u_\eps(t)\vert^2 \d x + \int_{0}^t\int_{\Rd}m_\eps(u_\eps) \vert \nabla p_\eps\vert^2 \d x \d t \le
\frac{1}{2}\int_{\Rd}\vert \Lsm u_0\vert^2 \d x,
\end{equation}
that is,
\begin{align}
\label{eq:eq:second_est_infty}
\frac{1}{2}\int_{\Rd}\vert \Lsm u_\eps(t)\vert^2 \d x + \int_{0}^{\infty}\int_{\Rd}m_\eps(u_\eps) \vert \nabla p_\eps\vert^2 \d x \d t \le
\frac{1}{2}\int_{\Rd}\vert \Lsm u_0\vert^2 \d x,
\end{align}
valid for a.a.  $t>0$.  Moreover, we observe that $u_\eps$ is indeed a bit more regular in space. In fact, being $p_\eps = \Ls u_\eps \in L^2(0,+\infty;H^1(\Rd))$ (recall that $m_\eps(u_\eps)\ge \eps$ a.e. in $\Rd\times (0,+\infty)$) we have
that $u_\eps\in L^2(0,+\infty;H^{1+2s}(\Rd))$. We will show that this estimate
produces, in the $\eps\to 0$ limit, the Energy Estimate \eqref{eq:energy_est}.

\medskip

{\noindent \bf Entropy Estimate}
System \eqref{eq:approx_prob_strong} admits a further
estimate that is in principle not uniform with respect to the parameter $\eps$. But this estimate will produce
the Entropy Estimate \eqref{eq:entropy_est} in the limit $\eps\to 0$.
We will need this observation:
%
for any $\eps>0$, we consider the smooth and positive real function $f_\eps$ such that
\begin{equation}
\label{eq:approx_entropy}
f_\eps '' = \frac{1}{m_\eps} \,\,\,\hbox{ in }\mathbb{R}.
\end{equation}
Without loss of generality we can choose $f_\eps$ in such a way that
$f_\eps(1) = f_\eps^{'}(1)= 0$. Thus,
\begin{equation}
\label{eq:def_f}
f_\eps(y) = \int_{1}^y\big( \int_{1}^w \frac{1}{r^+ + \eps} \d r \big) \d w,\,\,\,\,\,y\in \mathbb{R}.
\end{equation}
An important property of $f_\eps$ is that, when $y<0$, there holds
\begin{equation}
\label{eq:y<0}
f_\eps(y) \ge \frac{y^2}{2\eps}.
\end{equation}
To prove the above estimate we observe that for $y<0$
we have
\[
f_\eps(y) = \int_{y}^1\big( \int_{w}^1 \frac{1}{r^+ + \eps} \d r \big) \d w.
\]
Thus, setting $g(w):=\int_{w}^1 \frac{1}{r^+ + \eps} \d r$, we immediately have
that
\[
f_\eps(y)\ge \int_{y}^0 g(w) \d w.
\]
Moreover, when $w<0$, there holds
$ g(w) \ge \frac{1}{\eps}w. $

As a result, we have
\[
f_\eps(y) \ge \int_{y}^0 g(w) \d w \ge \frac{1}{\eps}\int_{y}^0 w \d w = \frac{y^2}{2\eps}.
\]

\medskip

In order to fully justify the argument, we work at the approximate level of problem
\eqref{eq:approx_prob_strong_M}. Therefore we let $u_{\eps,M}$ be a solution of
\eqref{eq:approx_prob_strong_M} and we consider, for any $M>0$, a positive real function $f_{\eps,M}$ such that
\[
f''_{\eps,M} = \frac{1}{m_{\eps}^M}.
\]
Note that $f_{\eps,M}$ is defined as in \eqref{eq:def_f} and that $f_{\eps,M}\xrightarrow{M\to +\infty}f_\eps$
for any $y\in \R$.

We test  \eqref{eq:approx_prob_strong_M} with $f'_{\eps,M}(u_{\eps,M})$.
Note that, being $u_{\eps,M}\in L^2(0,+\infty;H^{1+2s}(\Rd))$ then
$f'_{\eps,M}(u_{\eps,M})\in L^2(0,+\infty;H^{1}(\Rd))$, at least, together with $\partial_{t}u_{\eps,M}$.
Thus, the computations are justified.

Therefore, using that $p_{\eps,M} = \Ls u_{\eps,M}$ and that $f_{\eps,M} '' = \frac{1}{m_{\eps,M}}$, integrating with respect to time
we get
\begin{align}
\label{eq:first_est_M}
\int_{\Rd}f_{\eps,M}(u_{\eps,M}(t)) \d x + \int_{0}^t\int_{\Rd}\vert \Lsm(\nabla u_{\eps,M})\vert^2 \d x \d t = \int_{\Rd}f_{\eps,M}(u_0) \d x,
\,\,\,\hbox{ for a.a. } t\le T.
\end{align}

Therefore, if we let $M\to +\infty$, we get
the Entropy Estimate
\begin{align}
\label{eq:first_est}
\int_{\Rd}f_\eps(u_{\eps}(t)) \d x + \int_{0}^t\int_{\Rd}\vert \Lsm(\nabla u_{\eps})\vert^2 \d x \d t \le\int_{\Rd}f_{\eps}(u_0) \d x,
\,\,\,\hbox{ for a.a. } t\le T.
\end{align}

\medskip


\subsection{Passage to the limit: Proof of Theorem \ref{th:existence}}
\label{ss:limit_proc}
The energy and the entropy estimate give some important uniform estimates (with respect to $\eps$)
on the approximate solutions $u_\eps$. We work on bounded time intervals $(0,T)$ with $T>0$ for compactness reasons.
First of all, the energy estimate \eqref{eq:second_est}
gives that the sequence $\Lsm u_\eps$ is bounded, uniformly with respect to $\eps$, in $L^{\infty}(0,T; L^2(\Rd))$,
namely $u_\eps$ is bounded in $L^\infty(0,T; \dot{H}^s(\Rd))$.

Thus, the Hardy-Littlewood-Sobolev inequality \cite[Theorem V.1]{stein} furnishes that $u_\eps$ is bounded in $L^\infty(0,T;L^{\frac{2d}{d-2s}}(\Rd))$.
Consequently, we have that $m_\eps(u_\eps)$ is bounded in $L^\infty(0,T;L^{\frac{2d}{d-2s}}_{loc}(\Rd))$.
 The entropy estimate
\eqref{eq:first_est} gives that $\Lsm \nabla u_\eps$ is bounded in $L^2(0,T;L^2(\Rd))$. Hence,
 we gain some spatial regularity for $u_\eps$ and for
$p_\eps = \Ls u_\eps$, namely
\begin{equation}
\label{eq:regularity_up}
\| u_\eps\|_{L^2(0,T;\dot{H}^{1+s}(\Rd))} + \| p_\eps\|_{L^2(0,T;\dot{H}^{1-s}(\Rd))} \le C,
\end{equation}
with $C$ possibly depending on $T$.
The energy estimate \eqref{eq:second_est}
gives that
\[
\sqrt{m_\eps(u_\eps)}  \nabla p_\eps \,\,\,\hbox{ is uniformly bounded in } L^2(0,T;L^2(\Rd)).
\]
Thus, since $\sqrt{m_\eps(u_\eps)}$ is bounded in $L^\infty(0,T; L^{2p_s}_{loc}(\Rd))$ ($p_s:= \frac{2d}{d-2s}$) we get
that
\[
m_\eps(u_\eps)  \nabla p_\eps \,\,\,\hbox{ is uniformly bounded in } L^2(0,T;L^{\frac{2p_s}{1+p_s}}_{loc}(\Rd)).
\]
Consequently, a comparison in the equation
\eqref{eq:approx_prob_strong} gives the estimate on the time derivative $\partial_t u_\eps$,
namely
\begin{equation}
\label{eq:est_time_der}
\partial_t u_\eps \,\,\,\hbox{ uniformly bounded in } L^2(0,T;W^{-1, \frac{2p_s}{1+p_s}}(\Rd)).
\end{equation}

Then, the Aubin-Lions compactness lemma
gives that for any $\delta>0$, $u_\eps$ is strongly compact in
$L^2(0,T; H^{1+s -\delta}(K))$, for any compact $K\subset \Rd$. Thus, there exists $u\in L^2(0,T; H^{1+s}_{loc}(\Rd))$ and a subsequence of $\eps$ for which
\begin{align}
\label{eq:weak_u_Lp}
& u_\eps \xrightarrow{\eps\to 0} u \,\,\,\,\hbox{ weakly star in } L^\infty(0,T; L^{\frac{2d}{d-2s}}(\Rd))\\
\label{eq:strong_u}
&  u_\eps \xrightarrow{\eps\to 0} u \,\,\,\,\hbox{ strongly in } L^2(0,T; H^{1+s -\delta}(K)) \,\,\,\hbox{ for any } K \,\,\hbox{ compact in } \Rd.
\end{align}
Moreover, $u_\eps \xrightarrow{\eps \to 0} u $ almost everywhere in $\Rd\times (0,T)$. In particular, we have
that $m_\eps(u_\eps) \xrightarrow{\eps\to 0} (u)^{+}$ almost everywhere in $\Rd\times (0,T)$.
Finally,
since we have that for any $\vf \in C^{\infty}_{c}(\Rd\times [0,T))$
\[
\int_{0}^T\int_{\Rd}\Ls u_\eps \vf \d x \d t = \int_{0}^T \int_{\Rd}u_\eps \Ls \vf \d x \d t \xrightarrow{\eps\to 0}
\int_{0}^T\int_{\Rd} u \Ls\vf \d x \d t,
\]
we conclude that, denoting with $p$ the weak-star limit of $p_\eps$ in $L^{\infty}(0,T; H^{-s}(\Rd))$,
we have
\[
p = \Ls u,
\]
at least in the sense of distributions.
Actually much more is true. In fact,
the estimate \eqref{eq:regularity_up} gives that
$p_\eps$ is bounded in $L^2(0,T;H^{1-s}(\Rd))$. Thus, we have that
\begin{equation}
\label{eq:weak_conv_p}
p_\eps \xrightarrow{\eps \to 0}  p \,\,\,\hbox{ weakly in }L^2(0,T;H^{1-s}(\Rd)),
\end{equation}
which implies that $p$ is in $L^2(0,T;\dot{H}^{1-s}(\Rd))$ and that the relation $p = \Ls u$ holds, at least, almost everywhere in
$\Rd\times (0,T)$.

We have all the ingredients to pass to the limit in the following weak formulation of \eqref{eq:approx_prob_strong}
\begin{align}
\label{eq:weak_approx}
\int_{0}^{T}\int_{\Rd} u_\eps\partial_t \vf \d x \d t &-\int_{0}^T\int_{\Rd}p_\eps\nabla (m_\eps(u_\eps)) \cdot \nabla \vf \,\d x \d t
- \int_{0}^T\int_{\Rd}m_\eps(u_\eps)p_\eps \Delta \vf \,\d x \d t \nonumber\\
&= - \int_{\Rd} u_{0} \vf(x,0) \d x\,\,\,\,\,\forall \vf \in C^{\infty}_c(\Rd\times [0,+\infty)),\nonumber\\
p_\eps & = \Ls u_\eps\,\,\,\hbox{ a.e. in } \Rd\times(0,T)
\end{align}
We note that the first term in the left-hand side
converges to the expected limit thanks, e.g., to the dominated convergence.
Now we pass to the limit in the nonlinear term.
Since (see \cite[Lemme  1.1]{Stampacchia65}) almost everywhere there holds
\[
\nabla(u_\eps)_+ =\mathscr{H}(u_\eps)\nabla u_\eps:=
\begin{cases}
\nabla u_\eps \,\,\,\,\,\,\left\{u_\eps{> 0} \right\},\\
0, \,\,\,\,\hbox{ otherwise }
\end{cases}
\]
($\mathscr{H}$ is the Heaviside function)
\begin{equation}
\label{eq:limit_nonlin}
\int_{0}^T\int_{\Rd}p_\eps\nabla(m_\eps(u_\eps)) \cdot \nabla \vf \,\d x \d t
= \int_{0}^T\int_{\Rd }p_\eps\mathscr{H}(u_\eps)\nabla u_\eps\cdot\nabla \vf\,\d x \d t.
\end{equation}
Moreover, we decompose
\[
\int_{0}^T\int_{\Rd}m_\eps(u_\eps)p_\eps \Delta \vf \,\d x \d t = \int_{0}^T\int_{\Rd}(u_\eps)_+ p_\eps\Delta \vf
\,\d x \d t - \eps \int_{0}^T\int_{\Rd}\nabla p_\eps \cdot \nabla \vf \,\d x \d t = I_1^{\eps}+ I_2^{\eps}.
\]
The second term $I_2^{\eps}$ tends to zero when $\eps\searrow 0$. In fact, for a constant $C$ that depends on $\vf$, we have, thanks to the Schwarz inequality,
\[
\vert I_2^{\eps}\vert \le \eps\|\nabla p_\eps\|_{L^2(\Rd\times (0,T))}\|\nabla \vf\|_{L^2(\Rd\times (0,T)}
\le \eps C\int_{\Rd}\vert \Lsm u_0\vert^2 \d x,
\]
thanks to \eqref{eq:second_est}. The term $I_1^\eps$ tends
to the expected limit since we have that, for any compact $K\subset \Rd$,
\[
u_\eps \to u \,\,\,\,\hbox{ strongly in } L^2(0,T;H^{1+s-\delta}(K)), \,\,\,\,\forall \delta>0,
\]
and
\[
p_\eps \to p\,\,\,\,\,\,\,\,\hbox{ weakly in } L^2(0,T;L^2(K)).
\]
Thus,
\[
\lim_{\eps\to 0}\int_{0}^T\int_{\Rd}m_\eps(u_\eps)p_\eps \Delta \vf \,\d x \d t= \int_{0}^T\int_{\Rd}
(u)_+ p \Delta \vf\,\d x \d t,
\]
Moreover, since for any compact $K\subset \Rd$, we also have
\[
\nabla u_\eps \to \nabla u\,\,\,\,\hbox{ strongly in } L^2(0,T;H^{s-\delta}(K)), \,\,\,\,\,\forall \delta>0,
\]
we get that
\[
\lim_{\eps \to 0}\int_{0}^T\int_{\Rd}p_\eps\mathscr{H}(u_\eps)\nabla u_\eps\cdot\nabla \vf\,\d x \d t
=\int_{0}^T\int_{\Rd\cap\left\{u\ge 0 \right\} }p\nabla u\cdot\nabla \vf\,\d x \d t.
\]

As a result, we have that $u$ verifies, for any $\vf\in C^{\infty}_c(\Rd\times[0,T))$,
\begin{align}
\label{eq:weak_u_positive}
\int_{0}^{T}\int_{\Rd} u\partial_t \vf \d x \d t &- \int_{0}^T\int_{\Rd}(u)^+  p \Delta \vf \,\d x \d t
-\int_{0}^T\int_{\Rd\cap\left\{u\ge 0 \right\} }p\nabla u\cdot\nabla \vf\,\d x \d t\nonumber\\
&= - \int_{\Rd} u_{0} \vf(x,0) \d x\,\,\,\,\,\,\hbox{ and }\,\,\,\,\,
p = \Ls u\,\,\,\hbox{ a.e. in } \Rd\times(0,T).
\end{align}
Moreover,
we have that by passing to the limit in \eqref{eq:first_est} we obtain \eqref{eq:entropy_est} thanks to semincontinuity.
Finally, since \eqref{eq:second_est} implies that $\xi_\eps:=m_\eps^{1/2}(u_\eps)\nabla p_\eps$ is bounded uniformly
in $L^2(0,T;L^2(\Rd))$ we have that there exists a vector field $\xi\in L^2(0,T;L^2(\Rd))$ to which $\xi_\eps$ weakly converges and such that
\eqref{eq:energy_est} holds.

\medskip

{Thus, it remains to identify $\xi$ as in \eqref{eq:def_xi}.
To this purpose, we introduce the vector field $\zeta_\eps:= m_\eps(u_\eps) \nabla p_\eps = m_\eps(u_{\eps})^{1/2}\xi_\eps$
and we note that, on the one hand, $\zeta_\eps$ weakly converges to some $\zeta\in L^2(0,T;L^2(\Rd))$ and $\zeta=u^{1/2}\xi$.
On the other hand, we have that, since $\zeta_\eps=\nabla(p_\eps m_\eps(u_\eps)) - p_\eps
\mathcal{H}(u_\eps)\nabla u_\eps$,
\begin{equation}
\label{eq:id_distr}
\zeta_\eps\xrightarrow{\eps \to 0}\nabla(pu) - p\nabla u\,\,\,\,\,\hbox{ in the sense of distributions on }\Rd\times (0,\infty).
\end{equation}
Thus,
\begin{equation}
\label{eq:implicit_xi}
\zeta = \nabla(pu) - p\nabla u = u^{1/2}\xi \,\,\,\,\,\hbox{ a.e. in }\Rd\times (0,+\infty).
\end{equation}
In particular, for those points in which $u>0$ we can express $\xi$ in terms of $\zeta$ as $\xi = u^{-1/2}\zeta$
}

In order to prove that $u$ is indeed a solution of \eqref{eq:weak_formulation} it remains to
show that $u\ge 0$ almost every where in $\Rd\times (0,T)$. \\
Note that since (cf.\eqref{eq:est_time_der})
\begin{equation}
\label{eq:uteps}
\|\partial_t u_\eps\|_{L^2(0,T;(W^{1, \frac{2p_s}{1+p_s}}_{loc}(\Rd))^{'})} \le C, \,\,\,\,\hbox{ uniformly in } \eps>0,
\end{equation}
we get, by semicontinuity of norms, that $\forall T>0$,
\begin{equation}
\label{eq:ut}
\|\partial_t u\|_{L^2(0,T;W^{1, \frac{2p_s}{1+p_s}}_{loc}(\Rd))^{'})} \le C.
\end{equation}
Moreover, this estimate is also uniform with respect to time and thus
\begin{equation}
\label{eq:est_ut}
\|\partial_t u\|_{L^2(0,+\infty;W^{1, \frac{2p_s}{1+p_s}}_{loc}(\Rd))^{'})} \le C.
\end{equation}

\subsubsection{Nonnegativity}

To prove positivity we exploit the entropy estimate \eqref{eq:first_est}. More precisely, the positivity
of $u$ follows from the fact that
\begin{equation}
\label{eq:entropy11}
\sup_{\eps>0} \sup_{t\in [0,T]}\int_{\Rd}f_\eps(u_\eps(t)) \d x  < +\infty
\end{equation}
combined with \eqref{eq:y<0}. We aim at proving that $u\ge 0$ for almost any $(x,t)\in \mathbb{R}^d\times (0, T).$
To this end, we fix $t\in (0,T)$, a compact subset $K$ of $\Rd$ and we assume, by contradiction,
that the set
\[
\left\{ x\in \Rd\cap K: u(x,t) <0\right\}
\]
has positive measure. Since
\[
\left\{ x\in \Rd\cap K: u(x,t) <0\right\} = \bigcup_{n=1}^{+\infty}\left\{x\in \Rd\cap K: u(x,t) < -\frac{1}{n} \right\},
\]
this implies that
that for some fixed $\lambda>0$
the set
\[
N_\lambda :=\left\{x\in \Rd\cap K: u(x,t) \le -\lambda \right\}
\]
has positive Lebesque measure.
Now, since the sequence $u_\eps(\cdot, t)\xrightarrow{\eps\to 0} u(\cdot, t)$ almost everywhere in $K$, the Severini-Egorov Theorem
furnishes that for any $\eta>0$ there exists a measurable set $G_\eta\subset K$ such that $\vert K\setminus G_\eta\vert \le \eta$
and such that
\[
u_\eps(\cdot, t) \xrightarrow{\eps\to 0}u(\cdot,t) \,\,\,\,\hbox{ uniformly on } G_\eta.
\]
We fix $\eta$ and we find $\bar{\eps}>0$ such that if $\eps<\bar{\eps}$ there holds
\[
u_\eps(\cdot, t) \le -\frac{\lambda}{2}\,\,\,\,\,\hbox{ on } G_\eta\cap N_\lambda.
\]
On $G_\eta\cap N_\lambda$ we have (recall \eqref{eq:y<0})
\begin{align*}
f_\eps(u_\eps(x,t)) &=  \int_{u_\eps(x,t)}^1 g(w) \d w = \int_{u_\eps(x,t)}^{-\lambda/2} g(w) \d w
 + \int_{-\lambda/2}^1 g(w) \d w\nonumber\\
&\ge \int_{u_\eps(x,t)}^{-\lambda/2} g(w) \d w = f_\eps(-\lambda/2)\ge \frac{\lambda^2}{8\eps^2}.
\end{align*}
Thus, thanks to Fatou Lemma we get
\[
\liminf_{\eps\to 0}\int_{\Rd}f_\eps(u_\eps(x,t))\d x \ge \liminf_{\eps\to 0}\int_{G_\eta\cap N_\lambda}f_\eps(u_\eps(x,t))\d x
= +\infty.
\]
This is in contradiction with \eqref{eq:entropy11}, which would imply that for all $t\ge 0$
\[
\limsup_{\eps \to 0}\int_{\Rd}f_\eps(u_\eps(x,t))\d x <+\infty.
\]
Hence the positivity is proved. As a consequence we have that $u$ is a solution of \eqref{eq:weak_formulation}.

\subsubsection{Conservation of mass}
\label{ss:cons_mass}
We take a smooth cut-off function $g:[0,+\infty)\to [0,1]$
such that
\begin{equation}
\label{eq:cut_off}
\begin{cases}
g(r) = 1\,\,\,\hbox{ in } [0,1]\\
g(r) = 0\,\,\,\hbox{ in } [2, +\infty).
\end{cases}
\end{equation}
and such that
\[
\|g'\|_{L^\infty(\R)}+ \|g''\|_{L^\infty(\R)}\le 2.
\]
Then, for $R>0$ we set $\phi_{R}(x) := g\big(\frac{\vert x\vert}{R}\big)$.
For any $h>0$ and $t\in [0, T)$ such that $t+h<T$ we also introduce
\begin{equation}
\label{eq:zeta}
\zeta_{h,t}(r):=
\begin{cases}
1, \,\,\,\,\,0\le r\le t\\
1- \frac{r-t}{h},\,\,\,\,t\le r\le t+h \\
0, \,\,\,\,t+h\le r\le T.
\end{cases}
\end{equation}
There holds that for any $t\in [0,T)$ $\zeta_{h,t}(\cdot)\xrightarrow{h\to 0}\chi_{[0,t]}(\cdot)$.
We choose in the weak formulation \eqref{def:weak} the test function
\[
\varphi_{h,t,R}(x,r):=\zeta_{h,t}(r)\phi_{R}(x), \,\,\,\,\,\hbox{ for } (x,r)\in \Rd\times [0,T)
\]
and we obtain
\begin{align}
\label{eq:weak_mass}
-\int_{0}^T\int_{\Rd}u \partial_r \varphi_{h,t,R} \d x\d r &= \int_{\Rd\times (0,T)} up \Delta \varphi_{h,t,R}\d x \d r
+\int_{\Rd\times (0,T)}p\nabla u\cdot\nabla \varphi_{h,t,R} \d x \d r\nonumber\\
 &\int_{\Rd}u_0(x)\phi_{R}(x)\d x
\end{align}
Since $u, \nabla u$ and $p$ are, at least $L^1_{loc}(\Rd\times (0,T))$ functions
we have
\[
\lim_{h\to 0}\int_{0}^T\int_{\Rd}u \partial_r \varphi_{h,t,R} \d x\d r = -\int_{\Rd}u(x,t)\phi_{R}(x)\d x\,\,\,\,\hbox{ for a.a. } t\in (0,T)\\
\]
and
\begin{align*}
\lim_{h\to 0}\Big(\int_{\Rd\times (0,T)} up \Delta \varphi_{h,t,R}\d x \d r
+\int_{\Rd\times (0,T)}p\nabla u\cdot\nabla \varphi_{h,t,R} \d x \d r\Big)\\
=\int_{\Rd\times (0,t)} up \Delta \phi_{R}\d x \d r
+\int_{\Rd\times (0,t)}p\nabla u\cdot\nabla \phi_{R} \d x \d r
\end{align*}
Now, since $\nabla \phi_R(x) =\frac{1}{R} g'\big(\frac{\vert x\vert}{R}\big)\frac{x}{\vert x\vert} $
and $\Delta \phi_R(x) = \frac{1}{R^2}g''\big(\frac{\vert x\vert}{R}\big) + \frac{d-1}{R \vert x\vert}g'\big(\frac{\vert x\vert}{R}\big)
$ we get
\[
\vert \nabla \phi_R(x)\vert \le \frac{2}{R}, \,\,\,\,\,\,\vert \Delta \phi_R(x)\vert \le \frac{C}{R^2}.
\]
Thus, since, for any compact $K\subset \Rd$
there holds that $u\in L^2(0,T; L^2(K))$, $\nabla u\in L^2(0,T; L^2(K))$ and $p\in  L^2(0,T; L^2(K))$, we have that
\begin{equation}
\label{eq:mass_cons2}
\lim_{R\to +\infty}\abs{\int_{0}^t \int_{\Rd}u\, p \Delta \phi_R (x) \d x \d r +\int_{0}^t\int_{\Rd}p\nabla u\cdot\nabla \phi_R(x) \d x \d r}=0.
\end{equation}
Moreover, as $u\ge 0$ in $\Rd\times (0,T)$, the monotone convergence Theorem gives that, for almost any $t\in (0,T)$,
\[
\lim_{R\to +\infty} \int_{\Rd}u(x,t) \phi_R(x) \d x = \int_{\Rd}u(x,t) \d x.
\]
Consequently, since (recall that $u_0\in L^1(\Rd)$)
\[
\lim_{R\to \infty}\int_{\Rd}u_0\phi_R\d x = \int_{\Rd} u_0\d x <+\infty,
\]
by passing to the limit $R\to +\infty$ in \eqref{eq:weak_mass} we obtain, for almost any $t<T$
\[
\int_{\Rd}u(x,t) \d x  = \int_{\Rd} u_0\d x <+\infty,
\]
that gives the desired conservation of mass.

\subsubsection{First Moments estimate}
\begin{lemma}
\label{lem:first_mom}
Let $u$ be a weak solution as constructed before.
Then,
\begin{equation}
\label{eq:first_moment}
\int_{\Rd}\vert x\vert u(x,t)\d x \le \int_{\Rd}\vert x\vert u_0(x)\d x + C(T),\,\,\,\hbox{ for a.a. }t\le T.
\end{equation}
\end{lemma}
An interesting feature of this estimate is that it works for all weak solutions constructed in Theorem \ref{th:existence}.

\begin{proof}[Proof of Lemma \ref{lem:first_mom}]
First of all we notice that
weak solutions verify that
for every $\varphi\in C^{2}_{c}(\Rd)$
\begin{align}
\label{eq:weak}
\int_{\Rd} u(x,t)\vf \d x = \int_{\Rd} u_0(x)\vf \d x
+ \int_{0}^{t}\int_{\Rd}p u \Delta \vf \,\d x \d r
+\int_{0}^{t}\int_{\Rd} p\nabla  u \cdot\nabla \vf \,\d x \d r
\nonumber\\
p = \Ls u\,\,\,\hbox{ a.e. in } \Rd\times(0,+\infty).
\end{align}
We take as a test function in \eqref{eq:weak} the function $\phi (x) = \vert x\vert \phi_R(x)$ where $\phi_R$ is a smooth cut-off function (see \eqref{eq:cut_off} in Subsection \ref{ss:cons_mass}) such that
\[
\vert \nabla \phi_R(x)\vert \le \frac{2}{R}, \,\,\,\,\,\,\vert \Delta \phi_R(x)\vert \le \frac{C}{R^2}.
\]
To be precise, to obtain a smooth $\phi$ one should also round off the function $\vert x\vert$ around the origin. The proof is analogous and for the sake of simplicity we use $\vf = \vert x\vert \phi_R$.

We have
\[
\nabla \vf = \vert x\vert \nabla \phi_R + \frac{x}{\vert x\vert}\phi_R,
\]
and
\[
\Delta \vf = \vert x\vert \Delta\phi_R + 2  \frac{x}{\vert x\vert}\cdot \nabla\phi_R + \frac{d-1}{\vert x\vert}\phi_R.
\]
Thus,
\[
\int_{0}^{t}\int_{\Rd} p\nabla  u \cdot\nabla \vf \,\d x \d r =
\int_{0}^{t}\int_{\Rd} p\nabla  u \cdot \frac{x}{\vert x\vert}\phi_R\,\d x \d r
+ \int_{0}^{t}\int_{\Rd} \vert x\vert p\nabla  u\cdot \nabla \phi_R \,\d x \d r
\]
Due to the regularity of the weak solution, the first integral is clearly bounded by a constant that depends on the final time $T$.
Being $\nabla \phi_R$ supported on $R\le \vert x\vert\le 2R$, thanks to
$\vert\nabla \phi_R\vert \le C/R$, we have that also the second integral is bounded by a constant possibly depending on the final time $T$.

We bound the second integral in the right hand side of \eqref{eq:weak}.
We have
\begin{align*}
\int_{0}^{t}\int_{\Rd}p u \Delta \vf \,\d x \d r &=
\int_{0}^{t}\int_{\Rd}p u \vert x\vert \Delta \phi_R\,\d x \d r + 2\int_{0}^{t}\int_{\Rd}p u \frac{x\cdot \nabla \phi_R}{\vert x\vert}\,\d x \d r\\
&+  (d-1)\int_{0}^{t}\int_{\Rd}p u \frac{\phi_R}{\vert x\vert}\,\d x\d r.
\end{align*}
The first two integrals are easily bounded using the regularity of the weak solutions and the properties of $\nabla\phi_R$ and of $\Delta \phi_R$.
Regarding the third integral we can reason as follows.
First, we write ($R>1$)
\[
(d-1)\int_{0}^{t}\int_{\Rd}p u \frac{\phi_R}{\vert x\vert}\,\d x\d r =
(d-1)\int_{0}^{t}\int_{B_1(0)}p u \frac{\phi_R}{\vert x\vert}\,\d x\d r
\]
\[
+ (d-1) \int_{0}^{t}\int_{\Rd\setminus B_1(0)}p u \frac{\phi_R}{\vert x\vert}\,\d x\d r.
\]
Now, for $\vert x\vert >1$ we have that $\frac{\phi_R}{\vert x\vert}\le 1$ and thus, the second integral is bounded using that $up\in L^1(\Rd\times (0,T))$.
Regarding the first integral we first note that $u\in L^2(0,T; L^{q^*}(\Rd))$ and that
$p\in L^2(0,T;L^{r^*}(\Rd))$ due to Sobolev inequality.
More precisely,
\[
q^* = \frac{2d}{d-2-2s},\,\,\,\,r^* = \frac{2d}{d-2+ 2s}.
\]
Thus, defining $q\ge 1$ in such a way that
\[
\frac{1}{q} = 1 - \frac{1}{q^*} -\frac{1}{r^*},
\]
we have that $q<d$ and thus $\vert x\vert^{-1} \in L^{q}(B_1(0))$.
Therefore, the Young inequality shows that
\[
\int_{0}^{t}\int_{\Rd\setminus B_1(0)}p u \frac{\phi_R}{\vert x\vert}\,\d x\d r \le C(T).
\]
Collecting all the above estimates we have
\eqref{eq:first_moment}.
\end{proof}


\section{\bf Self-similar Solutions}
\label{sec:self_similar}

In this Section we construct self-similar weak solutions  of System \eqref{eq:prob_intro1} (in the sense of Definition \ref{def:weak}).  More precisely,  we look for solutions of the form
\begin{equation}
\label{eq:source_type_form}
u(x,t) = \frac{1}{(1+t)^\alpha}v\Big(\frac{x}{(1+t)^\beta}, \log (1+t)\Big),
\end{equation}
where the profile function $v:\mathbb{R}^d\times \mathbb{R}\to \mathbb{R}$
is to be appropriately determined and the parameters $\alpha$ and $\beta$ are given by
\begin{equation}
\label{eq:alpha/beta}
\alpha = \frac{d}{d  + 2(1+ s)}, \,\,\,\,\beta = \frac{1}{d+ 2(1+ s)},
\end{equation}
due to the constraints that we will find below. In what follows, we will set
$$
y:=\frac{x}{(1+t)^\beta}, \quad \tau:=\log (1+t), \quad w = \Ls v.
$$
It is interesting to observe that the profile function $v$ will have compact support. Hence,
the self-similar solutions will have compact support as well (in the space variable).
As it is now customary (see \cite{Ca_To2000} and \cite{Vaz2007} and references therein),
the self-similar solutions of \eqref{eq:prob_intro1} are related to stationary solutions of a nonlinear (and nonlocal in this case) Fokker-Planck type equation solved by the profile $v$. Thus, as a first step, we look for an equation to be satisfied by $v$.
Clearly, since $v$ is related to a weak solution $u$ by the relation \eqref{eq:source_type_form}, it has the very same (low) regularity. Thus, the following computations are only formal at this moment. Therefore, assuming all the regularity needed to justify the computations, we have
\begin{eqnarray*}
\label{eq:partialt}
\partial_t u = -\alpha (1+t)^{-\alpha-1}v -\beta (1+t)^{-\alpha-1} \nabla v\cdot y + (1+ t)^{-\alpha-1}\partial_\tau v.\\
p = \Ls u = \frac{1}{(1+t)^{\alpha}}\Big(\Ls v \Big)\Big(\frac{x}{(1+t)^\beta}\Big) (1+t)^{-2s\beta}
 = (1+t)^{-\alpha - 2s\beta} w.
\end{eqnarray*}
Moreover,
\begin{align}
\Delta p &= (1+t)^{-\alpha -2s\beta - 2 \beta} \Delta_y w\nonumber\\
\nabla u &= (1+t)^{-\alpha -\beta}\nabla_y v,\nonumber\\
\nabla p &= (1+t)^{-\alpha-\beta-2s\beta}\nabla_y w.
\end{align}
Thus, the problem
\[
\begin{cases}
\partial_t u - \div(u\nabla p) = \partial_t u - \nabla u\cdot \nabla p - u\Delta p = 0,\,\,&\hbox{ in } \mathbb{R}^d\times(0,+\infty)\\
p = \Ls u \,\,&\hbox{ in } \mathbb{R}^d\times(0,+\infty)
\end{cases}
\]
rewrites as
\begin{align}
(1+t)^{-\alpha-1}\partial_{\tau}v-\alpha (1+t)^{-\alpha-1}v -\beta (1+t)^{-\alpha-1} \nabla_y v\cdot y\nonumber\\
-(1+t)^{-2\alpha-2\beta-2s\beta}\nabla_y v\cdot\nabla_y w -(1+ t)^{-2\alpha-2\beta-2s\beta}v \Delta_y w = 0,\\
w = \Ls v.
\end{align}
Now, the choice made above for $\alpha$ and $\beta$ implies the algebraic relation
\begin{equation}
\label{eq:cond_alphabeta}
\alpha + 2\beta(1+s) = 1,
\end{equation}
that  allows us to eliminate the time factors in the above equation. We thus  obtain an expression involving only the rescaled variables $\tau$ and $y$. Namely,
\begin{equation}
\label{eq:v}
\begin{cases}
\partial_\tau v-\alpha v -\beta \nabla_y v\cdot y
-\nabla_y v\cdot\nabla_y w - v \Delta_y w = 0\,\,&\hbox{ in } \mathbb{R}^d\times(0,+\infty)\\
w = \Ls v\,\,\,&\hbox{ in } \Rd\times (0,+\infty).
\end{cases}
\end{equation}
Moreover, since also impose a second relation $\alpha=\beta d$, equation \eqref{eq:v} can be written in divergence form, so that conservation of mass will be guaranteed (at this stage only formally). More precisely, the system contains the following \sl nonlinear and nonlocal Fokker-Planck type equation:\rm
\begin{equation}
\label{eq:fokker-cahn}
\begin{cases}
\partial_\tau v -\div_{y}\Big(v\big( \nabla_y w + \beta y\big)\Big)= 0,\\
w = \Ls v.
\end{cases}
\end{equation}

\subsection{The structure of the stationary solutions}
\label{ssec:stationary_sol}
\

Self-similar solutions are thus related to stationary solutions of \eqref{eq:fokker-cahn}.
Therefore, we first analyse the structure of the stationary solutions.

\medskip

(i) First of all, we make a reduction in the set of possible solutions and concentrate on those stationary solutions of \eqref{eq:fokker-cahn} such that
\begin{equation}
\label{eq:stationary2}
\begin{cases}
v\nabla_y\big( w + \frac{\beta}{2} \vert y\vert^2\big) = 0\,\,\,\,&\hbox{ in } \,\,\Rd,\\
w = \Ls v\,\,\,&\hbox{ in } \Rd.
\end{cases}
\end{equation}
As in the parallel study made in \cite{CVlarge} for negative values of $s$, this reduction must be justified by the later analysis of the long-time behavior and the asymptotic convergence to a self-similar profile.

(ii) Assuming for the moment that we have continuous solutions,\footnote{This assumption will be justified  a posteriori in our analysis.} if we denote by $\mathscr{P}$ the positivity set of $v$, i.e. the set
\begin{equation}
\label{eq:positivityset_v}
\mathscr{P}:=\left\{y\in \Rd: v(y)>0\right\},
\end{equation}
then we have
\begin{equation}
\label{eq:stat3}
\begin{cases}
\nabla\big(w +\frac{\beta}{2}\vert y\vert^2\big) =0\,\,\,\,&\hbox{ in } \mathscr{P},\\
v\ge 0, \quad w = \Ls v \quad &\hbox{ in } \Rd.
\end{cases}
\end{equation}
Thus, on the every connected component $\mathscr{C}_i$ of $\mathscr{P}=\bigcup_{i\in \mathscr{N}}\mathscr{C}_i$, there exists a constant $c_i$  such that
\begin{equation*}
\label{eq:stat41}
\begin{cases}
\Ls v = c_i - \frac{\beta}{2}\vert y\vert^2\,\,\,&\hbox{ in }\mathscr{C}_i\\
v = 0 \,\,\,&\hbox{ in }\Rd\setminus \mathscr{P}.
\end{cases}
\end{equation*}
The above problem can be rewritten as
\begin{equation}
\label{eq:stat4}
\begin{cases}
\Ls v = \sum_{i\in \mathscr{N}}c_i\chi_i(x) - \frac{\beta}{2}\vert y\vert^2\,\,\,&\hbox{ in }\mathscr{P},\\
v = 0 \,\,\,\,&\hbox{ in }\Rd\setminus \mathscr{P},
\end{cases}
\end{equation}
($\chi_i$ is the characteristic function of $\mathscr{C}_i$).
Necessarily, the constants $c_i$ cannot be all negative, otherwise $\mathscr{P}=\emptyset$ thanks to the maximum principle. Note that since the operator is nonlocal we do not claim  positivity of all constants. In any case, this fact will not be important for our results.

(iii) Now we restrict to look at continuous solutions for which $\mathscr{P}$ is connected.
Let us denote with $v_1$ a solution of \eqref{eq:stat4}, namely a particular solution of \eqref{eq:stat3}.
In this case, problem \eqref{eq:stat4} becomes
\begin{equation}
\label{eq:stat4bis}
\begin{cases}
\Ls v = C - \frac{\beta}{2}\vert y\vert^2\,\,\,&\hbox{ in }\mathscr{P}\\
v = 0 \,\,\,\,&\hbox{ in }\Rd\setminus \mathscr{P}
\end{cases}
\end{equation}
($C\in (0,+\infty)$). By construction, $v_1$ is strictly positive on $\mathscr{P}$.
 Beside $v_1$, let us denote  with $v_2$ a  solution of the problem
 \begin{equation}
 \label{eq:stat5}
 \begin{cases}
 \Ls v_2 = 1\,\,&\hbox{ in }\mathscr{P}\\
 v_2 = 0\,\,\,&\hbox{ in }\Rd\setminus \mathscr{P}.
 \end{cases}
 \end{equation}
 It is {necessarily positive, thanks to the maximum principle \cite[Theorem 1.8]{disoaval}}. Therefore, by linearity the (continuous) solutions of \eqref{eq:stat3} for which the positivity set $\mathscr{P}$ is connected have the form
\begin{equation}
\label{eq:general_sol}
v = v_1 + K v_2,\quad K\in \mathbb{R}.
\end{equation}
In fact, let $v_1$ denote a particular solution of the linear equation
\[
\begin{cases}
\nabla (\Ls v) = -\beta y \,\,&\hbox{ in }\mathscr{P}\\
v = 0\,\,\,\,&\hbox{ in } \Rd\setminus \mathscr{P}.
\end{cases}
\]
Then, all the solutions to the equation \eqref{eq:stat3} are given by \eqref{eq:general_sol} provided $v_2$ solves
the homogeneous problem
\[
\begin{cases}
\nabla (\Ls v) = 0\,\,&\hbox{ in }\mathscr{P}\\
v = 0\,\,\,\,&\hbox{ in } \Rd\setminus \mathscr{P},
\end{cases}
\]
which corresponds to \eqref{eq:stat5}.

We will relate the $v_1$ component of the solution
 \eqref{eq:general_sol} to an obstacle problem
for which we prove existence and uniqueness of a smooth ($C^{1,\alpha}$, $\alpha\in (0,s)$, according to the obstacle problem regularity theory), radially decreasing solution.
In such a way we construct a kind of minimal energy solution.
As a consequence, we will conclude that the positivity set $\mathscr{P}$ of $v$ is a ball.

Following the analysis, the $v_2$ component of $v$ in the decomposition \eqref{eq:general_sol} is a kind of correction of $v_1$. The function $v_2$ solves \eqref{eq:stat5} in a ball, and is explicitly obtained as a rescaling of the solution ({the subscript $G$ refers to Getoor})
\begin{equation}
\label{eq:getoor_sol}
v_G(y) =\frac{1}{\kappa_{s,d}}(1 - \vert y\vert^2)_{+}^{s}
\end{equation}
 given in \cite[Th. 5.2]{Getoor}, where $\kappa_{s,d}:=2^{2s}\Gamma(s+2)\Gamma(s+ \frac{d}{2})\Gamma(\frac{d}{2})^{-1}$ and the ball is $B_1(0)$.
Note that this solution is $C^{0,s}$. More importantly (see \eqref{eq:getbad} below, this component does not satisfy the regularity assumptions of the solutions introduced in Section \ref{sec:existence}, and will be disregarded in our study, though we do not claim that they cannot be useful in other contexts.

\subsection{\bf The obstacle problem}

\normalcolor

 We introduce the following energy
\begin{equation}
\label{eq:energy_1}
\mathscr{E}(v):= \frac{1}{2}\int_{\Rd}\vert \Lsm v\vert^2 \d y + \int_{\Rd}\Big(\frac{\beta}{2} \vert y\vert^2 -1 \Big)v\d y
\end{equation}
that we minimize in a set of nonnegative functions.
{We chose the constant $1$ in the energy just for notational simplicity.  In the next subsubsections and \ref{sss:explicit} \ref{sss:adj_mass} we will show that the minimizer of \eqref{eq:energy_1}, once properly rescaled, will solve the stationary equation \eqref{eq:stat4bis}}.
Nonnegative minimizers of \eqref{eq:energy_1} exist
and are unique. Indeed we have
\begin{theor}
\label{eq:stationary_obstacle}
Let $\beta>0$ be fixed. Then, there exists a unique solution $v$ of the following constrained minimization problem
\begin{equation}
\label{eq:ob_problem}
\min\left\{\mathscr{E}(v), \,\,v\in \mathscr{K}\right\}
\end{equation}
where
\begin{equation}
\label{eq:K}
\mathscr{K}:=\left\{v\in H^s(\Rd): \Big(\frac{\beta}{2}\vert y\vert^2 - 1\Big)v\in L^1(\Rd),\,v\ge 0 \right\}.
\end{equation}
\end{theor}
\begin{proof}

First of all, we prove that the energy is bounded below.
Let us fix $R:=\sqrt{2/\beta}$ and note that
$\int_{\Rd\setminus B_R}\Big(\frac{\beta}{2}\vert y\vert^2-1\Big)v\d y\ge 0$ for any $v\in \mathcal{K}$. Then,
\[
\frac{1}{2}\| v\|^2_{\dot{H^{s}}(\Rd)}
+ \int_{B_R}\Big(\frac{\beta}{2}\vert y\vert^2-1\Big) v \d y.
\le  \mathscr{E}(v)
\]
Therefore, the Hardy-Littlewood-Sobolev inequality (see \cite[Theorem V.1]{stein}) brings
two positive constants $C_1$ and $C_2$ depending on $d,s,\beta$, such that
\begin{equation}
\label{eq:bounded_below}
\mathscr{E}(v) \ge C_1\| v\|^2_{\dot{H^{s}}(\Rd)} -C_2,\,\,\,\forall v\in \mathcal{K}.
\end{equation}
Let $v_n\in \mathscr{K}$ be a minimizing sequence, that is $\mathscr{E}(v_n)\xrightarrow{n\to +\infty}\hbox{inf}_{w\in \mathscr{K}}\mathscr{E}(w)$. We can assume that $v_n$ belongs to
a sublevel of the energy for $n$ sufficiently large. Thus, there exists some $\bar n$ and some $C>0$ such that
\begin{equation}
\label{eq:sublevel}
\mathscr{E}(v_n) \le C,\,\,\,\forall n\ge \bar n.
\end{equation}
Therefore, thanks to \eqref{eq:bounded_below} we have
\begin{equation}
\label{eq:coerc2}
\| v_n\|_{\dot{H^{s}}(\Rd)} \le C \,\,\,\hbox{ and }\Big\vert \int_{\Rd}\Big(\frac{\beta}{2}\vert y\vert^2 -1\Big) v_n(y) \d y \Big\vert \le C,
\end{equation}
and since
$\dot{H^{s}}(\Rd)$ is compactly embedded in $L^p(K)$
for any compact in $\Rd$ and any $p<p_s:= \frac{2d}{d-2s}$, we have that, up to a subsequence,
\[
v_n \xrightarrow{n\to +\infty}v\,\,\,\,\,\hbox{ in } L^{p}(K),\,\,\,\forall p<p_s:= \frac{2d}{d-2s}.
\]
In particular, $v_n\xrightarrow{n\to +\infty} v$ almost everywhere in $\Rd$. This guarantees that $v\ge 0$. To show that
$\Big(\frac{\beta}{2}\vert y\vert^2 - 1\Big)v\in L^1(\Rd)$, we use Fatou's Lemma.
We let $R = 2/\sqrt{\beta}$ and use the decomposition
\[
\int_{\RN}(\frac{\beta}{2}\vert y\vert^2-1) v_n \d y = \int_{\Rd\setminus B_R}
\Big(\frac{\beta}{2}\vert y\vert^2-1\Big) v_n \d y + \int_{B_R}\Big(\frac{\beta}{2}\vert y\vert^2-1\Big) v_n \d y.
\]
Thanks to the strong convergence in $L^p$ on compact sets
we have that
\[
\lim_{n\to\infty}\int_{B_R}\Big(\frac{\beta}{2}\vert y\vert^2-1\Big) v_n \d y  = \int_{B_R}(\frac{\beta}{2}\vert y\vert^2-1) v\, \d y,
\]
and thus $(\frac{\beta}{2}\vert y\vert^2-1) v\in L^1(B_R)$.
For the first integral, we have thanks to Fatou's Lemma (recall that
in $\Rd\setminus B_R$ we have that $\Big(\frac{\beta}{2}\vert y\vert^2-1\Big)v\ge 0$)
\begin{align}
\label{eq:semicont2}
+\infty>\inf_{w\in \mathscr{K}}\mathcal{E}(w)&\ge \liminf_{n\to +\infty}\int_{\Rd}\Big(\frac{\beta}{2}\vert y\vert^2-1\Big) v_n \d y\nonumber\\
& =
\liminf_{n\to +\infty}\Big(\int_{\Rd\setminus B_R}
\Big(\frac{\beta}{2}\vert y\vert^2-1\Big) v_n \d y\Big)
 + \int_{B_R}\Big(\frac{\beta}{2}\vert y\vert^2-1\Big) v \d y\nonumber\\
& \ge \int_{\Rd\setminus B_R}\Big(\frac{\beta}{2}\vert y\vert^2-1\Big) v \d y + \int_{B_R}\Big(\frac{\beta}{2}\vert y\vert^2-1\Big) v \d y.
\end{align}
Thus, we conclude that $(\frac{\beta}{2}\vert y\vert^2-1) v\in L^1(\Rd\setminus B_R)$ and therefore $v\in \mathscr{K}$.

Now, to conclude that $v$ is indeed a minimum for the energy, we observe that the semicontinuity  of the $H^s$ seminorm with respect to the weak convergence and \eqref{eq:semicont2} imply that
 \[
\hbox{inf}_{z\in \mathscr{K}}\mathscr{E}(z) = \liminf_{n\to \infty}\mathscr{E}(v_n) \ge \mathscr{E}(v),
 \]
 namely the minimality of $v$. The uniqueness follows from the strict convexity of the energy.
  \normalcolor
\end{proof}

In the next Theorem we prove some important properties of the solution of the obstacle problem \eqref{eq:ob_problem}.
To this purpose, we prepare the following
\begin{lemma}
\label{lemma:rearrangement}
For any $v\in \mathscr{K}$ there holds
\begin{equation}
\label{eq:energy_dec}
\mathscr{E}(v^*) \le \mathscr{E}(v),
\end{equation}
where $v^*$ is the symmetric decreasing rearrangement of $v$.
\end{lemma}
\begin{proof}
First of all, we observe that the symmetric decreasing rearrangement
reduces the Gagliardo seminorm (see \cite{almglieb}) and thus
\[
\int_{\Rd}\vert \Lsm v^*\vert^2 \d y \le \int_{\Rd}\vert \Lsm v\vert^2 \d y.
\]
Hence, to obtain \eqref{eq:energy_dec} we have to prove
that
\[
\int_{\Rd}\Big(\frac{\beta}{2}\vert y\vert^2 -1\Big) v^* \d y \le \int_{\Rd}\Big(\frac{\beta}{2}\vert y\vert^2 -1\Big) v \d y,
\]
which amounts to prove, since $\int_{\Rd}v \d y = \int_{\Rd}v^*\d y$, that
\begin{equation}
\label{eq:dec_mom}
\int_{\Rd} \vert y\vert^2 v^* \d y \le \int_{\Rd} \vert y\vert^2 v \d y.
\end{equation}
As a first step, we prove \eqref{eq:dec_mom} for compactly supported $v$. Thus, we let $B_R$ be a ball that
contains the support of $v$ and we set $g(y):=(R^2 - \vert y\vert^2)_+$. Since we have that
$g^* = g$, Theorem 3.4 in \cite{lieb_loss} gives
\begin{equation}
\int_{\Rd}(R^2 - \vert y\vert^2)v\d y =\int_{\Rd} g v  \d y \le \int_{\Rd}g v^* \d y = \int_{\Rd}(R^2 - \vert y\vert^2) v^* \d y,
\end{equation}
from which we have, being $\int_{\Rd}v \d y = \int_{\Rd}v^*\d y$,
\[
\int_{\Rd} \vert y\vert^2 v^* \d y \le \int_{\Rd} \vert y\vert^2 v \d y.
\]
The general case follows by approximation. In fact, for any $k>0$,
we consider $v_k:=\chi_{B_k(0)}v$ (with $\chi_{B_k(0)}$ we indicate
the characteristic function of $B_k(0)$). The validity of \eqref{eq:dec_mom}
for compactly supported functions gives
\[
\int_{\Rd}\vert y\vert^2 v_k^{*}\d y \le \int_{\Rd}\vert y\vert^2 v_k\d y.
\]
Moreover, since $\vert y\vert^2 v_k \le \vert y\vert^2 v$, the dominated convergence Theorem
gives that the right hand side converges to $\int_{\Rd}\vert y\vert^2 v \d y$. Hence, thanks to the Fatou's Lemma,
we get
\[
\int_{\Rd}\vert y\vert^2 v^{*}\d y \le \int_{\Rd}\vert y\vert^2 v\d y.
\]
\end{proof}
\begin{theor}
\label{th:properties}
For any $\beta>0$, the solution $v_O$ of the obstacle problem \eqref{eq:ob_problem} is smooth, namely $C^{1,\alpha}(\Rd)$ ($\alpha\in (0,s)$),
radial decreasing and with compact support. We denote with $M_1=M_1(d,s,\beta)$ its mass, namely
\begin{equation}
\label{eq:M1}
M_1:=\int_{\Rd}v_O(y)\d y.
\end{equation}
Moreover, $v_O$ satisfies the following Euler-Lagrange equations
\begin{equation}
\label{eq:compatibility}
\begin{cases}
v_O\ge 0, \,\,\,\,\hbox{ and } \,\Ls v_O \ge 1- \frac{\beta}{2}\vert y\vert^2\,\,\,&\hbox{ a.e. in } \Rd,\\
v_O = 0 \,\,\,\,&\hbox{ in } I,\\
v_O>0, \,\,\,\hbox{ and }\,\Ls v_O = 1- \frac\beta 2\vert y\vert^2\,\,\,\,&\hbox{ a.e. in }\Omega,
\end{cases}
\end{equation}
where we have denoted with $I$ the coincidence set $\left\{ y\in \Rd: v_O=0\right\}$ and with $\Omega$ its complement.
By radial symmetry, we have that $\Omega=B_R(0)$ for some $R>0$ and $I=\Rd\setminus B_R(0)$.
\end{theor}
\begin{proof}
The regularity of $v_O$ follows from \cite{SilvestrePhD} and the derivation of the Euler-Lagrange equations \eqref{eq:compatibility}
is standard. The validity of \eqref{eq:compatibility} implies that
\begin{equation}
\label{eq:suppo1}
B_{R_\beta}(0)\subseteq \Omega, \,\,\,\,\,\,R_\beta:=\sqrt{\frac{2}{\beta}}.
\end{equation}

In fact, suppose that there is a point $\bar y$ in $B_{R_\beta}(0)$ such that
$v_O(\bar y)=0$. Then, $\bar y$ is a global
minimum point for $v_O$ and thus $\Ls v_O(\bar y) \le 0$. But $\vert \bar y\vert <\sqrt{2/\beta}$
and thus the first of \eqref{eq:compatibility} gives $\Ls v_O(\bar y) >0$, absurd.

Next, we show that $\Omega$ cannot be the whole space.
{ In fact, thanks to
\cite[Theorem 1.2]{fall}, solving
\[
\Ls v = P\,\,\,\hbox{ in } \Rd
\]
with $P$ a polynomial, forces $u$ to be affine and $P$
to be equal to zero, which is clearly not the case in our situation.
}

 To conclude that $v_O$ is radially symmetric and has compact support, we argue as follows.
We denote with $v^{*}:\Rd\to \R$ the symmetric decreasing rearrangement of $v_O$. Lemma \ref{lemma:rearrangement}
gives that
\[
\mathscr{E}(v^*) \le \mathscr{E}(v_O),
\]
and thus $v^*$ is a competitor for $v_O$. The uniqueness of $v_O$ entails that $v\equiv v^*$ and thus that $v_O$ is radial symmetric
and decreasing. Hence, since $v_O$ can not be strictly positive everywhere in $\Rd$,
$\Omega$ should be a ball $B_R(0)$ with  $R>\sqrt{2/\beta}$.
\end{proof}

\subsubsection{\bf Explicit form.}\label{sss:explicit}
 The solution $v_O$  of the obstacle problem \eqref{eq:compatibility} can be explicitly computed. We set
\begin{equation}
\label{eq:dyda1}
{\tilde{v}}_{1}(y):=\frac{1}{\kappa_{s,d}}(1 - \vert y\vert^2)^{1+s}_{+}\,\,\,\,\,\hbox{ for } y\in \mathbb{R}^d,
\end{equation}
where $\kappa_{s,d}:=2^{2s}\Gamma(s+2)\Gamma(s+ \frac{d}{2})\Gamma(\frac{d}{2})^{-1}$, with $\Gamma(\cdot)$ being the Euler $\Gamma$-function. This function is supported in the ball of radius 1. According to the formulas derived by Dyda
in \cite{dyda}  we have
\begin{equation}
\label{eq:dyda2}
\begin{cases}
\Ls {\tilde{v}}_{1} (y)= 1 - \gamma_{s,d}\vert y\vert^2=:f(y)\,\,\,\,&\hbox{ in }B_{1}(0)\\
v_1(y) = 0 \,\,\,\,&\hbox{ in } \Rd\setminus B_1(0),
\end{cases}
\end{equation}
with $\gamma_{s,d}:=1+ \frac{2s}{d}>1$.
Notice that $\Ls {\tilde{v}}_{1} (y)$ is positive for small $|y|$ but negative for $|y|\sim 1$. Next, we need to change the constant $\gamma_{s,d}$ into $\beta/2$ in the last formula, and this is done by rescaling as follows: we introduce a parameter $\lambda>0$ and set
 \begin{equation}
 \label{eq:explicit_obstacle}
 v_D(y):= \frac{1}{\lambda^{2s}}{\tilde{v}}_1(\lambda y) = \frac{1}{\lambda^{2s}\kappa_{s,d}}(1- \lambda^2 \vert y\vert^2)_{+}^{1+s},
 \end{equation}
Fixing the value $\lambda:=\sqrt{\beta/(2\gamma_{s,d})}\,,$ and setting $R_D=1/\lambda$ we observe that:  \\

 (i) $v_D$ is supported in the ball $B_{R_D}(0)$ with
\begin{equation}
R_D:=\sqrt{\frac{2\gamma_{s,d}}{\beta}}=\left(2(1+\frac{2s}{d})(d+2(1+s))\right)^{1/2},
\end{equation}

(ii) we have the regularity ${\tilde{v}}_1\in C^{1,s}(\Rd)$, and

(iii) for every $y\in B_{R_D}(0)$ we have
\[
\Ls v_D (y)=\frac{1}{\lambda^{2s}}(\Ls {\tilde{v}}_1)(\lambda y)\lambda^{2s} = f(\lambda \,y) = 1- \frac{\beta}{2}\vert y\vert^2.
\]
Thus,  $v_D$ solves the problem
\begin{equation}
\label{eq:obstacle_dyda}
\begin{cases}
\Ls v_D (y) = 1- \frac{\beta}{2}\vert y\vert^2\,\,\,\,\,&\hbox{ in } B_{R_D}(0)\\
v_D = 0\,\,\,\,&\hbox{ in } \Rd\setminus B_{R_D}(0).
\end{cases}
\end{equation}
{Finally, we have the following important lemma establishing the link between this solution and the obstacle problem \eqref{eq:ob_problem}.}

\begin{lemma}
\label{lemma:vd=vo}
 The solution $v_O$  of the obstacle problem coincides with $v_D$, $v_D=v_O$.
\end{lemma}
\begin{proof}
We first prove that the supports of $v_D$ and of $v_O$ coincide and then that equality of solutions, $v_D\equiv v_O$.

(i) to deal with the supports we
 argue by contradiction. Assume that
their supports are different. This means that we may suppose that
$R>R_D$ (the opposite situation can be treated with the very same argument). We let
$$
\tilde{v}(y):= \Big(\frac{R_D}{R}\Big)^{-2s}v_{D}\Big(\frac{R_D}{R}y\Big)
 $$
 in such a way that $\tilde{v}$ is the unique solution of
\begin{equation}
\label{eq:tildev}
\begin{cases}
\Ls \tilde v = 1 - \frac{\beta}{2}\frac{R_D^{2}}{R^2}\vert y\vert^2\,\,\,&\hbox{ in } B_{R}(0),\\
\tilde v = 0 \,\,\,&\hbox{ in }\Rd\setminus B_{R}(0).
\end{cases}
\end{equation}
We set $w:=\tilde{v}-v_O$ and we note that $w$ is a $C^{1,s}(\Rd)$ solution of
\begin{equation}
\label{eq:wabsurd}
\begin{cases}
\Ls w = \frac{\beta}{2}\Big(1-\frac{R_D^{2}}{R^2}\Big)\vert y\vert^2\,\,\,&\hbox{ in } B_R(0)\\
w = 0 \,\,\,&\hbox{ in }\Rd\setminus B_R(0).
\end{cases}
\end{equation}
Therefore, since $\frac{\beta}{2}\Big(1-\frac{R_D^{2}}{R^2}\Big)\vert y\vert^2>0$ in $B_R(0)$, the fractional version of the Hopf Lemma
implies that either $w$ vanishes in $B_R(0)$, and thus $R=R_D$, or there
exists some $\delta_0>0$ such that for any $x\in \partial B_R(0)$
\[
\liminf_{B_r\ni z\to x}\frac{w(z)}{\hbox{dist}^s(z,\partial B_r)}\ge \delta_0, \,\,\,\,B_r
\,\,\,\hbox{ is an interior ball with radius } r\,\,\hbox{ at }x.
\]
but this is impossible due to the $C^{1,s}$ regularity of $w$.

\medskip

(ii) Now, we call $v:=v_O-v_D$ and we observe that it solves
\[
\begin{cases}
\Ls v = 0 \,\,\,\,&\hbox{ in } B_R(0)\\
v = 0\,\,\,\,&\hbox{ in } \Rd\setminus B_R(0).
\end{cases}
\]
Thus, $v=0$ in $\Rd$ and therefore $v_D = v_O$.

\end{proof}


\subsubsection{\bf Adjusting mass and constant.}
\label{sss:adj_mass} The fact that $v_D=v_O$ permits to construct a  solution $v_C$ of \eqref{eq:stat4bis} for any parameter $C>0$ simply by rescaling
the solution $v_O$ of the obstacle problem \eqref{eq:ob_problem}. The free constant $C>0$ allows to fix at will either the radius of the support or the mass of the self-similar solution.
More precisely, we have the following result.
\begin{prop}
\label{lemma:rescaling_obstacle}
For any $C>0$ there exists a unique solution  $v=v_C\in C^{1,\alpha}(\Rd)$, $\alpha\in (0,s)$
 of the obstacle problem
 \begin{equation}
\label{eq:compatibility_C}
\begin{cases}
v_C\ge 0, \,\,\,\,\hbox{ and } \,\Ls v_C \ge C- \frac{\beta}{2}\vert y\vert^2\,\,\,\hbox{ a.e. in } \Rd,\\
v_C = 0 \,\,\,\,\hbox{ in } I_C,\\
v_C>0, \,\,\,\hbox{ and }\,\Ls v_C = C- \frac\beta 2\vert y\vert^2\,\,\,\,\hbox{ a.e. in }B_{R_C}(0),
\end{cases}
\end{equation}
which is supported in the ball of radius
\begin{equation}
R_C= C^{1/2}R_D\,,
\end{equation}
and has mass
 \begin{equation}
 \label{eq:mass_law}
 M_C:=\int_{\Rd}v_C(y)\d y =M_1 C^{1+d/(2 +s)}.
 \end{equation}
\end{prop}

\begin{proof}
We let $v=v_O$ be the solution of the obstacle problem
\eqref{eq:ob_problem}--\eqref{eq:K}. We set $w_O:=\Ls v_O$. Then,
 for any $C>0$, we define
 \begin{equation}
 \label{eq:vc}
 v_C(y):=C^{1+s}v(C^{-1/2}y),
 \end{equation}
so that
 \begin{equation}
 \label{eq:wc}
 w_C(y):= C w(C^{-1/2}y)
 \end{equation}
 satisfies $w_C(y) = C(\Ls v)(C^{-1/2}y)$. Consequently, since $v$ and $w$
 solve the compatibility equations \eqref{eq:compatibility},
 we get that the couple $v_C$ and $w_C$ solves \eqref{eq:compatibility_C}.
 Moreover, a simple computation gives (recall \eqref{eq:explicit_obstacle})
\begin{equation}
\label{eq:mass_C}
M=\int_{\Rd}v_C \d y = M_1C^{1+d/2 +s}.
\end{equation}
\end{proof}

\noindent {\sc Constant and verification.} The value for the constant that is called $C_2$ in formula \eqref{eq:ss_intro} of the introduction is given by
\begin{equation}
\label{eq:C2value}
C_2^{1+s}=\frac{\lambda^2}{\kappa_{s,d}}=\frac{d}{2(d+2+2s)(d+2s)\kappa_{s,d}}.
\end{equation}
This constant coincides with known values for the limit cases $s=0$ and $s=1$ agree.
More precisely, for $s=0$ we get the self-similar solution of Barenblatt type for the PME with value  $C_2=1/6$ in 1D and $C_2=1/2(d+2)$ in higher dimensions. For $s=1$ (Thin Film) it is known that
$C_2^2=1/120$ in 1D,  while we get $C_2^2=1/8(d+2)(d+4)$ for $d\ge 1$.


\subsubsection{\bf Self-similar weak solutions with
a connected positivity set.}
\label{sss:}
Now, we address the question of the existence of
self-similar weak solutions to \eqref{eq:prob_intro1} with a connected positivity set.
As we will see, this is a regularity question about the solutions \eqref{eq:general_sol}. More precisely, we remark that we look for weak solutions in the sense of Definition
\ref{def:weak}. This means that $u$ belongs to $H^{1+s}(\Rd)$ for a.e. $t\in (0,+\infty)$.
The same regularity holds also for the self-similar profile $v$ given by \eqref{eq:source_type_form}. Therefore, the arbitrary constant $K$ in the decomposition formula \eqref{eq:general_sol} must vanish.
In fact, as we have already observed, the $v_2$ component of
the general solution \eqref{eq:general_sol} is indeed a rescaled
version of the Getoor solution $v_G(y) = \kappa_{s,d}^{-1}(1-\vert y\vert^2)_{+}^s$, $y\in \mathbb{R}^d$ and we have
\begin{equation}
\label{eq:getbad}
v_G \notin H^{1+s}(\Rd).
\end{equation}
{
It is interesting to observe that with minor modifications one can also prove that in general that $v_2\notin H^{1+s}(\Rd)$, where we recall $v_2$ solves
\[
\begin{cases}
\Ls v_2 = 1\,\,\,\,\hbox{ in } \mathscr{P}\\
v_2 = 0 \,\,\,\,\,\hbox{ in } \Rd \setminus \mathscr{P},
\end{cases}
\]
when $\mathscr{P}$ is smooth, bounded and satisfying the internal ball condition.
}

To prove \eqref{eq:getbad} we can reason as follows.
On the one hand, we observe that if $s\in (0,1/2]$ then
$\nabla v_G$ is neither in $L^2(\Rd)$. In fact, we have\normalcolor
\[
\nabla v_G(y) =
\begin{cases}
\kappa_{s,d}^{-1}2y(1-\vert y\vert^2)^{s-1}, \,\,\,\vert y\vert \le 1\\
0\,\,\,\hbox{ otherwise in } \Rd.
\end{cases}
\]
Therefore,
\[
\int_{\Rd}\vert\nabla v_G\vert^2 \d y = 4\kappa_{s,d}^{-2}\int_{B_{1}(0)}\vert y\vert^2\big( 1-\vert y\vert^2\big)^{2s-2}\d y = +\infty, \,\,\,\,\hbox{ if } s\le 1/2.
\]
On the other hand when $s\in (1/2,1)$,
we observe that $v_G(y):=g(\vert y\vert)$ with $g(t) = \kappa_{s,d}^{-1}(1-t^2)^s_{+}$.
The function $g$ does not belong to $H^{1+s}(\mathbb{R})$. In fact, if $g\in  H^{1+s}(\mathbb{R})$ then
we would have that $g'\in H^s(\mathbb{R})\subset C^0(\mathbb{R})$, thanks to Sobolev
embeddings. This is impossible since $g' \to -\infty$ for $t\to 1^{-}$.
Now, since (see e.g. \cite{grafakos})
\[
\widehat{v}_G(\xi) = \widehat{g}(\vert \xi\vert), \,\,\,\,\xi\in \Rd,
\]
if $v_G\in H^{1+s}(\Rd)$, then we would have
\begin{align}
+\infty > \int_{\Rd\cap \left\{ \vert \xi\vert>1\right\}}\vert \xi\vert^{2s+2}\widehat{v}_G^2(\xi)\d \xi=
\omega_d\int_{1}^{+\infty} \rho^{2s +1 + d}g^{2}(\rho) \d \rho\ge \omega_d\int_{1}^{+\infty} \rho^{2s +2}g^{2}(\rho) \d \rho,
\end{align}
where $\omega_d$ is the measure of the unitary sphere in $\Rd$, $\d\ge 2$.
Thus, since we already know that $g\in L^2(\Rd)$, the last inequality would imply that $g\in H^{1+s}(\Rd)$, absurd.

Therefore, we must take $K=0$ in \eqref{eq:general_sol} and thus the self-similar solutions complying
with the regularity prescribed by Theorem \ref{th:existence} are rescaled version of the model
solution $v_D$ in \eqref{eq:explicit_obstacle}.
In particular, (see Lemma \ref{lemma:rescaling_obstacle}) the constant $C$ is fixed according
to the mass law \eqref{eq:mass_law}. Moreover, the fact that $v_D$ is indeed the solution of an obstacle problem reflects in a kind of minimality, with respect to the energy \eqref{eq:energy_1}, of the self-similar solution.

 The following Theorem clarifies the situation.
 \begin{theor}
 \label{th:stationary}
 Given $M>0$, we let $v_C$ be the solution of
 \eqref{lemma:rescaling_obstacle} with the constant $C$ complying with the mass law \eqref{eq:mass_law}. Then, for
 \[
 \alpha = \frac{d}{d  + 2(1+ s)}, \,\,\,\,\beta = \frac{1}{d+ 2(1+ s)},
 \]
 the self-similar function
 \begin{equation}
 \label{eq:self_similar_th}
 u_C(x,t):=\frac{1}{(1+t)^\alpha}v_C\Big(\frac{x}{(1+t)^\beta}\Big),
 \end{equation}
 is a weak solution (in the sense of Definition \ref{def:weak}) of
 \eqref{eq:prob_intro1} with mass $M>0$, and it satisfies
  \begin{equation}
 \label{eq:initial_delta_th}
 \lim_{t\to -1^+}u_C(x,t) = M \delta(x) \,\,\,\,\hbox{ in } \mathscr{D}'(\Rd).
 \end{equation}
 \end{theor}
\begin{proof}

 Starting from $v_C$, we define, for (see \eqref{eq:alpha/beta})
 $\alpha= \frac{d}{d  + 2(1+ s)}, \,\,\,\,\beta = \frac{1}{d+ 2(1+ s)}$,
 \begin{equation}
 \label{eq:self_similar}
 u_C(x,t) := \frac{1}{(1+t)^\alpha}v_C\Big(\frac{x}{(1+t)^\beta}\Big), \,\,\,\,(x,t)\in \Rd\times (0,+\infty),
 \end{equation}
 and we obtain, by a direct computation that it is a distributional (self-similar) solution of \eqref{eq:prob_intro1} such that
  \begin{equation}
 \label{eq:initial_delta}
 \lim_{t\to -1^+}u_C(x,t) = M \delta(x) \,\,\,\,\hbox{ in } \mathscr{D}'(\Rd).
 \end{equation}

 \end{proof}

\section{\bf Long time analysis}
\label{sec:long_time}
In this Section we address the long time behavior of the weak solutions constructed in Theorem \ref{th:existence}.
{Our first result on the long-time behavior is Theorem \ref{th:long_time1} in which we prove that the set of cluster points for $\tau\to +\infty$ (that is, the $\omega$-limit set defined in \eqref{eq:omega_lim} below) of the weak solution to the Fokker-Planck equation \eqref{eq:FP} (see below for the definition) is not empty and that its elements are indeed weak stationary solutions of \eqref{eq:FP}. This proof needs an extra assumption on the regularity of the class of weak solutions, see  \eqref{eq:extra_assumption}, that seems technical to us. Let us briefly explain the problem: Unfortunately, the basic energy estimate (see \eqref{eq:energy_est}) available for the weak solutions of \eqref{eq:prob_intro1} does not rescale directly to an analogous energy estimate (see \eqref{eq:energy_estv}) for the weak solutions of the Fokker-Planck equation. This estimate
should contain, as a formal computation reveals, both the second moment
and the fractional energy $\frac{1}{2}\int_{\Rd}\vert \Lsm u\vert^2$ and
 appears to be given by a proper balance between these two terms.
 We obtain \eqref{eq:energy_estv} by rescaling an improved
 energy estimate for the weak solutions of \eqref{eq:prob_intro1} that contains
 both a fractional energy $\frac{1}{2}\int_{\Rd}\vert \Lsm u\vert^2$ and  the second moment (see Proposition \ref{prop:entropy_dissipation_v}). These two terms are properly weighted by precise time dependent factors related to the rescaling \eqref{eq:self_similar_change_intro}.
One of the main ingredient in the proof of this estimate is the following equality (see Lemma \ref{lem:fourier})
\begin{equation}
\label{eq:fourier_nice_intro}
\frac{1}{2}\int_{\Rd}\vert \Lsm u\vert^2 \d x = -\frac{1}{d-2s}\int_{\Rd}p (x\cdot \nabla u) \d x.
\end{equation}
This identity resembles a Pohozaev identity and furnishes the exact balance between the second moment and the fractional energy. As the proof of Lemma
\ref{lem:fourier} reveals it holds for functions with some decay at infinity
in order to guarantee that the right hand side makes sense.

Therefore, this is the class for which we address the long time behaviour.
It is important to observe that weak solutions with compact support
actually satisfy \eqref{eq:fourier_nice_intro}. It is an open problem
to prove that \eqref{eq:fourier_nice_intro} holds for all weak solutions.

Our second result is Theorem \ref{th:long_time} in which we are interested in relating the long-time dynamics of \eqref{eq:prob_intro1}
with the self-similar solutions constructed in Section \ref{sec:self_similar}. At this stage our analysis needs some connectedness
assumption on the elements of the $\omega$-limit set of a weak solution $v$. This assumption permits to conclude that the only stationary solution that attracts the dynamics for large times is the compactly supported self-similar solution $v_C$
constructed in Theorem \ref{th:stationary}
 with the constant $C$ adjusted to match the mass constraint. As a result, we will obtain the long-time asymptotics
\[
u(\cdot,t)-v_C(\cdot,t)\xrightarrow{t\to +\infty}0\,\,\,\hbox{ in }L^1(\Rd).
\]
}
%

As a starting point we prepare the following technical Lemma
\begin{lemma}
\label{lemma:entropy}
Given $f\in L^1_{loc}(\Rd)$ with $f\ge 0$ and
\begin{equation}
\label{eq:hypoentropy}
\int_{\Rd}(1+\vert x\vert^2) f(x) \d x \le C_1,\,\,\,\,\,\,\int_{\Rd}\vert f\vert^{p^*} \d x \le C_2(p^*), \,\,\,\,p^*>1,
\end{equation}
there holds
\begin{equation}
\label{eq:entropy}
\Big\vert \int_{\Rd}f(x)\log f(x)\d x\Big\vert\le C(p^*).
\end{equation}
\end{lemma}
\begin{proof}
We split
\begin{align}
\label{eq:entropy1}
\Big\vert \int_{\Rd}f(x)\log f(x)\d x\Big\vert &\le \Big\vert \int_{\Rd\cap\left\{f\ge 1 \right\}} f(x)\log f(x) \d x\Big\vert +
\Big\vert \int_{\Rd\cap\left\{0\le f\le 1 \right\}} f(x)\log f(x) \d x\Big\vert\nonumber\\
& =: A + B.
\end{align}
Note that \eqref{eq:hypoentropy} and interpolation imply that $f$ is actually controlled
in $L^p$ for any $p\in [1,p^*]$. Thus, since $t\log t\le t^{1+\eps}$ on $[1,+\infty)$
 for some $\eps>0$, we conclude that
\begin{equation}
\label{eq:A}
A \le \int_{\Rd\cap \left\{f\ge 1\right\}}f(x)^{1 +\eps} \d x \le C_2(p^*).
\end{equation}
To control $B$,
we first note that
\[
B \le \int_{\left\{0\le f\le 1\right\}}f(x) \log\Big(\frac{1}{f(x)} \Big)\d x.
\]
To control the integral in the right hand side, we split it in two parts (see, e.g.,
 \cite{dipe_lions}). We
have
\begin{align}
B \le \int_{\left\{0\le f\le 1\right\}\cap \left\{f(x) \ge e^{-\vert x\vert^2} \right\}} f(x) \log\Big(\frac{1}{f(x)} \Big)\d x
+ \int_{\left\{0\le f\le 1\right\}\cap \left\{f(x) \le e^{-\vert x\vert^2} \right\}}f(x) \log\Big(\frac{1}{f(x)} \Big)\d x.
\end{align}
Now, since $t\to -\log t$ is decreasing we have that, on the set where $f(x) \ge e^{-\vert x\vert^2}$,
$\log\Big(\frac{1}{f(x)}\Big) \le \vert x\vert^2$. Thus, the first integral is bounded as
\[
\int_{\left\{0\le f\le 1\right\}\cap \left\{f(x) \ge e^{-\vert x\vert^2} \right\}} f(x) \log\Big(\frac{1}{f(x)} \Big)\d x \le \int_{\Rd}\vert x\vert^2f(x)\d x\le C_1.
\]
As regards the second integral, we use the fact that $t\log(1/t)\le C_3\sqrt{t}$ ($C_3>0$)
when $t\in (0,1)$. Thus,
on the set $\left\{0\le f\le 1\right\}\cap \left\{f(x) \le e^{-\vert x\vert^2} \right\}$ we have
\[
f(x) \log\Big(\frac{1}{f(x)}\Big) \le \sqrt{f(x)}\le C_3 e^{-\frac{1}{2}\vert x\vert^2}.
\]
Thus, the second integral is bounded by
\[
\int_{\left\{0\le f\le 1\right\}\cap \left\{f(x) \le e^{-\vert x\vert^2} \right\}}f(x) \log\Big(\frac{1}{f(x)} \Big)\d x \le C_3 \int_{\Rd}
e^{-\frac{1}{2}\vert x\vert^2}\d x\le C_3\sqrt{(2\pi)^d}.
\]
Collecting all the estimates we have the thesis.
\end{proof}

\subsection{Moments and Refined Energy Estimate}
\label{ss:mo_ene}
In this Subsection we show that if the initial condition $u_0$
has finite second moments, then the weak solutions starting
from $u_0$ maintain their second moments finite.
This fact combined with the energy estimate \eqref{eq:energy_est}
gives a refined energy estimate that turns to be fundamental for the long time behaviour analysis.
The estimate on the second moments and the refined energy estimate
hold for those weak solutions constructed in Theorem \ref{th:existence} that verify that for a.a. $t\in (0,+\infty)$
\begin{align}
\label{eq:extra_assumption}
 (x\cdot\nabla u) \in L^2(\Rd)
\end{align}
Under this condition we prove (see Lemma \ref{lem:fourier} below)
 that
\begin{equation}
\label{eq:fourier_nice}
\frac{d-2s}{2}\int_{\Rd}\vert \Lsm u\vert^2 \d x = -\int_{\Rd}p\, (x\cdot\nabla u)\d x.
\end{equation}
The condition \eqref{eq:extra_assumption} is not optimal in terms of
regularity and serves to guarantee that the right hand side of
\eqref{eq:fourier_nice} makes sense. In particular, what seems to be
needed for the proof is a good decay at infinity for the solutions.
It is interesting to observe that \eqref{eq:fourier_nice} holds, without invoking \eqref{eq:extra_assumption}, for weak solutions with compact support.
It is an interesting open problem to verify its validity
for all the weak solutions given by Theorem \ref{th:existence}.

\medskip

\noindent {\sl\ref{ss:mo_ene}.1.}

Control of the second moment. We state such control for a special class of weak solutions satisfying the just stated assumptions.

To ease the presentation and to convey the main ideas, we work at first with
smooth solutions with a good decay at infinity.
Thus, we let $u$ be a smooth solution of
\[
\begin{cases}
\partial_t u = \div(u\nabla p)\,\,\,\,\hbox{ in } \Rd\times (0,+\infty),\\
p = \Ls u, \,\,\,\,\hbox{ in } \Rd\times (0,+\infty).
\end{cases}
\]
We have
\begin{align*}
\frac{d}{dt}\frac{1}{2}\int_{\Rd}\vert x\vert^2 u\d x
& = -\int_{\Rd}u(x\cdot \nabla p)\d x\nonumber\\
& = \int_{\Rd} p (x\cdot \nabla u)\d x + d\int_{\Rd}up\d x.
\end{align*}
Thus, thanks to Lemma \ref{lem:fourier} below we get
\begin{equation}
\label{eq:smooth_second}
\frac{d}{dt}\frac{1}{2}\int_{\Rd}\vert x\vert^2 u\d x = \frac{d+2s}{2}\int_{\Rd}\vert \Lsm u\vert^2 \d x.
\end{equation}
This identity is interesting since similar relations are available (\cite{TO_private}) for other evolutions of gradient flow type such as the Porous medium equation, both the classical one and both the fractional one (see, for this last case equation \ref{eq:sec_mom_fract_pm} below).
The computation above is of course only formal since we can not use
$\vert x\vert^2$ as a test function in the definition of weak solution.
However, thanks to the estimate on the first moments, we have the following
\begin{lemma}[Second Moments]
\label{lem:sec_mom}
Let $u$ be a weak solution such that \eqref{eq:extra_assumption} holds, then
\begin{equation}
\label{eq:second_moment}
\int_{\Rd}\frac{\vert x\vert^2}{2}u(x,t) \d x \le \int_{\Rd}\frac{\vert x\vert^2}{2}u_0(x)\d x + \frac{d+2s}{2}\int_{\Rd\times (0,t)}\vert \Lsm u\vert^2\d x\d r.
\end{equation}
\end{lemma}

\begin{proof}[Proof of Lemma \ref{lem:sec_mom}] To prove \eqref{eq:second_moment} we first use the  control the first moments of $u$, proved in Lemma \ref{lem:first_mom}.
For any $R>0$ we consider the real function $g_R:[0,+\infty)\to \R$ such that
\begin{equation}
\label{eq:gr}
g_R(t) =
\begin{cases}
\frac{t^2}{2}\,\,\,\,\,\,\,t\le R\\
Rt - \frac{R^2}{2}\,\,\,\,\,\,\,t\ge R.
\end{cases}
\end{equation}
Then we define $\vf(x) :=g_R(\vert x\vert)$ and observe that
$\vf$, thanks to the first moment estimate, can be used as a test function in the weak formulation \eqref{eq:weak}.

We bound, uniformly in $R$ the integrals in the right hand side of the weak formulation.
We have that
\begin{align*}
\int_{0}^{t}\int_{\Rd}p u \Delta \vf \,\d x \d r =\int_{0}^{t}\int_{\Rd}p u g''_R(\vert x\vert)\d x\d r + (d-1)\int_{0}^{t}\int_{\Rd}p u \frac{g'_{R}(\vert x\vert)}{\vert x\vert }\d x\d r,
\end{align*}
thus being $\frac{g'_{R}(\vert x\vert)}{\vert x\vert } \xrightarrow{R\to +\infty} 1$ and $g'_R(\vert x\vert)\le 1$, we have by dominated convergence that
\[
\lim_{R\to +\infty}\int_{0}^{t}\int_{\Rd}p u \Delta \vf \,\d x \d r = d\int_{0}^{t}\int_{\Rd}p u\,\d x \d r = d\int_{0}^{t}\int_{\Rd} \vert \Lsm u \vert^2\d x\d r.
\]
Now we come to the term
\[
\int_{0}^{t}\int_{\Rd}p(\nabla u\cdot \nabla \vf)\,\d x \d r =
\int_{0}^{t}\int_{\Rd}p(\nabla u \cdot x) \frac{g'_{R}(\vert x\vert)}{\vert x\vert }\d x\d r.
\]
Now, \eqref{eq:extra_assumption} gives
that
\[
\lim_{R\to \infty} \int_{0}^{t}\int_{\Rd}p(\nabla u \cdot x) \frac{g'_{R}(\vert x\vert)}{\vert x\vert }\d x\d r = \int_{0}^t\int_{\Rd}p(\nabla u \cdot x)\d x \d r =-\frac{d-2s}{2}\int_{0}^{t}\int_{\Rd} \vert \Lsm u \vert^2\d x\d r.
\]
Therefore, collecting all the computations we obtain
\begin{align}
\label{eq:sec_mom}
\int_{\Rd}\frac{\vert x\vert^2}{2}u(x,t) \d x  + \frac{d-2s}{2}
\int_{0}^{t}\int_{\Rd} \vert \Lsm u \vert^2\d x\d r&\le
 \int_{\Rd}\frac{\vert x\vert^2}{2}u_0(x)\d x\nonumber\\
&+ d \int_{0}^{t}\int_{\Rd} \vert \Lsm u \vert^2\d x\d r,
\end{align}
namely \eqref{eq:second_moment}.
\end{proof}

\noindent {\sl \ref{ss:mo_ene}.2. Energies and dissipation.}

\begin{prop}
\label{prop:entropy_dissipation_v}
Let $u$ be a weak solution given by the existence Theorem \ref{th:existence} that verifies also \eqref{eq:extra_assumption}.
Then, setting
\begin{equation}
\label{eq:E_rescaled}
E(u(t)):=\frac{(1+t)^{1-2\beta}}{2}\int_{\Rd}\vert\Lsm u(x,t)\vert^2\d x+
\frac{\beta (1+t)^{-2\beta}}{2}\int_{\Rd}\vert x\vert^2 u(x,t)\d x,
\end{equation}
where $\beta:=\frac{1}{d + 2(1+s)}$,
there holds
\begin{equation}
\label{eq:entropy_diss}
E(u(t)) - E(u_0) +\int_{0}^t\int_{\Rd}(1+r)^{1-2\beta}\vert G\vert^2 \d x\d r \le 0,
\end{equation}
where the vector field $G\in L^2(0,T;L^2(\Rd))$ is related to
the vector field $\xi$ in \eqref{eq:energy_est} by
\[
u^{1/2}G = u^{1/2}\xi + \beta(1+t)^{-1}u x,
\]
namely
\begin{equation}
\label{eq:defG}
\nabla\Big((p+(1+t)^{-1}\frac{\beta}{2}\vert x\vert^2)u\Big) -
 \Big(p + (1+t)^{-1}\frac{\beta}{2}\vert x\vert^2\Big)\nabla u= u^{1/2}G\,\,\,\hbox{ a.e. in }\Rd\times (0,+\infty).
\end{equation}
\end{prop}

\begin{proof}[Proof of Proposition \ref{prop:entropy_dissipation_v}]
The proof of this Proposition requires to control the second moments of $u$. We recall that Theorem \ref{th:existence} shows that weak solutions
satisfy the energy estimate
\begin{equation}
\label{eq:energy_est_2}
\frac{1}{2}\int_{\Rd}\vert \Lsm u(t)\vert^2 \d x +
\int_{0}^t\int_{\Rd}\xi^2 \d x \d r
 \le
\frac{1}{2}\int_{\Rd}\vert \Lsm u_0\vert^2 \d x\,\,\hbox{ for a.a. } t \in (0,+\infty),
\end{equation}
where the vector field ${\bf\xi}\in L^2(0,+\infty;L^2(\Rd))$
 satisfies
\begin{equation}
\label{eq:def_xi_2}
  \nabla(u p) - p\nabla u = u^{1/2}\xi \,\,\,\,\,\hbox{ almost everywhere in }\Rd\times (0,+\infty).
\end{equation}

We note that \eqref{eq:second_moment} and \eqref{eq:energy_est_2} imply that (actually, both estimates can be shown to hold for almost any $\tau\le t$), respectively,
\[
\frac{\d}{\d t} \frac{\beta}{2}\int_{\Rd}\vert x\vert^2 u(x,t)\d x \le \int_{\Rd}\vert \Lsm u\vert ^s \d x,
\]
and
\[
\frac{\d}{\d t} \frac{1}{2} \int_{\Rd}\vert \Lsm u\vert ^s \d x \le - \int_{\Rd}\vert \xi\vert^2 \d x,
\]
in the sense of distributions on $(0,+\infty)$.

We compute $\frac{\d }{\d t} E(t)$.
We have
\begin{align*}
\frac{\d }{\d t}\Big( \frac{(1+t)^{1-2\beta}}{2}\int_{\Rd}\vert\Lsm u(x,t)\vert^2\d x\Big)
&\le \frac{1-2\beta}{2} (1+t)^{-2\beta}\int_{\Rd}\vert\Lsm u(x,t)\vert^2\d x\nonumber
\\
&-(1+t)^{1-2\beta}\int_{\Rd}\vert \xi\vert^2\d x,
\end{align*}
and
\begin{align*}
\frac{\d}{\d t}\Big(\frac{\beta (1+t)^{-2\beta}}{2}\int_{\Rd}\vert x\vert^2 u(x,t)\d x\Big)
& \le -\beta^2(1+t)^{-2\beta-1}\int_{\Rd}\vert x\vert^2 u\d x \nonumber\\
&+ (1+t)^{-2\beta}\beta\frac{d+2s}{2}\int_{\Rd}\vert \Lsm u\vert^2\d x.
\end{align*}
Therefore,
\begin{align}
\label{eq:new_energy1}
\frac{\d }{\d t} E(t) &\le (1+t)^{-2\beta}\Big(\frac{1-2\beta}{2}+\beta\frac{d+2s}{2}\Big)\int_{\Rd}\vert \Lsm u\vert^2\d x\nonumber\\
&-(1+t)^{1-2\beta}\int_{\Rd}\vert \xi\vert^2\d x - \beta^2(1+t)^{-2\beta-1}\int_{\Rd}\vert x\vert^2 u\d x\nonumber\\
& = \beta(d+2s)(1+t)^{-2\beta}\int_{\Rd}\vert \Lsm u\vert^2\d x\nonumber\\
&-(1+t)^{1-2\beta}\int_{\Rd}\vert \xi\vert^2\d x - \beta^2(1+t)^{-2\beta-1}\int_{\Rd}\vert x\vert^2 u\d x,
\end{align}
where we used that
\[
\frac{1-2\beta}{2\beta} = \frac{d+2s}{2}
\]
We concentrate on the last two terms.
We have that
\begin{eqnarray*}
(1+t)^{1-2\beta}\int_{\Rd}\vert \xi\vert^2\d x + \beta^2(1+t)^{-2\beta-1}\int_{\Rd}\vert x\vert^2 u\d x\\
=(1+t)^{1-2\beta}\Big(\int_{\Rd}\vert \xi\vert^2\d x + \beta^2(1+t)^{-2}\int_{\Rd}\vert x\vert^2 u\d x\Big)\\
=(1+t)^{1-2\beta}\int_{\Rd}\vert \xi + \beta(1+t)^{-1}x u^{1/2}\vert^2 \d x -
2\beta(1+t)^{-2\beta}\int_{\Rd}u^{1/2}\xi\cdot x \d x.
\end{eqnarray*}
Now, the definition of $\xi$ gives that $u^{1/2}\xi = \nabla (pu) - p\nabla u$. Moreover, since $\xi$ in $L^2(\Rd)$ for a.a. $t$ and $u^{1/2}x \in L^2(\Rd)$
for a.a. $t$, there holds that
$u^{1/2}x\cdot \xi = (\nabla (pu) - p\nabla u)\cdot x\in L^1(\Rd)$. Therefore,  thanks to Lemma \ref{lem:fourier}
 we have
\begin{align*}
-
2\beta(1+t)^{-2\beta}\int_{\Rd}u^{1/2}\xi\cdot x \d x
& = -2\beta (1+t)^{-2\beta}\int_{\Rd}(\nabla (pu) - p\nabla u)\cdot x\d x\\
& =
\beta(d+2s)(1+t)^{-2\beta}\int_{\Rd}\vert \Lsm u\vert \d x.
\end{align*}
Therefore \eqref{eq:new_energy1} becomes
\begin{align}
\label{eq:new_energy}
\frac{\d }{\d t} E(t) &\le -(1+t)^{1-2\beta}\int_{\Rd}\vert \xi + \beta(1+t)^{-1}x u^{1/2}\vert^2 \d x,
\end{align}
that is the thesis.
\end{proof}

Now we prove the validity of the key equality \eqref{eq:fourier_nice}.
\begin{lemma}
\label{lem:fourier}
Let $u\in H^{1+s}(\Rd)$ satisfying \eqref{eq:extra_assumption}.
 Then
\begin{equation}
\label{eq:fourier1}
\frac{d-2s}{2}\int_{\Rd}\vert \Lsm u\vert^2 \d x
= -\int_{\Rd}p( x\cdot\nabla u)\d x
\end{equation}
\end{lemma}

\begin{proof}[Proof of Lemma \ref{lem:fourier}]
We let
$u$ as in \eqref{eq:extra_assumption}.
Integration by parts gives, for any $v\in \mathcal{S}(\Rd)$ (hence in $\mathcal{S}'(\Rd)$), that
\[
\widehat{x_j \frac{\partial v}{\partial x_j}}=\int_{\Rd}e^{-i\xi\cdot x}x_j \frac{\partial v}{\partial x_j} \d x =
 i \xi_j \widehat{x_j v} - \hat{v} = -\xi_j \frac{\partial \hat{v}}{\partial \xi_j}-\hat{u}.
\]
Therefore,
\begin{align}
\label{eq:Fourier1}
\widehat{(x,\nabla v)} = \sum_{j=1}^d \widehat{x_j \frac{\partial v}{\partial x_j}}
= \sum_{j=1}^d\Big(-\xi_j \frac{\partial \hat{v}}{\partial \xi_j}-\hat{v}\Big)=-\div(\xi \hat{v}).
\end{align}
The Plancherel identify furnishes
\[
\int_{\Rd}\vert \Lsm u \vert^2 \d x = \int_{\Rd}p u \d x = \int_{\Rd}\hat p\hat u \d \xi.
\]
Therefore, since $\hat p = \vert \xi\vert^{2s} \hat u$ and
$\vert \xi\vert^{2s}=\frac{1}{d+2s}\div(\xi\vert \xi\vert^{2s})$,
we have
\[
\int_{\Rd}\vert \Lsm u\vert^2 \d x =\frac{1}{d+2s}
 \int_{\Rd}\div(\xi \vert \xi\vert^{2s})\hat u^2\d \xi =
 -\frac{2}{d+2s}\int_{\Rd}\vert \xi\vert^{2s}\hat u(\xi\cdot\nabla \hat u)\d \xi.
\]
Now, $(\xi,\nabla \hat u) = \div(\xi\hat u) - d \hat u= -\widehat{(x,\nabla u)} - d\hat{u}$. Therefore
\begin{align*}
\int_{\Rd}\vert \Lsm u\vert^2 \d x &=-\frac{2}{d+2s}\int_{\Rd}\vert \xi\vert^{2s}\hat u(\xi\cdot\nabla \hat u)\d \xi\\
&= \frac{2}{d+2s}\int_{\Rd}p (x,\nabla u) \d x + \frac{2d}{d+2s}\int_{\Rd}\vert \Lsm u\vert^2 \d x,
\end{align*}
namely
\[
\frac{d-2s}{2}\int_{\Rd}\vert \Lsm u\vert^2 \d x = -\int_{\Rd}p (x,\nabla u) \d x.
\]
An important consequence of \eqref{eq:fourier1} is that, if we know that
\[
(\nabla (up) - p\nabla u)\cdot x \in L^1(\Rd),
\]
there holds that
\begin{equation}
\label{eq:fourier2}
 \frac{1}{2}\int_{\Rd}\vert \Lsm u\vert^2 \d x= -\frac{1}{d+2s}\int_{\Rd}\big((\nabla (pu) - p\nabla u\big)\cdot x\d x.
\end{equation}
In fact,
\begin{align*}
-\int_{\Rd}\big((\nabla (p u) - p\nabla u\big)\cdot x\d x
&=d\int_{\Rd} p u \d x + \int_{\Rd}p (x\cdot\nabla u) \d x\\
& =d\int_{\Rd}\vert \Lsm u\vert^2 \d x -  \frac{d-2s}{2}\int_{\Rd}\vert \Lsm u\vert^2 \d x\\
& =\frac{d+2s}{2}\int_{\Rd}\vert \Lsm u\vert^2 \d x.
\end{align*}

We conclude by noting that, interestingly, by mimicking the proof above we can prove the
following
\begin{lemma}
\label{lemma:key_fract_pm}
Let $u\in H^{s}(\Rd)$ satisfying \eqref{eq:extra_assumption}. Then, denoting
with $p = (-\Delta)^{-s} u$ there holds
\begin{equation}
\label{eq:key_fract_pm}
\frac{d-2s}{2}\int_{\Rd}\vert (-\Delta)^{-s/2} u\vert^2 \d x
= -\int_{\Rd}u( x\cdot\nabla p)\d x
\end{equation}
\end{lemma}
Therefore, at least formally, we have that smooth solutions of the
Fractional Porous medium equation (see \cite{CV2011})
\[
\begin{cases}
\partial u = \div(u\nabla p),\,\,\,\hbox{ in } \Rd\times (0,+\infty),\\
\Ls p = u,\,\,\,\hbox{ in } \Rd\times (0,+\infty),
\end{cases}
\]
satisfy
\begin{equation}
\label{eq:sec_mom_fract_pm}
\frac{d}{dt}\frac{1}{2}\int_{\Rd}\vert x\vert^2 u\d x = \frac{d-2s}{2}\int_{\Rd}\vert (\Delta)^{-s/2} u\vert^2 \d x
\end{equation}
\end{proof}

\noindent {\bf Asymptotic behaviour}.
Now we come to the long time analysis.
As we saw in Section \ref{sec:self_similar}, given a smooth solution of \eqref{eq:prob_intro1} the rescaled solution $v$ according to \eqref{eq:source_type_form}  solves the Fokker-Planck type equation
\begin{equation}
\label{eq:FP}
\begin{cases}
\partial_\tau v -\div_{y}\Big(v\nabla_y \big(w + \frac{\beta}{2}\vert y\vert^2\big)\Big)= 0,\\
w = \Ls v.
\end{cases}
\end{equation}
On the other hand, given $v$ a solution of \eqref{eq:FP}
the function $u$ defined by
\begin{equation}
\label{eq:vtou}
v(y,\tau) = e^{-\alpha\tau} u(y e^{\beta \tau}, e^{\tau}-1),
\end{equation}
with $\alpha$ and $\beta$ as in \eqref{eq:alpha/beta},
is a solution of \eqref{eq:prob_intro1}.

It is easy to show that given a weak solution of \eqref{eq:prob_intro1} the rescaled solution $v$ according to \eqref{eq:source_type_form} is a weak solution of the Fokker-Planck equation \eqref{eq:FP} in the following sense
\begin{defn}
\label{def:weak_v}{\sl
Given $v_0 \in L^1_{loc}(\Rd)$ and nonnegative,
we say that $v$ is a weak solution of
 \eqref{eq:prob_intro1} if
 \begin{enumerate}
 \item $v\ge 0$ a.e. on $\Rd\times (0,+\infty),$
 \item $v\in L^\infty(0,+\infty; H^{s}(\Rd))\cap L^2(0,+\infty;H^{1+s}(\Rd)),$
 \item $\Ls v
 \in L^2(0,+\infty;H^{1-s}(\Rd)),$
 \item The following relation holds for any test
function $\varphi\in C^{\infty}_{c}(\Rd\times [0,+\infty))$
\begin{align}
\label{eq:weak_formulation_FP}
-\int_{0}^{+\infty}\int_{\Rd} v\partial_t \vf \d y \d \tau
& -\int_{0}^{+\infty}\int_{\Rd}\tilde w v \Delta \vf \,\d y \d \tau \nonumber\\
&-\int_{0}^{+\infty}\int_{\Rd} \tilde w\nabla  v \cdot\nabla \vf \,\d y \d \tau =  \int_{\Rd} v_0 \vf(y,0) \d y\nonumber\\
 w &= \Ls v,\,\,\tilde w = w +\frac{\beta}{2}\vert y\vert^2\,\hbox{ a.e. in } \Rd\times(0,+\infty).
\end{align}
\end{enumerate}
}

\end{defn}

The weak solutions of \eqref{eq:FP} that are obtained from rescaling
the weak solutions of the thin film equation obtained in Theorem \ref{th:existence} enjoy similar estimates.
We have the following

%
\begin{prop}
\label{prop:est_v}
Let $u_0:\Rd\to \R$ be a measurable function such that
\begin{align}
\label{eq:mass_longu0}
&\int_{\Rd}u_0(x)\d x = M,\\
&\mathscr{F}(u_0) = \int_{\Rd}u_0 \log u_0 \d x <+\infty,\label{eq:entropy_longu0}\\
&\mathscr{E}(u_0):=\int_{\Rd}\vert \Lsm u_0\vert^2 \d y+ \int_{\Rd}\vert y\vert^2 u_0\d y <+\infty.\label{eq:energy_longu0}
\end{align}
Let $u$ be a weak solution given by Theorem \ref{th:existence} with $u_0:\Rd\to \R$ as initial condition.
We let $v$ be the corresponding weak solution of \eqref{eq:FP}
obtained from rescaling $u$ as in \eqref{eq:source_type_form}.
Then,
\begin{enumerate}
\item {\sl Mass Conservation}
\begin{equation}
\label{eq:cons_mass_v}
\int_{\Rd}v(y,\tau)\d y = \int_{\Rd}v(y,0) \d y = \int_{\Rd}u_0(x)\d x\,\,\,\,\,\hbox{ for a.a. } \tau \in (0, +\infty),
\end{equation}
\item {\sl Entropy estimate.}
\begin{equation}
\label{eq:entropy_estv}
\mathscr{F}(v(\tau)) + \int_{0}^\tau\int_{\Rd}\vert \Lsm(\nabla_y v)\vert^2\d y \d r
\le \mathscr{F}(u_0) + \alpha M\tau, \,\,\,\forall \tau\ge 0,
\end{equation}
where, as in \eqref{eq:entropy_est},
\[
\mathscr{F}(v):=\int_{\Rd}v\log v \d y,
\]
\item {\sl Energy Estimate.}\,\,If $u$ verifies \eqref{eq:extra_assumption},
\begin{equation}
\label{eq:energy_estv}
\mathscr{E}(v(\tau)) +\int_{0}^{+\infty}\int_{\Rd}\vert H\vert^2
\d r \d y \le \mathscr{E}(v(0)),\,\,\,\forall \tau\ge 0,
\end{equation}
where the vector field $H\in L^2(0,+\infty;L^2(\Rd))$ is given by
\begin{equation}
\label{eq:defH}
\nabla\Big((w+\frac{\beta}{2}\vert y\vert^2)v\Big) -
 \Big(w + \frac{\beta}{2}\vert y\vert^2\Big)\nabla v= v^{1/2}H\,\,\,\hbox{ a.e. in }\Rd\times (0,+\infty).
\end{equation}
\end{enumerate}
\end{prop}

\begin{proof}
The conservation of mass, the entropy and energy estimates follow by rescaling using
\eqref{eq:source_type_form}. More precisely,
to obtain \eqref{eq:cons_mass_v} and \eqref{eq:entropy_estv} we simply rescale the analogous estimates \eqref{eq:mass_cons} and \eqref{eq:entropy_est} for $u$. To obtain estimate \eqref{eq:energy_estv}, we just rescale \eqref{eq:entropy_diss}.
\end{proof}

The first step in the long time analysis is the following Theorem in which we prove that the set of the cluster points for large times
of the weak solutions to \eqref{eq:FP} is not empty and its elements are indeed stationary solutions.
More precisely, we set
\begin{equation}
\label{eq:omega_lim}
\omega(v):=\left\{v_{\infty}\in H^s(\Rd)\cap L^1(\Rd, (1+\vert y\vert^2)\mathscr{L}^d): \exists \tau_n\nearrow +\infty \hbox{ with } v(\tau_n)\xrightarrow{n\to +\infty}v_{\infty} \hbox{ in } L^1(\Rd)\right\}
\end{equation}
(here $\mathscr{L}^d$ denotes the $d$-dimensional Lebesgue measure) and we prove the following Theorem
\begin{theor}
\label{th:long_time1}
Let us take an initial condition  $u_0$ satisfying \eqref{eq:mass_longu0}-\eqref{eq:energy_longu0} and a weak solution $u$ given by Theorem \ref{th:existence} with $u_0:\Rd\to \R$ as initial condition and satisfying \eqref{eq:extra_assumption}.
Then, denoting with
$v$ the weak (rescaled) solution of \eqref{eq:source_type_form} in the sense of Definition \ref{def:weak_v},
given a sequence of times $\left\{\tau_n\right\}$ such that $\tau_n\nearrow +\infty$, there exists a not relabelled subsequence ({that we still denote with $\tau_n$)} and a function
\begin{equation}
\label{eq:vinfty}
v_\infty \in H^{1+s}(\Rd)\cap L^1(\Rd, (1+\vert y\vert^2)\mathscr{L}^d)
\end{equation}
such that
\begin{equation}
\label{eq:conv_long_time}
v(\tau_n) \xrightarrow{n\to +\infty} v_\infty\,\,\,\,\hbox{ strongly in } {L^1(\Rd)},
\end{equation}
namely $v_\infty\in\omega(v)$.
Moreover,
any $v_\infty\in \omega(v)$ is a weak stationary solution of \eqref{eq:FP} that is,
\begin{equation}
\label{eq:weak_stationary}
\begin{cases}
\nabla\Big((w_\infty+\frac{\beta}{2}\vert y\vert^2)v_\infty\Big) -
 \Big(w_\infty + \frac{\beta}{2}\vert y\vert^2\Big)\nabla v_\infty = 0,\,\,\,\,\,\,\,\hbox{ in }\Rd\\
\Ls v_\infty = w_\infty\,\,\,\hbox{ in }\Rd.
 \end{cases}
\end{equation}
\end{theor}

\begin{proof}
Let $\left\{\tau_n\right\}$ be fixed in such a way that $\tau_n\nearrow +\infty$. Let $v$ be a weak solution given by Proposition
\ref{prop:est_v}. Then,
\begin{align}
\label{eq:bound_massa}
\int_{\Rd}v(\tau_n)\d y &= M,\,\,\,\,\,\,\forall n\in \mathbb{N},\\
 \sup_n\left\{\int_{\Rd}\vert \Lsm v(\tau_n)\vert^2 \d y + \int_{\Rd}\vert y\vert^2 v(\tau_n)\d y\right\} &
 \le \int_{\Rd}\vert \Lsm u_0\vert^2 \d y+ \int_{\Rd}\vert y\vert^2 u_0\d y <+\infty.\label{eq:bound_energia}
\end{align}
Therefore, the fractional version of the Nash's Inequality (or using interpolation),
implies that
\[
\| v(\tau_n)\|_{L^2(\Rd)}\le c(d,s), \,\,\,\forall\tau\ge 0.
\]
and therefore, using in particular the uniform bound (see \eqref{eq:bound_energia}) on the second moment,
$\left\{v(\tau_n)\right\}$ is relatively
compact in $L^2(\Rd)$ and in $L^1(\Rd)$ thanks to Rellich-Kondrachov, Dunford-Pettis and Vitali Theorems.
We let $v_\infty$ denote the limit
of $v(t_n)$. We have
\[
\int_{\Rd}v_\infty \d y = \int_{\Rd}u_0\d x,
\]
and
by semi-continuity, thanks to \eqref{eq:bound_energia},
\[
\mathscr{E}(v_\infty) {:=} \int_{\Rd}\vert \Lsm v_\infty\vert^2 \d y+ \int_{\Rd}\vert y\vert^2 v_{\infty}\d y<+\infty,
\]
namely \eqref{eq:vinfty} that says that
\[
\omega(v)\neq \emptyset.
\]
 It remains to show that the limit $v_\infty$ is indeed a stationary solution.
To this end, we standardly define $v_n(\cdot):=v(\cdot + \tau_n)$. For any $n\in \mathbb{N}$, $v_n$
is a weak solution in the sense of Definition \ref{def:weak_v}
with initial condition $v_n(0)=v(\tau_n)$.
Therefore $v_n$ satisfies both the estimates \eqref{eq:entropy_estv} and
 \eqref{eq:energy_estv}. The second one gives
\begin{align}
\label{eq:1_n}
\frac{1}{2}\Big(\int_{\Rd}\vert \Lsm v_n(\tau)\vert^2 \d y+ \int_{\Rd}\vert y\vert^2 v_n(\tau)\d y \Big) +
\int_{0}^\tau\int_{\Rd}\vert H_n\vert^2 \d y
\nonumber\\
\le \frac{1}{2}\Big(\int_{\Rd}\vert \Lsm v_n(0)\vert^2 \d y+ \int_{\Rd}\vert y\vert^2 v_n(0)\d y \Big), \,\,\,\,\forall \tau<+\infty
\end{align}
where
\[
\tilde{w}_n(\cdot):= \tilde w(\cdot +\tau_n), \,\,\,H_n(\cdot):=H(\cdot +\tau_n).
\]
Since \eqref{eq:energy_estv} is uniform with respect to $\tau$, we can bound the right hand side above with
\[
\frac{1}{2}\Big(\int_{\Rd}\vert \Lsm u_0\vert^2 \d y+ \int_{\Rd}\vert y\vert^2 u_0\d y \Big)
\]
and conclude that, uniformly with respect to $n$,
\begin{equation}
\label{eq:est1_n}
\|v_n\|_{L^\infty(0,+\infty; H^s(\Rd))} + \|v_n\|_{L^\infty(0,+\infty; L^1(\Rd,(1+\vert y\vert^2)\mathscr{L}^d))} \le C.
\end{equation}
Thus, using the Sobolev inequality, we conclude that
\begin{equation}
\label{eq:est11_n}
\|v_n\|_{L^\infty(0,+\infty; L^{p_s}(\Rd))} + \|v_n\|_{L^\infty(0,+\infty; L^1(\Rd,(1+\vert y\vert^2)\mathscr{L}^d))} \le C, \,\,\,p_s:=\frac{2d}{d-2s}.
\end{equation}
Consequently, using Lemma \ref{lemma:entropy} with $f(\cdot)=v_n(\cdot,\tau)$ for $\tau\ge 0$, we get
\begin{equation}
\label{eq:entropy_n}
\Big\vert\int_{\Rd}v_n(y, \tau) \log v_n(y,\tau)\d y \Big\vert\le C, \,\,\,\forall \tau \ge 0,
\end{equation}
where the constant $C$ depends only on the dimension $d$ and on $s$.
The entropy estimate \eqref{eq:entropy_estv} for $v_n$ reads ($M:=\int_{\Rd}v_n(y) \d y$)
\[
\mathscr{F}(v_n(\tau)) + \int_{0}^\tau\int_{\Rd}\vert \Lsm(\nabla_y v)\vert^2\d y \d r
\le \mathscr{F}(v_n(0)) + \alpha M\tau
\]
Thus, for any fixed $T>0$, thanks to \eqref{eq:entropy_n} we get
\begin{equation}
\int_{0}^T \int_{\Rd}\vert \Lsm(\nabla v_n)\vert^2 \d y \le C(T,M).
\end{equation}
As a result, combining the above estimate with \eqref{eq:est1_n} we get
\begin{equation}
\label{eq:est3_n}
\|v_n\|_{L^2(0,T;H^{1+s}(\Rd))} + \|w_n\|_{L^2(0,T;H^{1-s}(\Rd))} \le C(T).
\end{equation}
The weak formulation \eqref{eq:weak_formulation_FP} can be rewritten (for $v_n$)
as
\[
\int_{\Rd\times (0,T)}v_n\partial_t \eta \d y\d \tau = \int_{\Rd\times (0,T)}\tilde{w}_n \hbox{div}(\eta\nabla v_n ) \d y \d \tau,
\]
for all $\eta\in C^1_{c}(\Rd\times (0,T))$.
Thus, the since $w_n$ is bounded in $L^2(0,T;H^{1-s}(\Rd))$ and $v_n$ is bounded in $L^2(0,T;H^{1+s}(\Rd))$ we get a bound
for $\partial_\tau v_n$ in some $L^2(0,T;W^{-1,r}(\Rd))$ for $r>1$. Therefore,
there exist $\bar{v}_\infty$ and $\bar{w}_\infty$ and a not relabelled subsequence, such that
\begin{align}
& v_n \xrightarrow{n\to +\infty} \bar{v}_\infty\,\,\hbox{ weakly star in } L^\infty(0,+\infty; H^s(\Rd))
\cap L^2(0,T;H^{1+s}(\Rd)),\label{eq:weak_starv}\\
& w_n \xrightarrow{n\to +\infty} \bar{w}_\infty\,\,\,\,\,\hbox{ weakly star in } L^\infty(0,+\infty; H^{-s}(\Rd))\cap L^2(0,T;H^{1-s}(\Rd)),
\label{eq:weak_starw}\\
& v_n \xrightarrow{n\to +\infty} \bar{v}_\infty\,\,\,\,\hbox{ weakly in } L^1(\Rd\times (0,T)), \,\,\,\forall T>0,\label{eq:L1weak}\\
&v_n \xrightarrow{n\to +\infty} \bar{v}_\infty\,\,\,\,\hbox{ strongly in } L^2(0,T;H^{1+s-\delta}_{loc}(\Rd)), \,\,\,\forall \delta>0, \,\,\forall T>0,\label{eq:strongvn}
\end{align}
where the weak $L^1$-convergence follows from Dunford-Pettis Theorem thanks to the estimate \eqref{eq:est11_n}.
Note that \eqref{eq:L1weak} and \eqref{eq:strongvn} imply
\begin{equation}
v_n \xrightarrow{n\to +\infty} \bar{v}_\infty\,\,\,\,\hbox{ strongly in } L^1(\Rd\times (0,T))\,\,\,\forall T>0,\label{eq:L1strong}.
\end{equation}
Clearly, $\bar{v}_\infty\ge 0$ almost every where in $\Rd\times (0,+\infty)$.
Now we proceed with the identification of $\bar{v}_\infty$ as a weak solution of \eqref{eq:FP} emanating from $v_\infty$.
First of all, $\bar{w}_\infty$ is identified as $\bar{w}_\infty = \Ls \bar{v}_\infty$,
at least in the sense of distributions. The very same convergence \eqref{eq:weak_starw} holds also for
$\tilde{w}_n$ and we have $\tilde{w}_\infty= \bar{w}_\infty + \frac{\beta}{2}\vert y\vert^2$.
Moreover, testing the weak $L^1$ convergence with $\phi\equiv 1$ in $\Rd\times (0,\tau)$
we get, for any $\tau>0$,
\[
M\tau =
\lim_{n\to \infty}\int_{0}^\tau\int_{\Rd}v_n(y,r)\d y\d r =
\int_{0}^\tau\int_{\Rd}\bar{v}_\infty(y,r)\d y\d r.
\]
Thus,
\begin{equation}
\label{eq:mass_cons_longtime}
\int_{\Rd}\bar{v}_\infty(y,\tau)\d y = M, \,\,\,\,\forall \tau>0,
\end{equation}
namely the mass conservation.
The convergences above are enough, as in the proof of Theorem \ref{th:existence} to pass to the limit
in the weak formulation \eqref{eq:weak_formulation_FP} and obtain that ${\bar{v}_\infty}$ is
indeed a weak solution of \eqref{eq:FP}.

Moreover, \eqref{eq:energy_estv} gives that, for any $T>0$,
\begin{equation}
\int_{0}^T\int_{\Rd}\vert H_n\vert^2 \d y\d \tau = \int_{\tau_n}^{T+\tau_n}\int_{\Rd}\vert H\vert^2 \d y\d \tau\xrightarrow{n\to +\infty}0.
\end{equation}
Therefore, thanks to the above proved weak and strong convergences we conclude that
\begin{equation}
\label{eq:zero_diss}
\int_{0}^{+\infty}\int_{\Rd}\vert H_\infty\vert^2
\d y \d r = 0,
\end{equation}
where the vector field $H_\infty\in L^2(0,+\infty;L^2(\Rd))$ is the
weak limit of $H_n$ and satisfies
\begin{equation*}
\label{eq:defHinf}
\bar{v}_{\infty}^{1/2}H_\infty=\nabla\Big((\bar{w}_\infty+\frac{\beta}{2}\vert y\vert^2)\bar{v}_\infty\Big) -
 \Big(\bar{w}_\infty + \frac{\beta}{2}\vert y\vert^2\Big)\nabla \bar{v}_\infty,
 \end{equation*}

Therefore \eqref{eq:zero_diss} gives that $H_\infty=0$ almost everywhere in $\Rd\times (0,+\infty)$.
Thus,
($\tilde{w}_\infty:=\bar{w}_\infty + \frac{\beta}{2}\vert y\vert^2$)
\begin{equation}
\label{eq:weak_vinf}
\int_{0}^{T}\int_{\Rd} \bar{v}_\infty\partial_t \vf \d y \d \tau = -\int_{0}^T \int_{\Rd}\Big(\nabla(\tilde{w}_\infty {\bar v}_\infty)-\tilde{w}_\infty\nabla \bar{v}_\infty\Big) \cdot \nabla \vf \d y\d r = 0,
\end{equation}
for any $\vf \in C^1_{c}(\Rd\times (0,T))$ and thus we have that $\bar{v}_\infty$ is constant in time, hence
$\bar{v}_\infty = v_\infty$ for all $\tau\ge 0$.
In particular, we conclude that $v_\infty$ satisfies \eqref{eq:weak_stationary}.
\end{proof}

We have the following
\begin{prop}
\label{eq:winf_cost}
Let $v$ be a weak solution of \eqref{eq:FP} constructed according to
Proposition \ref{prop:est_v} and let $v_\infty\in \omega(v)$. Then, in each connected component $\mathscr{C}_i$ of
\[
\mathscr{P}_\infty:=\left\{y\in \Rd: v_\infty(y)>0 \right\} = \bigcup_{i}C_i
\]
we have that there exists a constant $c_i$
 such that $w_\infty=\Ls v_\infty = c_i - \frac{\beta}{2}\vert y\vert^2$.
\end{prop}
\begin{proof}
We observe that the definition of $H_\infty$ implies that
\begin{equation}
\label{eq:equality}
\nabla(\tilde{w}_\infty v_\infty) -
 \tilde{w}_\infty\nabla v_\infty = 0 \,\,\,\,\hbox{ a.e. in } \Rd.
\end{equation}
Thus,
from  it follows that, in any connected component of $\mathscr{P}_\infty$,
we get that $\tilde{w}_\infty$ is constant.

In fact, for any $\delta>0$ let us consider the set
\[
\mathscr{P}_{\delta}:=\left\{x\in \Rd: v_\infty \ge \delta\right\}.
\]
Due to the Sobolev regularity of $v_\infty$ this set is quasi open (see \cite{ki-ma} for the definition).
Now, for any fixed $R>0$ and $x_0\in \Rd$, thanks to \eqref{eq:equality}, we have that
\[
\nabla(\tilde{w}_\infty v_\infty) \in L^q(B_{R}(x_0)),
\]
for some $q\ge 1$. Therefore, $\tilde{w}_{\infty} \in W^{1,q}(B_R(x_0)\cap P_\delta)$ (see \cite[Lemma 2.5]{ki-ma}) with
\[
\nabla \tilde{w}_{\infty}  = \frac{1}{v_{\infty}}\big( \nabla(\tilde{w}_\infty v_\infty) -
 \tilde{w}_\infty\nabla v_\infty\big),\,\,\,\hbox{ a.e. in }B_R(x_0)\cap P_\delta.
\]
Consequently, \eqref{eq:equality} implies that $\nabla  \tilde{w}_{\infty} = 0$
almost everywhere in  $B_R(x_0)\cap P_\delta$, for any $x_0\in \Rd$, for any $R>0$ and for any $\delta>0$
which implies that $\tilde{w}_{\infty}$ is constant on any connected component of $\mathscr{P}_\infty$.
\end{proof}
We can now state the main result of this Section.
\begin{theor}
\label{th:long_time}

Let us given a measurable function $u_0:\Rd\to \R$ such that

\begin{align}
\label{eq:u0long_hypo}
& \int_{\Rd}u_0\log u_0 \d x<+\infty,\\
& \int_{\Rd}\vert \Lsm u_0\vert^2 \d x+ \int_{\Rd}\vert x\vert^2 u_0\d x<+\infty,
\\
& \int_{\Rd}u_0 \d x = M,
\end{align}
and let $u$ be a weak solution of \eqref{eq:prob_intro1} given by Theorem \ref{th:existence} with initial datum $u_0\ge 0$ and satisfying \eqref{eq:extra_assumption}.
Let $v$ be the rescaled weak solution according to \eqref{eq:source_type_form} and to Proposition \ref{prop:est_v}.
Let us assume that for any $v_\infty\in \omega(v)$ the set $\mathscr{P}_\infty$ is connected.
Then, the following convergence holds
\begin{equation}
\label{eq:long_timeV}
v(\cdot, \tau) \xrightarrow{\tau \to +\infty}v_C(\cdot), \,\,\,\,\hbox{ in } L^1(\Rd),
\end{equation}
where $v_C$ is the solution of the obstacle problem provided by Theorem \ref{th:stationary}
 with the constant $C$ determined by the mass law \eqref{eq:mass_law}.
 Therefore (recall \eqref{eq:source_type_form},  \eqref{eq:self_similar_th} and \eqref{eq:vtou}), in terms $u$ we have the following large times convergence
  \begin{equation}
 \label{eq:large_times_u}
 u(\cdot, t) - u_C(\cdot,t) \xrightarrow{t\to +\infty} 0\,\,\,\,\,\hbox{ in } L^1(\Rd),
 \end{equation}
 namely the convergence of the corresponding weak solution of \eqref{eq:prob_intro1} to the self-similar solution $u_C$.

\end{theor}
\begin{proof}
The Proposition above shows that $v_\infty$ solves
\begin{equation}
\label{eq:stat_infinity}
\begin{cases}
\Ls v_\infty = \sum_{i\in \mathscr{N}}c_i\chi_i(x) - \frac{\beta}{2}\vert y\vert^2\,\,\,\hbox{ in }\mathscr{P}_\infty,\\
v_\infty = 0 \,\,\,\,\hbox{ in }\Rd\setminus \mathscr{P}_\infty,
\end{cases}
\end{equation}
($\chi_i$ is the characteristic function of $\mathscr{C}_i$).

Thus, the assumption of connectedness of $\mathscr{P}_\infty$,
gives that
$v_\infty$ can represented as in \eqref{eq:general_sol} with
$K=0$ due to the regularity. In this way, thanks to Lemma \ref{lemma:vd=vo}, we conclude that $v_\infty$ is the obstacle solution $v_C$ with the constant $C$ given according to the mass law \eqref{eq:mass_law}.
Therefore, up to a subsequence (see \eqref{eq:L1strong}),
\[
v(\cdot,\tau_n)\xrightarrow{n\to \infty} v_\infty(\cdot) = v_C(\cdot) \,\,\,\hbox{ in } L^1(\Rd).
\]
Then, the uniqueness of the solution of the obstacle problem, gives that the convergence above holds not only for a subsequence of times and therefore \eqref{eq:long_timeV} is satisfied. The convergence for $u$ follows from the definition of $v$ in \eqref{eq:source_type_form}.
\end{proof}

We conclude this Section with some comments.
Both Theorems \ref{th:long_time1} and \ref{th:long_time}
work for those weak solutions that satisfy the extra assumption \eqref{eq:extra_assumption}. As we observed, the proof
that all weak solutions satisfy \eqref{eq:extra_assumption} constitutes
a challenging open problem.

We observe that we can actually dispense with this assumption
at the price of introducing an extra approximation at the level
of the Fokker Planck equation.
This approximation is analogous to the approximation we used
for proving existence in Theorem \ref{th:existence} and produces
weak solutions of the Fokker Planck equation that satisfy the estimates \eqref{eq:entropy_estv} and \eqref{eq:energy_estv}.
This would correspond in studying
as a first the long time behavior of these weak solutions of the
Fokker Planck equation and then in obtaining as a second step the convergence to the self-similar solution of the weak solutions
of the thin film equation by rescaling.
Unfortunately, this procedure has a potential oddity since, due to nonuniqueness, the weak solution of the thin film equation that we obtain from rescaling back the weak solution of the Fokker Planck is not necessarily one of the weak solutions we construct in Theorem \ref{th:existence}.

This explains why we chose to include \eqref{eq:extra_assumption} as a suitable
extra regularity assumption.



\section{\bf An extension to higher order problems with similar  structure}
\label{sec.higher4}

An important feature of equation \eqref{eq:prob_intro1} is its conservation law structure, that we may display as
\[
\begin{cases}
\partial_t u = \div( u {\bf F}),\\
{\bf F} = \nabla p,\\
p = \Ls u.
\end{cases}
\]
The particular equation depends on the closing relationship between $u$ and $p$. For instance, to obtain equation \eqref{eq:prob_intro2}
one considers $-s$ instead of $s$ and to obtain the (local) porous medium equation one considers $s=0$). More in general, a interesting open problem is the analysis of
\begin{equation}
\label{eq:closing}
p = \mathscr{K}[u],
\end{equation}
where $\mathscr{K}$ can be a local or nonlocal operator, even of higher order than $2$.
The case $\mathscr{K}=(-\Delta)^m$ with $m>1$ has been first studied to our knowledge in \cite{bernis-fried90} and then in \cite{Flittonking, gala_king1, gala_king2, ChavesGal} and others. Work is mostly done in one space dimension.

\medskip

\noindent {\bf Self-similar higher order solutions.} \noindent As an advance to the theory of higher order equations, we contribute here  the calculation the regular self-similar solution for the equations of the form
\begin{equation}
\label{eq:ho}
\begin{cases}
\partial_t u - \div (u \nabla p) = 0, \,\,\hbox{ in } \Rd \times (0,T)\\
p = A_{2+2s}u, \,\,\hbox{ in }\Rd \times (0,T)\\
u(x,0) = u_0(x), \,\,\,\hbox{ in } \Rd,
\end{cases}
\end{equation}
where $s\in (0,1)$, $ A_{2+2s}=(-\Delta)^{1+s}=\Ls \circ (-\Delta )$ and $\Ls$,   is the fractional Laplacian as in previous section. The dimension $d\ge 1$. The order of the equation is then $4+2s\in (4,6)$.
The theory of existence for general equations of the type \eqref{eq:ho} has not been done but it should follow
the steps of Section \ref{sec:existence}.

\smallskip

\noindent (i) If again we look for solutions of the self-similar form
\begin{equation}
\label{eq:source_type_form_ho}
u(x,t) = \frac{1}{(1+t)^\alpha}v\Big(\frac{x}{(1+t)^\beta}, \log (1+t)\Big),
\end{equation}
where the function $v:\mathbb{R}^d\times \mathbb{R}\to \mathbb{R}$
is to be appropriately determined and the parameters $\alpha$ and $\beta$ are now given by
\begin{equation}
\label{eq:alpha/beta_ho}
\alpha = \frac{d}{d  + 2(2+ s)}, \,\,\,\,\beta = \frac{1}{d+ 2(2+ s)}.
\end{equation}
due to the constraints that we will find below. We set
$$
y:=\frac{x}{(1+t)^\beta}, \quad \tau:=\log (1+t), \quad w = \Ls v.
$$
Assuming all the regularity needed to justify the computations, and after calculations that have no novelty, we arrive following \sl nonlinear and nonlocal Fokker-Planck type equation:\rm
\begin{equation}
\label{eq:fokker-cahn_ho}
\begin{cases}
\partial_\tau v -\div_{y}\Big(v\big( \nabla_y w + \beta y\big)\Big)= 0, \,\,\,\hbox{ in } \Rd\times (0,+\infty)\\
w = \Ls (-\Delta v) \,\,\,\,\,\hbox{ in } \Rd\times (0,+\infty).
\end{cases}
\end{equation}

\noindent (ii) We make a reduction in the set of possible solutions and concentrate on those stationary solutions of \eqref{eq:fokker-cahn_ho} such that
\begin{equation}
\label{eq:stationary2_ho}
\begin{cases}
v\nabla_y\big( w + \frac{\beta}{2} \vert y\vert^2\big) = 0\,\,\,\,\hbox{ in } \,\,\Rd,\\
w =\Ls (-\Delta v)\,\,\,\hbox{ in } \Rd.
\end{cases}
\end{equation}
As in the parallel study made in \cite{CVlarge} for negative values of $s$, this reduction must be justified by the later analysis of the long-time behavior and the asymptotic convergence to a self-similar profile.

Obtaining a solution is then reduced to the famous complementarity rule: either $v=0$ or
$\nabla_y\big( w + \frac{\beta}{2} \vert y\vert^2\big) = 0$. Furthermore, and the second condition will be simplified to finding a ball where $w =C - \frac{\beta}{2} \vert y\vert^2$ for some $C\in \mathbb{R}$.

\subsection{\bf Explicit form.} The solution of the stationary self-similar problem can be explicitly computed as follows.
\begin{theor}
\label{th:explicit}
Consider the function
\begin{equation}
\label{eq:trial}
V(y):=(A - a\vert y\vert^2)^{2+s}_{+}\,\,\,\,\,\hbox{ for } y\in \mathbb{R}^d,
\end{equation}
which is positive in the ball $R=(A/a)^2$. There exists $a=a(d,s)$ such
 $V$ solves the problem
\begin{equation}
\label{eq:stationary_ho}
\begin{cases}
\Ls (-\Delta) V= C-\frac{\beta}2\vert y\vert^2\,\,\,\hbox{ in } B_R\\
V = 0 \,\,\,\,\hbox{ in } \Rd\setminus B_R\,.
\end{cases}
\end{equation}
The precise value of $a$ is computed  below, \eqref{eq:a}. $A>0$ is a free constant and
$C = c(s,d)A$ with $c(s,d)$ computed at the end of the proof.
\end{theor}

\noindent {\sc Proof.}  (i)
Let us first calculate the Laplacian of $V$ in $B_R$. Writing $V=f(Z)$ with $f(Z)=Z^{2+s}$, $Z=A - ar^2$, and $r=|y|$, we use the formula
$$
\Delta f(Z)=f'(Z)\,\Delta Z + f''(Z)|\nabla Z|^2,
$$
to get at all points where $V>0$
$$
-\Delta V=(2+s)Z^{1+s}(2ad)-(2+s)(1+s)Z^s(4a^2r^2)=
$$
$$
=(2+s)Z^{s}\{(2ad)(A-ar^2)-4a^2(1+s)r^2\}.
$$
The coefficient of $-r^2$ in the last parenthesis is $ 2a^2d + 4a^2(1+s)$ that we write $ \mu a^2$ with $\mu=2d+4(1+s))$. Therefore, we get
$$
-\Delta V=(2+s)Z^{s}\mu a(A-ar^2)+ (2d-\mu)(2+s)Aa\,Z^{s},
$$
and finally we get
$$
-\Delta V=F_1(y)+F_2(y)\,,
$$
with
\begin{equation}\label{1}
F_1(y)=(2+s)\mu a\,(A-ar^2)^{1+s}, \quad F_2(y)= -4(1+s)(2+s)A a\,(A-ar^2)^{s}.
\end{equation}
The splitting into these two functions will be very convenient. Note that $-\Delta V=0$ outside of the support, and $-\Delta V$ is a smooth function globally, since there is no delta function (measure) at the support boundary because the normal derivative of $V$ at $r=R$ is zero.

\medskip

\noindent (ii) Next we prepare some very precise calculations. It is convenient to define
\begin{equation}
\label{eq:dyda1_ho}
v_{1}(y):=\frac{1}{\kappa_{s,d}}(1 - \vert y\vert^2)^{1+s}_{+}\,\,\,\,\,\hbox{ for } y\in \mathbb{R}^d,
\end{equation}
where $\kappa_{s,d}:=2^{2s}\Gamma(s+2)\Gamma(s+ \frac{d}{2})\Gamma(\frac{d}{2})^{-1}$, with $\Gamma(\cdot)$ being the Euler $\Gamma$-function. This function is supported in the ball of radius 1.  According to Dyda \cite{dyda}  we have
\begin{equation}
\label{eq:dyda2_ho}
\begin{cases}
\Ls v_{1} (y)= 1 - \gamma_{s,d}\vert y\vert^2=:f(y)\,\,\,\,\hbox{ in }B_{1}(0)\\
v_1(y) = 0 \,\,\,\,\hbox{ in } \Rd\setminus B_1(0),
\end{cases}
\end{equation}
with $\gamma_{s,d}:=1+ \frac{2s}{d}>1$.
Notice that $\Ls v_{1} (y)$ is positive for small $|y|$ but negative for $|y|\sim 1$. Next, we need to change the constant $\gamma_{s,d}$ into $\beta/2$ in the last formula, and this is done by rescaling as follows: we introduce a parameter $\lambda>0$ and set
 \begin{equation}
 \label{eq:explicit_obstacle_ho}
 v_\lambda (y):= \frac{1}{\lambda^{2s}}v_1(\lambda y) = \frac{1}{\lambda^{2s}\kappa_{s,d}}(1- \lambda^2 \vert y\vert^2)_{+}^{1+s}, \end{equation}
 For every $y\in B_{1/\lambda}(0)$ we have
\[
\Ls v_\lambda (y)=\frac{1}{\lambda^{2s}}(\Ls v_1)(\lambda y)\lambda^{2s} = (\Ls v_1)(\lambda \,y).
\]
Fixing the value $\lambda:=\sqrt{\beta/(2\gamma_{s,d})}\,,$  we get the result:
$$
\Ls v_\lambda(y)= 1- \frac{\beta}{2}\vert y\vert^2.
$$
{
 We need to introduce another rescaling \ $ v_K(y):=K^{1+s}v(K^{-1/2}y)$,
that satisfies
 \begin{equation}
 \label{eq:wc_ho}
 \Ls v_K(y):= K \Ls v(K^{-1/2}y).
 \end{equation}
 Combining both scalings we can define
 \begin{equation}\label{2}
 \widehat v(y)=v_{\lambda,K}(y)=\frac{1}{\lambda^{2s}\kappa_{s,d}}
 (K-{\lambda^2}\vert y\vert^2)_{+}^{1+s}\,.
\end{equation}
This function has the property that
\begin{equation}
 \Ls \widehat v= K-\frac{\beta}{2}\vert y\vert^2
\end{equation}
 in the positivity set of $ \widehat v$, the ball of radius $(2K/\beta)^{1/2}=(2K(d+4+2s))^{1/2}$.

\medskip

\noindent (iii) In this step we proceed towards the solution $V$ by adjusting $F_1(y)$ in  \eqref{1} to formula \eqref{2}. Forgetting for the moment about $A$ and $K$ which are the free constants, we determine the main constant $a>0$ by the relationship
$$
(2+s)\mu a^{2+s}=
 \frac{\lambda^{2}}{\kappa_{s,d}}.
$$
This produces a formula for  $a=a(d,s)$:
\begin{equation}
\label{eq:a}
a^{-(2+s)} = 
2(2+s)(d+4+2s)\gamma_{s,d}\mu_{s,d}\kappa_{s,d}\,.
\end{equation}

\noindent {\sc Verification.} For $s=0$ and $d=1$ we are dealing with the Thin Film equation in one dimension, and then $a^2 = 1/120$ that is consistent with the explicit solution found by Bernis-Peletier-Williams in \cite{BPW}.

\medskip

Moreover, the free constants  $K>0$ and $A>0$ are related by
\begin{equation}
\label{eq:condizioneAa}
((2+s)\mu a )^{1/(1+s)} A = (\lambda^{2s}\kappa_{s,d})^{-1/(1+s)} K\,,
\end{equation}
so that
\begin{equation}
A = K \frac{a}{\lambda^2}
\label{eq:A_ho}.
\end{equation}
Now, with this choice of $A$ and $a$ we get
\begin{equation}
\label{eq:sol1}
\Ls F_1(y) = K-\frac{\beta}{2}\vert y\vert^2,\,\,\,\,\hbox{ in }B_R(0)
\end{equation}
\medskip
\noindent (iv)
Next, we tackle $F_2(y)$. The choice \eqref{eq:a} and \eqref{eq:A_ho} of $a$ and $A$ fix
the value of $\Ls F_2$ on $B_R(0)$.
More precisely, we observe that
\[
F_2(y) = -4(1+s)(2+s)A a\,(A-ar^2)^{s} = -4(1+s)(2+s)A^{1+s} a \,v_G\Big(\frac{a^{1/2}y}{A^{1/2}}\Big),
\]
where $v_G$ is the Getoor solution \eqref{eq:getoor_sol}.
Thus,
\begin{equation}
\label{eq:getoor_ho}
\Ls F_2(y) = -4(1+s)(2+s)Aa^{1+s}\kappa_{s,d}=:K_2,\,\,\,\,\,\hbox{ in } B_R(0).
\end{equation}
As a result, we have
\begin{equation}
\label{eq:F1+F_2}
\begin{cases}
(\Ls (-\Delta) V)(y) = \Ls (F_1 +F_2)(y) = K-K_2 -\frac{\beta}{2}\vert y\vert^2\,\,\,\,\,\hbox{ in } B_R(0)\\
V = 0 \,\,\,\,\hbox{ in } \Rd\setminus B_R(0).
\end{cases}
\end{equation}

Therefore,
$$
\Ls(-\Delta)V=K-K_2-\frac{\beta}2r^2.
$$
The proof is done with $C=K-K_2=c(s,d)A$ with

\begin{equation}
 \label{eq:C_ho}
 \quad c(s,d)=\frac{\beta}{2a\gamma_{s,d}}\Big(1 - \frac{4(1+s)}{\kappa_{s,d}(2d + 4(1+s))} \Big).
 \end{equation}
  \qed

\section{\bf Open problems}
\label{sec:OP}
In this final Section we collect some open problems that we find worth considering.

\medskip

\noindent
{
 $\bullet$ \sc Gradient Flow. \rm An interesting open problem, motivated by the decaying of the energy $\mathscr{E}$ defined by \eqref{energy}, is whether the evolution \eqref{eq:prob_intro1} is a Wasserstein gradient flow for $\mathscr{E}$.
This is the case indeed for the related model \eqref{eq:prob_intro2}, which was shown in \cite{LMS} to be a Wasserstein gradient flow
for the $\frac{1}{2}\|\cdot\|_{H^{-s}(\Rd)}^2$-norm, and for the Thin Film equation ($s=1$) (see \cite{McMS}).
More in general, an interesting problem is to understand whether \eqref{eq:prob_intro1} with a concave mobility $m(u) = u^\gamma$
is indeed a gradient flow for $\mathscr{E}$ with respect to a weighted Wasserstein distance of the type of \cite{DNS}.
}

\medskip

  \noindent $\bullet$
  \sc Compactly supported solutions. \rm
  In Section \ref{sec:self_similar} we have constructed self-similar solutions with compact support. These are weak solutions
   (for $t\ge 0$) according to Definition \ref{def:weak} that originate from a Dirac Mass located in $t=-1$. For the moment these are the only solutions
 we are able to construct that are compactly supported.
 It is clearly interesting to understand whether compactly supported initial conditions generate compactly supported solutions.
 This is indeed a quite complicated question since the equation is formally of order $2+2s$ and thus we do not have comparison arguments
 at our disposal.
 If the solutions are compactly supported a free boundary appears and must be studied. This is a difficult open problem that was been thoroughly investigated for the PME,  see for instance \cite{Vaz2007} and the recent work \cite{KKV}, where extensive references are given. The topic has also attracted lot of attention for the Thin Film equation, see without any claim of completeness \cite{bernis96}, \cite{grunn} and \cite{GK}. A general reference for the mathematics of free boundaries is \cite{CaffSal05}.

\medskip

\noindent $\bullet$ \sc Self-similar solutions\rm.
As discussed in the paper, the self-similar solutions of equation
\eqref{eq:prob_intro1} are given by the Barenblatt profiles
\begin{equation}
 \label{eq:ss_open}
 v(y)=(C_1-C_2  |y|^2)_+^{1+s}\,,
 \end{equation}
which coincides with the Barenblatt profile for the standard Porous Medium equation with $m = \frac{s+2}{s+1}$. We find this coincidence quite interesting
and worth to be further analysed. Note on this regard that
when $s=1$ (hence $m=3/2$), namely thin films with linear mobility, this observation
has been already successfully used in \cite{CT_thinfilm} for the long time behaviour of the thin film equation.

\noindent $\bullet$ \sc Uniqueness\rm.
So far we have proved existence of a weak solution. A natural question is to understand whether some uniqueness holds, at least in $1$-D. This is an interesting problem already for $s=1$, namely the Thin Film equation (see \cite{MMS} and references therein).
In particular, it would be interesting to see if there is uniqueness when there is a Dirac Mass as initial data. This uniqueness result, if true,
would be important in the convergence to self-similar solutions as in the so called ``three steps method" for the classical porous medium equation, see \cite[Chapter 18 ]{Vaz2007}.)

\medskip

\noindent $\bullet$  \sc Multi-Bump stationary states. \rm
Theorem \ref{th:long_time} requires the hypothesis of connectedness
of the omega-limit set of a weak solution $v$ of \eqref{eq:FP}.
An interesting problem is clearly to understand if this assumption is
really necessary. In particular, it would be interesting to
exclude the presence stationary states with disconnected support
or to provide examples of multi-bump asymptotic limits. This problem
is clearly related to the construction of self-similar solutions for which the positivity set is disconnected.

\medskip

\noindent $\bullet$ \sc Singular limits. \rm
As we have already pointed out, Equation \eqref{eq:prob_intro1} interpolates between
the Porous Medium equation ($s=0$) and the Thin Film equation ($s=1$). A natural question is to investigate these singular limits for the constructed solutions, and rigorously relate these three equations.

\medskip

\noindent $\bullet$ \sc Power Law mobility function. \rm
The analysis of \eqref{eq:prob_intro1} has been restricted to a linear mobility function.
The case of a power law mobility function of the type $m(u)=u^n$ is, to the best of our knowledge,
open in dimension $d\ge 2$ (see \cite{ranat} for the one dimensional case
in a bounded interval with Neumann boundary conditions)
and deserves to be studied.
In particular, it would be interesting to understand the relation (if any) between the order of fractional differentiation
$s$, the exponent $n$ and the dimension $d$ for the existence of nontrivial compactly supported self-similar solutions.
When $s=1$ and $d=1$ a quite complete picture is given in \cite{BPW},  while for $s=0$ (PME) the situation is understood  in all dimensions \cite{Vaz2007}. For the porous medium equation with nonlocal pressure, case $-1<s<0$, this is studied in \cite{dTSV, dTSV1, dTSV2}, and for $s=-1$ in \cite{SerfVaz}.

\medskip

\noindent $\bullet$ \sc Relation to Cahn Hilliard Equation. \rm
The analysis of \eqref{eq:prob_intro1} suggests that it would be
interesting to consider the following evolution
\begin{equation}
\begin{cases}
\partial_t u = \div (m(u)\nabla p)\,\,\,\hbox{ in } \,\,\,\Rd\times (0,+\infty)\\
 w = \Ls u + f(u),\,\,\,\hbox{ in }\,\,\,\Rd\times (0,+\infty),
\end{cases}
\end{equation}
where $f:\mathbb{R}\to \mathbb{R}$.
The equation above can be considered as fractional version of the Cahn-Hilliard equation with nonconstant mobility, and to the best
of our knowledge, it has been studied only in \cite{ABG} for bounded
domains with Neumann boundary conditions and with $m$ independent of $u$.
The Cahn-Hilliard equation plays a central role in material science
and its analysis (see, among the others, \cite{barret-blowey}, \cite{barret-garcke}, \cite{elliot-garcke}) suggests that there should be a precise relation
between the mobility function and the nonlinearity $f$.

\medskip

\noindent  $\bullet$ \sc Integrated equation. \rm A transformation that has been very useful in the study of similar  equations of order from 0 to 2 in one space dimension is the integration transformation
$$
v(x,t)=\int_{-\infty}^x u(y,t)\,dy.
$$
This allows to pass from equation \eqref{eq:prob_intro1}, i.e., $u_u=(u(p(u)_x)_x$,
to  \ $v_t=v_xp(v_x)_x$, which for $p(u)=\Ls u$ gives
$$
v_t= v_x\,(\Ls v)_{xx}\,.
$$
Our results can be transferred to the latter equation but otherwise no more seems to be known. Let us point out that the study of that equation for $-1<s<0$  has been very fruitful thanks to the maximum principle that allows for the theory of viscosity solutions and comparison results, cf. \cite{BKM2010, dTSV}.

\medskip

\noindent  $\bullet$
\sc Numerics. \rm  The theoretical results would  greatly benefit from the development of efficient numerical methods for \eqref{eq:prob_intro1}, in particular in dimension one, in view of the potential application to cracks dynamics (see \cite{imbert_mellet11}
and \cite{imbert_mellet15} and the references therein.)

\section*{Acknowledgements}
{ \small AS is a member of the GNAMPA (Gruppo Nazionale per l'Analisi Matematica, la Probabilit\`a e le loro Applicazioni)
group of INdAM and acknowledges the partial support of the MIUR-PRIN Grant 2010A2TFX2 "Calculus of Variations".   \
JLV was funded by MTM2017-85449-P (Spain) and benefitted from an Honorary Professorship at Univ. Complutense de Madrid. He acknowledges the hospitality of the Dipartimento di Matematica ``F. Casorati" of the
University of Pavia where part of this work has been done.
The authors would like to thank Giuseppe Toscani for interesting discussions on the subject of the paper.
The authors would like to acknowledge the referees for the careful reading of the paper and for their positive criticism.}

\medskip


\medskip

\noindent 2010 \textit{Mathematics Subject Classification.} 35R11 (35G25, 35K46, 35K65 35C06).

\smallskip

\noindent \textit{Keywords and phrases.} Fractional operators, thin film equations, self-similar solutions, obstacle problem.
\normalcolor

\smallskip

\noindent {\sc Addresses:}

\noindent Antonio Segatti. Dipartimento di Matematica ``F. Casorati", Universit\`a di Pavia,\\
Via Ferrata 1, 27100 Pavia, Italy. \\ e-mail address:~\texttt{antonio.segatti@unipv.it}

\smallskip

\noindent Juan Luis V\'azquez. Depto. de Matem\'{a}ticas, Universidad
Aut\'{o}noma de Madrid, \\Campus de Cantoblanco, 28049 Madrid, Spain.  \\ e-mail address:~\texttt{juanluis.vazquez@uam.es}


\addcontentsline{toc}{section}{~~~References}
\begin{thebibliography}{10}

\bibitem{ABG}
{\sc H. Abels, S. Bosia, M. Grasselli.}
{\sl Cahn-Hilliard Equation with Nonlocal Singular Free Energies,}
Ann. Mat. Pura Applicata, {\bf 194}, (2014), 1071--1106.



\bibitem{ASSS}
{\sc G. Akagi, G. Schimperna, A. Segatti, L.V. Spinolo.}
{\sl Quantitative estimates on localized finite differences for the fractional Poisson problem, and applications to regularity and spectral stability}
Commun. Math. Sci, {\bf 18}, (2018), 913--961.



\bibitem{almglieb} {\sc F. J. Almgren, E. H. Lieb.}
{\sl Symmetric decreasing rearrangement is sometimes continuous},
   J. Amer. Math. Soc., {\bf 2}, (1989), 683--773.


\bibitem{ambro-serf}
{\sc L. Ambrosio,  S. Serfaty.}
      {\sl A gradient flow approach to an evolution problem arising in
              superconductivity}, Comm. Pure Appl. Math., {\bf 61},
      (2008), 1495--1539.

\bibitem{AR}
{\sc D. G. Aronson.}
{\sl The porous medium equation},
In Nonlinear diffusion problems (Montecatini Terme, 1985), volume 1224 of Lecture Notes in Math., (1986), pages 1--46, Springer,
Berlin.


\bibitem{Bar52}
{\sc G. I. Barenblatt.} {\sl On some unsteady motions of a liquid and gas in a
porous medium}, Akad. Nauk SSSR. Prikl. Mat. Meh., {\bf 16} (1952), pp. 67--78.

\bibitem{Bar62}
{\sc G. I. Barenblatt.} {\sl The mathematical theory of equilibrium cracks formed in brittle fracture,}
 Adv. Appl. Mech., {\bf 7} (1962), pp. 55--129.


\bibitem{barret-blowey}
{\sc J. W. Barret, J.F. Blowey.}
{\sl Finite element approximation of the Cahn-Hilliard equation
with concentration dependent mobility},
Math. Comp., {\bf 68}, (1999), 487--517.


\bibitem{barret-garcke}
{\sc J. W. Barret, J.F. Blowey, H. Garcke.}
{\sl Finite element approximation of the Cahn-Hilliard equation
with degenerate mobility},
SIAM J. Numer. Anal., {\bf 37}, (1999), 286--318.



%

\bibitem{BG}
{\sc J. Becker, G. Gr\"un.}
{\sl The thin film equation: recent advances and some new perspective}
J. Phys.: Condens. Matter {\bf 17}, (2005), 291--307.




\bibitem{BBD} {\sc E. Beretta, M. Bertsch, R. Dal Passo.}
{\sl Nonnegative solutions of a fourth-order nonlinear degenerate parabolic equation},
Arch. Rational Mech. Anal. {\bf 129} (1995), no. 2, 175--200.


\bibitem{bernis96}
{\sc F. Bernis.}
{\sl Finite speed of propagation and continuity of the interface for thin viscous flows},
Adv. Differential Equations, {\bf 1}, (1996), 337--368.


\bibitem{bernis-fried90}
{\sc F. Bernis,  A. Friedman.}
{\sl Higher order nonlinear degenerate parabolic equations,}
J. Differential Equations, {\bf 83} 1990, 179-206.

\bibitem{BPW}
{\sc F. Bernis, L.A. Peletier, S.M. Williams.}
{\sl Source type solutions of a fourth order nonlinear degenerate
parabolic equation,}
Nonlinear Anal., {\bf 18} (1992), 217--234.


\bibitem{BDGG}
{\sc M. Bertsch, R. Dal Passo, H. Garke, G. Gr\"un.}
{\sl The thin viscous flow equation in higher space dimensions},
Adv. Differential Equations {\bf 3} (1998), 417--440.

\bibitem{BKM2010} {\sc P. Biler, G. Karch,  R. Monneau.}
 {\sl Nonlinear diffusion of dislocation density and self-similar solutions,}
 Comm. Math. Phys. {\bf 294} (2010), 145--168.

\bibitem{BIK}{\sc P. Biler, G. Karch, C. Imbert.}
{\sl The Nonlocal Porous Medium Equation: Barenblatt Profiles and other Weak Solutions,}
Arch. Ration. Mech. Anal., {\bf 215}, (2015), 497--529.

\bibitem{CaffSal05} {\sc L. A. Caffarelli, L. Salsa.}
{``A geometric approach to free boundary problems''.}
Graduate Studies in Mathematics, {\bf 68}. American Mathematical Society, Providence, RI, 2005.

\bibitem{CS2007} {\sc L. A. Caffarelli, L. Silvestre.}
{\sl An extension problem related to the fractional Laplacian},
 Comm. Partial Differential Equations,
 \textbf{32}, (2007), 1245--1260.



\bibitem{CSS}{\sc L. A. Caffarelli, S. Salsa, L. Silvestre.}
{\sl Regularity estimates for the solution and the free boundary of
              the obstacle problem for the fractional {L}aplacian},
              Invent. Math. {\bf 171} (2008), 425--461

\bibitem{CV2011} {\sc L. A. Caffarelli,  J. L. V\'azquez.}
{\sl Nonlinear porous medium flow with fractional potential pressure}.
Arch. Ration. Mech. Anal. {\bf 202} (2011), 537--565.


\bibitem{CVlarge}{\sc L. A. Caffarelli, J. L. V\'azquez.}
{\sl Asymptotic behaviour of a porous medium equation with fractional diffusion},
Discrete Contin. Dyn. Syst. {\bf 29} (2011), 1393--1404.


\bibitem{Ca_To2000}{\sc J. A. Carrillo, G. Toscani.}
{\sl Asymptotic {$L^1$}-decay of solutions of the porous medium
              equation to self-similarity},
   Indiana Univ. Math. J.,
    {49}, (2000), 113--142.



\bibitem{CT_thinfilm}{\sc J. A. Carrillo, G. Toscani.},
{\sl Long-time asymptotics for strong solutions of the thin film
              equation},
  Comm. Math. Phys., {\bf 225}, (2002), 551--571.



\bibitem{ChavesGal} {\sc M. Chaves, V. A. Galaktionov.}
{\sl On source-type solutions and the Cauchy problem for a doubly degenerate sixth-order thin film equation. I. Local oscillatory properties.}
Nonlinear Anal. {\bf 72} (2010), no. 11, 4030--4048.


\bibitem{DMc}
{\sc J. Denzler, R. J. McCann.}
{\sl Nonlinear diffusion from a delocalized source: affine self-
similarity, time reversal \& nonradial focusing geometries},
Ann. Inst. H. Poincar\'e Anal. Non Lin\'eaire. {\bf 25}, (2008),
865--888.

  \bibitem{dipe_lions}
{\sc R. J. DiPerna, P. L. Lions.}
     {\sl On the {F}okker-{P}lanck-{B}oltzmann equation},
   Comm. Math. Phys.,  {\bf 120}, (1988) 1--23.



\bibitem{disoaval}
{\sc S. Dipierro, N. Soave, E. Valdinoci.}
   {\sl On fractional elliptic equations in {L}ipschitz sets and
              epigraphs: regularity, monotonicity and rigidity results},
              Math. Ann., {\bf 369}, (2017), 1283--1326.



\bibitem{DNS}
{\sc J. Dolbeault, B. Nazaret, G. Savar\'{e}.}
     {\sl A new class of transport distances between measures},
   Calc. Var. Partial Differential Equations, {\bf 34}, (2009), 193--231.


\bibitem{dyda}{\sc B. Dyda.}
{\sl Fractional calculus for power functions and eigenvalues of the
              fractional {L}aplacian}, Fract. Calc. Appl. Anal. {\bf 12}, (2012),
              536--555.



\bibitem{elliot-garcke}
{\sc C. M. Elliot, H. Garcke.}
{\sl On the Cahn-Hilliard equation with degenerate mobility,}
SIAM J. Math. Anal., {\bf 27}, (1996), 404--423.


{\bibitem{BF97}{\sc R. Ferreira, F. Bernis.}
{\sl Source-type solutions to thin-film equations in higher
              dimensions,}
   European J. Appl. Math., {\bf 8}, (1997), 507--524.
  }


\bibitem{gala_king1}
{\sc J. D. Evans, V. A. Galaktionov, J. R. King.}
{\sl Unstable sixth-order thin film equation. {I}. {B}low-up
              similarity solutions},
              Nonlinearity, {\bf 20}, (2007), 1799--1841.

\bibitem{gala_king2}
{\sc J. D. Evans, V. A. Galaktionov, J. R. King.}
{\sl Unstable sixth-order thin film equation. {II}. {G}lobal
              similarity patterns},
              Nonlinearity, {\bf 20}, (2007), 1843--1881.


\bibitem{Flittonking} {\sc J. C. Flitton, J. R. King.}
{\sl Moving-boundary and fixed-domain problems for a sixth-order
   thin-film equation},
European J. Appl. Math. {\bf 15}, (2004), 6, 713--754-



\bibitem{Getoor}
{\sc R. K. Getoor.}
   {\sl First passage times for symmetric stable processes in space},
   Trans. Amer. Math. Soc., {\bf 101}, (1961), 75--90.


\bibitem{GK}
{\sc L. Giacomelli,  H. Kn\"{u}pfer.}
{\sl A Free Boundary Problem of Fourth Order: Classical Solutions in Weighted Hšlder Spaces},
Comm. Partial Differential Equations
{\bf 35}, (2010), 2059--2091.


\bibitem{GKO}
{\sc L. Giacomelli,  H. Kn\"{u}pfer, F. Otto.}
{\sl Smooth zero-contact-angle solutions to a thin-film equation
              around the steady state},
              J. Differential Equations, {\bf 245}, (2008), 1554--1506.



\bibitem{GL1} {\sc G. Giacomin, J. L. Lebowitz.} {\sl Phase segregation dynamics in particle systems with long range interaction I. Macroscopic limits}, J. Stat. Phys. {\bf 87} (1997),  37--61.

\bibitem{GL2} {\sc G. Giacomin, J. L. Lebowitz.} {\sl Phase segregation dynamics in particle systems with long range interaction II. Interface motion}, SIAM J. Appl. Math. {\bf 58} (1998) 1707–29


\bibitem{gilb_trud}
{\sc D. Gilbarg, N. S. Trudinger.}
{``Elliptic partial differential equations of second order''},
Classics in Mathematics, Reprint of the 1998 edition,
Springer-Verlag, Berlin (2001).


\bibitem{grafakos}{\sc L. Grafakos, G. Teschl.}
{\sl On {F}ourier transforms of radial functions and distributions},
   J. Fourier Anal. Appl.,  {\bf 19}, (2013), 1069--179.


\bibitem{grunn}
{\sc G. Gr\"unn.}
{\sl Droplet Spreading Under Weak Slippage-Existence for the Cauchy
Problem,}
Comm. Partial Differential Equations, {\bf 29}, (2005), 1697--1744.


\bibitem{Head}  {\sc A. K. Head.} {\sl Dislocation group dynamics II. Similarity solutions of the continuum approximation.} Phil. Mag. {\bf 26} (1972), 65--72.


\bibitem{imbert_mellet11}{\sc C. Imbert, A. Mellet.}
{\sl Existence of solutions for a higher order non-local equation
appearing in crack dynamics},
Nonlinearity, {\bf 24}, (2011), 3487--3514.

{
\bibitem{imbert_mellet15}
{\sc C. Imbert, A. Mellet.}
{\sl Self-similar solutions for a fractional thin film equation
              governing hydraulic fracture},
 Comm. Math. Phys., {\bf 340}, (2015), 1187--1229.
}


\bibitem{fall}
{\sc M. M. Fall.}
{\sl Entire {$s$}-harmonic functions are affine},
Proc. Amer. Math. Soc.,  {\bf 144}, (2016), 2587--2592.


\bibitem{KKV}
{\sc  C. Kienzler, H. Koch, J. L. V\'azquez.}
{\sl Flatness implies smoothness for solutions of the porous medium equation}, Cal. Var. PDEs {\bf 57}, 1 (2018), 57:18. 

\bibitem{ki-ma}
{\sc T. Kilpel\"{a}inen, J. Mal\'{y}.}
     {\sl Supersolutions to degenerate elliptic equation on quasi open
              sets},  Comm. Partial Differential Equations, {\bf 17},
      (1992), 371--405.

\bibitem{knu}
{\sc H. Kn\"upfer.}
{\it Classical Solution for a Thin Film equation},
 PhD Thesis, Universit\"at Bonn, (2007).

\bibitem{land}{\sc N. S. Landkof.}
{``Foundations of modern potential theory''},
Springer-Verlag, Berlin (1972).



\bibitem{LMS}{\sc S. Lisini, E. Mainini, A. Segatti.}
{\sl A gradient flow approach to the porous medium equation with fractional pressure},
Arch. Ration. Mech. Anal. {\bf 227}, (2018), 567--606.
%
%
\bibitem{lieb_loss}
{\sc E. Lieb, M. Loss.}
{``Analysis''}
Graduate Studies in Mathematics 14,
American Mathematical Society, Providence (2001)

\bibitem{MMS}
{\sc M. Majdoub, N. Masmoudi, S. Tayachi.}
{\sl Uniqueness for the thin-film equation with a {D}irac mass as
              initial data},
              Proc. Amer. Math. Soc. {\bf 146}, (2018), 2623--2635.

\bibitem{McMS}
{\sc D. Matthes, R. J. McCann, G. Savar\'{e}.}
   {\sl A family of nonlinear fourth order equations of gradient flow
              type},
   Comm. Partial Differential Equations, {\bf 34}, (2009), 1352--1397.



\bibitem{My}
{\sc T. G. Myers.}
{\sl Thin films with high surface tension},
SIAM Rev. {\bf 40} (1998), no. 3, 441--462.




\bibitem{ros_serra}
{\sc X. Ros-Oton, J. Serra.}
    {\sl The {D}irichlet problem for the fractional {L}aplacian:
              regularity up to the boundary},
    J. Math. Pures Appl. (9), {\bf 101}, (2014), 275--302.





\bibitem{SerfVaz}{\sc S. Serfaty,  J.~L. V{\'a}zquez.} {\sl A mean field equation as limit of nonlinear diffusions with fractional Laplacian operators.} Calc. Var. Partial Differential Equations {\bf 49} (2014), no. 3-4, 1091--1120.

\bibitem{SilvestrePhD}{\sc L. Silvestre.}
{\it Regularity of the obstacle problem for a fractional power of the laplace operator},
PhD Thesis, University of Austin (2005).



\bibitem{SH}{\sc N. F. Smyth, J. M. Hill.}
{\sl Higher order nonlinear diffusion},
I.M.A. J. Appl. Math. {\bf 40}, (1988), 73--86.



\bibitem{Stampacchia65}{\sc G. Stampacchia.}
{\sl Le probl\`eme de {D}irichlet pour les \'equations elliptiques du
              second ordre \`a coefficients discontinus,}
              Ann. Inst. Fourier (Grenoble), {\bf 15}, (1965), 189--258.


\bibitem{dTSV}{\sc D. Stan, F. del Teso, J. L. V\'azquez.}
{\sl Finite and infinite speed of propagation for porous medium equations with nonlocal pressure}, J. Differential Equations {\bf 260} (2016), no. 2, 1154--1199.

\bibitem{dTSV1}{\sc D. Stan, F. del Teso, J. L. V\'azquez.}
{\sl Porous medium equation with nonlocal pressure,}
Current Research in Nonlinear Analysis, Springer Optim. Appl. {\bf 135} Springer, Cham, 2018, pp. 277--308.


\bibitem{dTSV2}{\sc D. Stan, F. del Teso, J. L. V\'azquez.}
{\sl Existence of weak solutions for  porous medium equations with nonlocal pressure, } Arch. Ration. Mech. Anal. {\bf 233}, (2019), 451--496.



\bibitem{stein}{\sc E. M. Stein.}
{``Singular integrals and differentiability properties of
              functions ''},
   Princeton Mathematical Series, No. 30 (1970).


\bibitem{ranat}{\sc R. Tarhini.}
{\sl Study of a family of higher order nonlocal degenerate parabolic equations: from the porous medium equation to the thin film equation}.
J. Differential Equations {\bf 259}, (2015), 5782--5812.

\bibitem{ranatphd}{\sc R. Tarhini.}
{\it Existence et r\'egularit\'e des solutions de deux \'equations paraboliques, d\'eg\'en\'er\'ees et non-locales}, PhD-Thesis, Universit\`e Paris-Est, (2017).

\bibitem{TO_private}{\sc G. Toscani.}
{\it Private Communication.}



%
\bibitem{Vaz2007} {\sc J. L. V\'azquez.}
{``The Porous Medium Equation. Mathematical Theory''},
 Oxford Mathematical Monographs. The Clarendon Press, Oxford University Press, Oxford  (2007).

%
\bibitem{VazAbel} {\sc J. L.  V{\'a}zquez.}
{\sl Nonlinear Diffusion with Fractional Laplacian Operators.}
 in ``Nonlinear partial differential equations: the Abel Symposium 2010'',
Holden, Helge  \& Karlsen, Kenneth H. eds., Springer, 2012. Pp. 271--298.

\bibitem{VazCIME} 	{\sc J. L.  V{\'a}zquez.} {\sl The mathematical theories of diffusion. Nonlinear and fractional diffusion}  in ''Nonlocal and Nonlinear Diffusions and Interactions: New Methods and Directions'', Springer Lecture Notes in Mathematics, 
    C.I.M.E. Foundation Subseries.


\bibitem{ZK1950} 	{\sc Y. B. Zeldovich, A. Kompaneets}. {\sl Towards a theory of heat conduction with thermal conductivity depending on the temperature}, Collection
of papers dedicated to 70th birthday of Academician A.F. Ioffe, Izd. Akad.
Nauk SSSR, Moscow, (1950), pp. 61--71.



\end{thebibliography}
\end{document}